\documentclass{article}

\title{Effective bounds for adelic Galois representations attached to elliptic curves over the rationals}
\author{Lorenzo Furio}
\date{}

\usepackage[usenames,dvipsnames]{color}
\usepackage{amsmath}%
\usepackage{amsfonts}%
\usepackage{amssymb}%
\usepackage{amsthm}%
\usepackage{graphicx}
\usepackage{multirow}
\usepackage[utf8]{inputenc}
\usepackage[english]{babel}
\usepackage{mathrsfs}
\usepackage{array}
\usepackage{rotating}
\usepackage{multirow}
\usepackage{faktor}
\usepackage{tikz}
\usepackage{tikz-cd}
\usepackage{bbm}
\usepackage[title,titletoc]{appendix}
\usepackage{quiver}
\usepackage{enumerate}

\usepackage{a4wide}

\usepackage{hyperref}

\hypersetup{
	colorlinks,
	citecolor=black,
	linkcolor=black,
	urlcolor=black
}

\newtheorem{theorem}{Theorem}[section]

\newtheorem{claim}[theorem]{Claim}

\newtheorem{conjecture}[theorem]{Conjecture}
\newtheorem{corollary}[theorem]{Corollary}

\newtheorem{lemma}[theorem]{Lemma}
\newtheorem{notation}[theorem]{Notation}

\newtheorem{proposition}[theorem]{Proposition}

\newtheorem{question}[theorem]{Question}

\theoremstyle{definition}
\newtheorem{definition}[theorem]{Definition}

\newtheoremstyle{named}{}{}{\itshape}{}{\bfseries}{.}{.5em}{#1 \thmnote{#3}}
\theoremstyle{named}

\theoremstyle{remark}
\newtheorem{remark}[theorem]{Remark}

\newcommand{\N}{\mathbb{N}}
\newcommand{\Z}{\mathbb{Z}}
\newcommand{\Q}{\mathbb{Q}}
\newcommand{\C}{\mathbb{C}}
\newcommand{\F}{\mathbb{F}}
\newcommand{\OK}{\mathcal{O}_K}

\newcommand{\Gal}{\operatorname{Gal}}

\newcommand{\Aut}{\operatorname{Aut}}
\newcommand{\GL}{\operatorname{GL}}
\newcommand{\SL}{\operatorname{SL}}
\newcommand{\GalQ}{\operatorname{Gal}\left(\faktor{\overline{\mathbb{Q}}}{\mathbb{Q}}\right)}
\newcommand{\GalK}{\operatorname{Gal}\left(\faktor{\overline{K}}{K}\right)}
\newcommand{\Fheight}{\operatorname{h}_\mathcal{F}}
\newcommand{\uhp}{\mathcal{H}}
\newcommand{\Qab}{\mathbb{Q}^{\operatorname{ab}}}

\newcommand{\Ind}{\operatorname{Ind}}
\newcommand{\Ent}{\operatorname{Ent}}

\numberwithin{equation}{section}

\begin{document}
	
	\maketitle
	
	\begin{abstract}
		Given an elliptic curve $E$ defined over $\Q$ without complex multiplication, we provide an explicit sharp bound on the index of the image of the adelic representation $\rho_E$. In particular, if $\Fheight(E)$ is the stable Faltings height of $E$, we show that $[\GL_2(\widehat{\Z}) : \operatorname{Im}\rho_E]$ is bounded above by $10^{21} (\Fheight(E)+40)^{4.42}$, and, for $\Fheight(E)$ tending to infinity, by $\Fheight(E)^{3+o(1)}$. We also classify the possible (conjecturally non-existent) images of the representations $\rho_{E,p^n}$ whenever $\operatorname{Im}\rho_{E,p}$ is contained in the normaliser of a non-split Cartan. This result improves previous work of Zywina and Lombardo.
	\end{abstract}
	
	\tableofcontents

	\section{Introduction}

	Let $K$ be a number field and let $E$ be an elliptic curve defined over $K$. It is known that for a positive integer $N$, the action of the absolute Galois group of $K$ on the $N$-torsion points of $E$ defines a representation
	\begin{equation*}
		\rho_{E,N}: \GalK \to \Aut(E[N]) \cong \GL_2\left(\faktor{\Z}{N\Z}\right).
	\end{equation*}
	If $p $ is a fixed prime, we can restrict to values of $N$ of the form $p^n$ and take the limit over $n$. We then obtain the representation
	\begin{equation*}
		\rho_{E,p^\infty}: \GalK \to \Aut(T_pE) \cong \GL_2(\Z_p),
	\end{equation*}
	where $T_pE = \varprojlim E[p^n]$ is the $p$-adic Tate module of $E$, and taking the product over all primes we define the adelic representation
	\begin{equation*}
		\rho_E = \prod_{p \text{ prime}} \rho_{E,p^\infty} : \GalK \to \prod_{p \text{ prime}} \Aut(T_pE) \cong \GL_2(\widehat{\Z}).
	\end{equation*}
	In 1972 Serre \cite{serre72} proved his celebrated open image theorem, stating that if the elliptic curve $E$ does not have (potential) complex multiplication, then the image of the representation $\rho_{E}$ ; equivalently, the image of $\rho_E$ has finite index in $\GL_2(\widehat{\Z})$. From now on we will say that $E$ does not have complex multiplication (CM) it does not have complex multiplication over $\overline{K}$.
	In the same paper, Serre asked a question, which can be proved to be equivalent to the following.
	
	\begin{question}\label{question: uniformity, adelic version}
		Let $K$ be a number field. Does there exist a constant $c$, depending only on $K$, such that for every non-CM elliptic curve $\faktor{E}{K}$ we have $[\GL_2(\widehat{\Z}) : \operatorname{Im}\rho_E] \le c$?
	\end{question}
	
	Although this problem has been widely studied, the question is still open, even in the case $K=\Q$. However, in this case, it is conjectured that it can be answered affirmatively (see \cite[Conjecture 1.5]{zywina22}).
	
	\begin{conjecture}\label{conj: adelic suq}
		For every elliptic curve $\faktor{E}{\Q}$, the index of the image of $\rho_{E}$ is at most $2736$.
	\end{conjecture}
	
	At present, we are not able to find a uniform bound on the adelic index of non-CM elliptic curves. However, one can try to find some bounds on this index depending on the curve $E$.
	Recently, Zywina \cite{zywina11} provided a bound on the adelic index in the case in which the elliptic curve $E$ is defined over $\Q$, polynomial in terms of the height $\operatorname{h}(j(E))$. Moreover, he also gave a bound in terms of the conductor of $E$.
	
	\begin{theorem}[Zywina]
		Let $E$ be a non-CM elliptic curve defined over $\Q$.
		\begin{enumerate}
			\item There are absolute constants $C$ and $\gamma$ such that
			$$[\GL_2(\widehat{\Z}) : \operatorname{Im}\rho_E] \le C(\max\{1, \operatorname{h}(j(E))\})^\gamma,$$
			where $\operatorname{h}(j(E))$ is the logarithmic Weil height of the $j$-invariant of $E$.
			\item Let $N$ be the product of the primes of bad reduction of $E$. There is an absolute constant $C$ such that
			$$[\GL_2(\widehat{\Z}) : \operatorname{Im}\rho_E] \le C\left(68N(1+\log\log N)^\frac{1}{2}\right)^{24\omega(N)},$$
			where $\omega(N)$ is the number of distinct prime divisors of $N$.
		\end{enumerate}
	\end{theorem}
	
	The bound in terms of the height of $j(E)$ relies on previous results of Masser and W\"ustholz on isogenies \cite{masser-wustholz-1, masser-wustholz-2}, while the bound in terms of the conductor builds on the work of Serre \cite{serre81} (under GRH) and Kraus \cite{kraus95} to bound the largest prime $p$ for which $\rho_{E,p}$ is not surjective. Later, Cojocaru \cite[Theorem 2]{cojocaru05} provided a similar statement for the product of non-surjective primes; however, as Zywina notes in \cite[Remark 3.4]{zywina11}, there seems to be an error in the second part of \cite[Theorem 2]{cojocaru05}. More recent effective results assuming GRH are given in \cite{maylewang24}.
	
	Both of Zywina's bounds are ineffective. Later, Lombardo \cite[Remark 1.1]{lombardo15} gave a bound for the adelic index for elliptic curves defined over a generic number field. This bound is effective and polynomial in terms of the Faltings height of the curve (which is approximately $\frac{1}{12}\operatorname{h}(j(E))$ as shown in Theorem \ref{thm:heights}).
	
	\begin{theorem}[Lombardo]
		Let $E$ be a non-CM elliptic curve defined over a number field $K$ and let $\Fheight(E)$ be the stable Faltings height of $E$. We have
		$$\left[\GL_2(\widehat{\Z}) : \rho_E\left(\Gal(\overline{K}/K)\right) \right] < \gamma_1 \cdot [K:\Q]^{\gamma_2} \cdot \max\left\{ 1, \Fheight(E), \log[K:\Q] \right\}^{\gamma_2},$$
		where $\gamma_1= \exp(1.9 \cdot 10^{10})$ and $\gamma_2=12395$.
	\end{theorem}
	
	The main aim of this article is to provide a much better bound on the adelic index when $E$ is defined over $\Q$. Moreover, we also give an improved and effective version of Zywina's bound in terms of the conductor. The following result combines Theorem \ref{thm:adelicbound} and Theorem \ref{thm:conductoradelicbound}.
	
	\begin{theorem}\label{thm:adelicboundwithj}
		Let $\faktor{E}{\Q}$ be an elliptic curve without CM and let $\Fheight(E)$ be its stable Faltings height, with the normalisation given in \cite{deligne85}.
		\begin{enumerate}
			\item If $\Fheight(E)$ is the stable Faltings height of $E$, we have $$[\GL_2(\widehat{\Z}) : \operatorname{Im}\rho_E] < 10^{21} (\Fheight(E) + 40)^{4.42}.$$
			\item We have $$[\GL_2(\widehat{\Z}) : \operatorname{Im}\rho_E] < \Fheight(E)^{3+O\left(\frac{1}{\log\log\Fheight(E)}\right)}$$ as $\Fheight(E)$ tends to $\infty$, where the constant is explicit.
			\item If $N$ is the product of the primes of bad reduction of $E$ and $\omega(N)$ is the number of distinct prime factors of $N$, we have
			$$[\GL_2(\widehat{\Z}) : \operatorname{Im}\rho_E] < 2488320 \left(51 N(1+\log\log N)^\frac{1}{2}\right)^{3 \omega(N)}.$$
		\end{enumerate}
	\end{theorem}
	
	The approach in the proof of Theorem \ref{thm:adelicboundwithj} is different from that of Lombardo. Indeed, it builds on Zywina's techniques, sharpening the study of $p$-adic groups to improve the bound, and exploiting new recent results on Galois images of small level by Rouse--Sutherland--Zureick-Brown. We will give more details on the proof later.
	
	\medskip
	
	A more general goal is to classify all the possible images of $\rho_E$ inside $\GL_2(\widehat{\Z})$. This is known as Mazur's `Program B' \cite{mazur77b}. In recent years, much progress has been made in the case of elliptic curves defined over $\Q$. Most of the results are given either in the `vertical' or `horizontal' direction, i.e. they either classify the possible images of the $p$-adic representations $\rho_{E,p^\infty}$, or study the entanglement phenomenon at composite level.
	
	Recently, building on the work of many authors (for example \cite{greenberg12, greenberg14, momose86, momoseshimura02, rousezb15, sutherlandzywina17}), Rouse, Sutherland, and Zureick-Brown \cite{rszb22} gave a detailed description of all the possible $p$-adic images for all primes $p$ whenever the image of $\rho_{E,p}$ is not contained in the normaliser of a non-split Cartan.
	As part of the proof of Theorem \ref{thm:adelicboundwithj}, we give a classification of possible $p$-adic images in the non-split Cartan case. We state the result below, as it is of independent interest and represents progress on Mazur’s Program B.
	
	Building on the work of Zywina \cite{zywina15}, we give a restricted list of subgroups $G < \GL_2(\Z_p)$ such that $\operatorname{Im}\rho_{E,p^\infty}$ is possibly equal to $G$ whenever $\operatorname{Im}\rho_{E,p}$ is contained in the normaliser of a non-split Cartan. Given an integer $\varepsilon$ which is not a square modulo $p$, denote with 
	\begin{equation*}
		C_{ns}^+(p^n):= \left\lbrace \begin{pmatrix} 1 & 0 \\ 0 & \pm 1 \end{pmatrix} \cdot \begin{pmatrix} a & \varepsilon b \\ b &  a \end{pmatrix} \,\middle|\, a,b \in \faktor{\Z}{p^n\Z}, \ (a,b) \not\equiv (0,0) \mod p \right\rbrace \subseteq \GL_2\left(\faktor{\Z}{p^n\Z}\right)
	\end{equation*}
	the normaliser of a non-split Cartan (see Section \ref{sec:grouptheory} for definitions). The following theorem is a simple form of Theorem \ref{thm:ellipticcartantower}. The reader should consult Theorem \ref{thm:ellipticcartantower} to see a classification covering the cases $p=3,5$.
	
	\begin{theorem}\label{thm:introellcartantower}
		Let $\faktor{E}{\Q}$ be an elliptic curve without complex multiplication. Let $p>5$ be a prime such that $\operatorname{Im}\rho_{E,p} \subseteq C_{ns}^+(p)$ up to conjugation, and let $n \ge 1$ be the smallest integer such that $\operatorname{Im}\rho_{E,p^\infty} \supseteq I + p^n M_{2 \times 2}(\Z_p)$. One of the following holds:
		\begin{itemize}
			\item The image of $\rho_{E,p^n}$ is equal to $C_{ns}^+(p^n)$ up to conjugation;
			\item $n=2$ and
			$$\operatorname{Im}\rho_{E,p^2} \cong C_{ns}^+(p) \ltimes \left\lbrace I + p\begin{pmatrix} a & \varepsilon b \\ -b & c \end{pmatrix} \right\rbrace,$$
			with the semidirect product defined by the conjugation action.
		\end{itemize}
	\end{theorem}
	
	On the other hand, many authors started studying the `horizontal' classification problem, i.e. the classification of entanglement fields. Serre \cite[Proposition 22]{serre72} proved that for every non-CM elliptic curve $E$ defined over $\Q$, the image of $\rho_E$ lies in an index-$2$ subgroup of $\GL_2(\widehat{\Z})$, even if the $p$-adic representation $\rho_{E,p^\infty}$ is surjective for every prime $p$. More recent results focus on the study of the intersection of $\Q(E[p])$ and $\Q(E[q])$, for two different primes $p,q$, usually small (see for example \cite{entanglement1, entanglement2, entanglement3, entanglement4, lozanorobledo23}).
	In Section \ref{subsec:entanglement}, we prove some general theorems to bound the degree of entanglement fields, especially in the case where one of the division fields has Galois group contained in the normaliser of a non-split Cartan. We then give a bound on the growth of the adelic index with respect to the product of the $p$-adic indices due to the entanglement phenomenon. In particular, in Remark \ref{rmk: adelic vs p-adic} we show that if $\faktor{E}{\Q}$ is a non-CM elliptic curve that does not satisfy the uniformity conjecture (Conjecture \ref{conj: adelic suq}), then the ratio of the product of the $p$-adic indices $\prod_{p \text{ prime}} [\GL_2(\Z_p) : \operatorname{Im}\rho_{E,p^\infty}]$ to the adelic index $[\GL_2(\widehat{\Z}) : \operatorname{Im}\rho_E]$ is smaller than $73728 \cdot 6^\alpha$, where $\alpha$ is the number of primes $p>5$ for which the image of $\rho_{E,p}$ is contained in $C_{ns}^+(p)$.
	
	\medskip
	
	We now describe the strategy behind the proof of Theorem \ref{thm:adelicboundwithj}. It combines the different results described above about the growth of the adelic index in the `vertical' and `horizontal' directions. In particular, the proof can be articulated in three main steps.
	\begin{itemize}
		\item For every odd prime $p$, we classify the possible images of $\rho_{E,p^n}$ whenever the image of $\rho_{E,p}$ is contained in the normaliser of a non-split Cartan (Theorem \ref{thm:introellcartantower}). The main aim will be to show that if $n$ is the smallest integer for which $\operatorname{Im}\rho_{E,p^\infty}$ contains $I + p^nM_{2 \times 2}(\Z_p)$, then the image of $\rho_{E,p^n}$ is exactly $C_{ns}^+(p^n)$. This will allow us to obtain a good bound on the $p$-adic index. The arguments involved are mainly group-theoretic and are presented in Section \ref{sec:grouptheory}.
		\item We provide an `effective surjectivity theorem' to show that the product of the prime powers $p^n$ for which the image of $\rho_{E,p^n}$ is contained in $C_{ns}^+(p^n)$ is bounded linearly in the stable Faltings height of $E$ (Theorems \ref{thm:effiso} and \ref{thm:totaleffiso}). This is used to bound the product of the $p$-adic indices for all the primes $p$ such that $\operatorname{Im}\rho_{E,p} \subseteq C_{ns}^+(p)$. Theorem \ref{thm:effiso} generalises a theorem by Le Fourn (\cite[Theorem 5.2]{lefourn16}) and implements the improvements developed in \cite{furiolombardo23}. Le Fourn's theorem is based on the effective isogeny theorem of Gaudron and R\'emond \cite{gaudron-remond}, which refines the arguments originally due to Masser and W\"ustholz \cite{masser-wustholz-1, masser-wustholz-2}. All the details are given in Section \ref{sec:effiso}.
		\item We give a bound on the entanglement phenomenon among all primes to obtain the bound on the adelic index from the bound on the product of the $p$-adic indices obtained via the surjectivity theorem. The main ingredient to obtain a good bound is the study of the ramification index of $p$ inside $\Q(E[p^n])$. Indeed, when the image of $\rho_{E,p^n}$ is contained in the normaliser of a non-split Cartan subgroup, $p$ is `almost totally' ramified in $\Q(E[p])$. On the other hand, by a variant of the N\'eron--Ogg--Shafarevich criterion (see for example \cite[Proposition 1]{kraus90}) we know that the ramification index of $p$ inside $\Q(E[N])$ for $p \nmid N$ is low. This shows that the intersection $\Q(E[p]) \cap \Q(E[N])$ is small. The ramification arguments rely on the work of Lozano-Robledo \cite{lozanorobledo16} and Smith \cite{smith23}. All the statements and proofs about this strategy are contained in Section \ref{subsec:entanglement}, and they are then combined in Lemma \ref{lemma:boundnonsuq}.
	\end{itemize}
	
	The proofs of Lemma \ref{lemma:boundnonsuq} and Lemma \ref{lemma:boundsuq} contains some computations in MAGMA which can be found in the file \verb*|cyc-index.m| in the ancillary files of the arXiv version of this paper \cite{furio24}.

	\subsection{Current progress on Serre's Uniformity Question}

	As already mentioned, Question \ref{question: uniformity, adelic version} was posed by Serre \cite{serre72} in different formulation. The original one, known as \emph{Serre's Uniformity Question}, is the following.
	
	\begin{question}[Serre's uniformity question]\label{question: uniformity for Q, mod-p version}
		Let $K$ be a number field. Does there exist a constant $c$, depending only on $K$, such that for every non-CM elliptic curve $\faktor{E}{K}$ and for every prime $p>c$ the residual representation $$\rho_{E,p} : \Gal\left(\faktor{\overline{K}}{K}\right) \to \Aut(E[p]) \cong \operatorname{GL}_2(\F_p)$$ is surjective?
	\end{question}
	
	As for Conjecture \ref{conj: adelic suq}, it is conjectured that this question can be answered affirmatively.
	
	\begin{conjecture}\label{conj:suqQ}
		For every elliptic curve $\faktor{E}{\Q}$ without CM and for every prime $p>37$, the representation $\rho_{E,p}$ is surjective.
	\end{conjecture}
	
	Over the years, many mathematicians provided various partial results towards an answer to Conjecture \ref{conj:suqQ}. Whenever the representation $\rho_{E,p}$ is not surjective, its image must be contained in a maximal subgroup of $\GL_2(\F_p)$. Serre classified all the maximal subgroups of $\GL_2(\F_p)$ and proved that they can be of three types: some so-called `exceptional' subgroups, the Borel subgroups, and the normalisers of (split or non-split) Cartan subgroups.
	He then showed \cite[§8.4, Lemma 18]{serre81} that for $p>13$ the exceptional subgroups cannot contain the image of $\rho_{E, p}$. Later, Mazur \cite{mazur78} proved that there are no isogenies of prime degree $p$ between non-CM elliptic curves over $\mathbb{Q}$ for $p>37$: this is equivalent to the fact that for $p>37$ the image of $\rho_{E, p}$ is not contained in a Borel subgroup. More precisely, he proved the following theorem.
	
	\begin{theorem}[Mazur]\label{thm:mazurisogeny}
		Let $\faktor{E}{\Q}$ be an elliptic curve without CM, and let $p$ be a prime such that $E$ admits a rational isogeny of degree $p$. One of the following is true:
		\begin{itemize}
			\item $p \in \{2,3,5,7,13\}$;
			\item $p=11$ and $j(E) \in \left\lbrace -11^2, \ -11 \cdot 131^3 \right\rbrace$;
			\item $p=17$ and $j(E) \in \left\{-2^{-1} \cdot 17^2 \cdot 101^3, \ -2^{-17} \cdot 17 \cdot 373^3\right\}$;
			\item $p=37$ and $j(E) \in \left\{-7 \cdot 11^3, \ -7 \cdot 137^3 \cdot 2083^3\right\}$.
		\end{itemize}
	\end{theorem}
	
	More recently, Bilu and Parent developed their version of Runge's method for modular curves \cite{bilu11runge}, which, together with reduction results from Mazur \cite{mazur78}, Momose \cite{momose84}, and Merel \cite{merel07}, and explicit isogeny theorems obtained by Gaudron and R\'emond \cite{gaudron-remond}, allowed them to prove that $\operatorname{Im}\rho_{E,p}$ is not contained in the normaliser of a split Cartan subgroup for sufficiently large $p$ \cite{bilu11split}. The result was then sharpened by Bilu--Parent--Rebolledo \cite{bilu13}, who showed that the same statement holds for every $p \geq 11$, with the possible exception of $p=13$.
	Finally, the result was extended to also cover the prime $p=13$ by means of the so-called \textit{quadratic Chabauty} method \cite{balakrishnan19}.
	
	\begin{theorem}\label{thm:split}
		Let $\faktor{E}{\Q}$ be an elliptic curve without CM. For every prime $p > 7$ the image of the representation $\rho_{E,p}$ is not contained in the normaliser of a split Cartan subgroup.
	\end{theorem}
	
	The only unsolved case is that in which the image of $\rho_{E,p}$ is contained in the normaliser of a non-split Cartan subgroup. However, Le Fourn and Lemos \cite{lefournlemos21} proved that for $p > 1.4 \cdot 10^7$ the image of $\rho_{E,p}$ cannot be contained in a proper subgroup of the normaliser of a non-split Cartan. This result was then sharpened by the author and Lombardo \cite{furiolombardo23} to also cover all primes $p>5$.
	
	\begin{theorem}\label{thm:index3incartan}
		Let $\faktor{E}{\Q}$ be an elliptic curve without CM, and let $p \ge 5$ be a prime such that the image of $\rho_{E,p}$ is contained in the normaliser of a non-split Cartan subgroup. We have that either $\operatorname{Im}\rho_{E,p}$ is equal to this normaliser, or $p=5$ and $\operatorname{Im}\rho_{E,p}$ has index $3$ in it.
	\end{theorem}
	
	\begin{proof}
		The result follows combining \cite[Theorems 1.4, 1.5, 1.6]{zywina15}, \cite[Proposition 1.13]{zywina15}, \cite[Corollary 1.3]{balakrishnan19}, and \cite[Theorem 1.6]{furiolombardo23}.
	\end{proof}
	
	We now briefly sketch the proof of the equivalence between Question \ref{question: uniformity, adelic version} and Question \ref{question: uniformity for Q, mod-p version}.
	First we notice that a uniform bound on the adelic index, easily implies that the non-surjective primes are uniformly bounded, hence it suffices to show the other implication.
	If there exists an integer $c$ such that for every prime $p$ greater than $c$ the mod-$p$ representation is surjective, then by \cite[IV-23, Lemma 3]{serre-abrep} the same holds for $p$-adic representations. Consider now, for every prime $p$ smaller than or equal to $c$, all the possible subgroups of $\operatorname{GL}_2(\F_p)$ and their corresponding modular curves. Some of these curves will have finitely many $K$-rational points and we can ignore them. For the other modular curves, if we consider a rational point on one of them and the corresponding elliptic curve $E$, either the image of the $p$-adic representation $\rho_{E,p^\infty}$ contains the group $I + pM_{2 \times 2}(\Z_p)$ (and hence the $p$-adic index is bounded), or $E$ corresponds to a $K$-rational point on a level $p^2$ modular curve. We can then repeat the same argument for modular curves of level $p^2$, and go on with higher powers of $p$. Since the genus of the modular curves grows with the level, there exists an integer $n$ such that all the modular curves of level $p^n$ have genus greater than $1$, and then they will have a finite number of $K$-rational points by Faltings theorem (see \cite[Theorem 1.3]{arai08}).
	This gives a bound on the indices of the $p$-adic representations. Serre proved that the image of $\rho_{E}$ has finite index in the product $\prod \operatorname{Im}\rho_{E,p^\infty}$ over the finite set of primes containing $2,3,5$ and those primes for which $\rho_{E,p^\infty}$ is not surjective \cite[IV-26, Lemma 5]{serre-abrep}. Since for every $p$ the pro-$p$ Sylow subgroup of $\GL_2(\Z_p)$ has a finite index, it suffices to show that the intersection of the image of $\prod \rho_{E,p^\infty}$ with the product of the pro-$p$ Sylow subgroups has finite index. However, a subgroup of a product of $p$-groups (for different primes $p$) is a product of subgroups, and since the projections on $\GL_2(\Z_p)$ have finite index, so have their product.
	
	\paragraph{Unsolved cases}
	
	At present, the only open case of Question \ref{question: uniformity for Q, mod-p version} over $\Q$ is that where the image $\rho_{E,p}$ is conjugated to the normaliser of a non-split Cartan subgroup. Methods such as the formal immersion argument and Runge's method for modular curves, developed to study the Borel and split Cartan cases, fail to tackle it. It can be shown that the same problems occur for non-split Cartan subgroups modulo $p^n$, which is why we were unable to exclude these cases.
	
	In Theorem \ref{thm:introellcartantower}, another open case other than $C_{ns}^+(p^n)$ is the level $p^2$ group $G(p^2)$ isomorphic to $C_{ns}^+(p) \ltimes (\Z/p\Z)^3$. In this case, the group-theoretic results of Section \ref{sec:grouptheory} are ineffective. For the modular curves $X_{G(p^2)}$ corresponding to the groups $G(p^2)$, we have again the same difficulties as for the modular curves $X_{ns}^+(p^n)$. However, they seem to be much more likely to be addressed, for instance, by local arguments at $p$. As an example, in the case $p=7$ the curve $X_{G(p^2)}$ appears as an unsolved case in the work of Rouse, Sutherland and Zureick-Brown (\cite[Theorem 1.6]{rszb22}). This curve has a rational CM point, making impossible to exclude the $G(7^2)$ case by purely group-theoretic means. If one managed to exclude the case in which $\operatorname{Im}\rho_{E,p^2} = G(p^2)$, they would be able to improve Theorem \ref{thm:adelicboundwithj} and show that $[\GL_2(\widehat{\Z}) : \operatorname{Im}\rho_E] < \Fheight(E)^{2+o(1)}$.
	This last problem, is being treated by the author in an ongoing project with Matt Bisatt and Davide Lombardo.
	
	\paragraph{Cited preprints}
	
	In this paper, we frequently refer to several unpublished works of Zywina, namely \cite{zywina11, zywina15, zywina15index, zywina22}. The first of these, \cite{zywina11}, which largely inspired the present work, was thoroughly reviewed by the author, and all the required lemmas are reproduced in this article together with their proofs. The second work, \cite{zywina15}, contains results that are now classical and have since been verified and superseded by more recent publications, such as \cite{sutherlandzywina17, rszb22, lefournlemos21}. From the third paper, \cite{zywina15index}, we only use Lemma~2.3, whose proof has been independently verified by the author. The fourth paper, \cite{zywina22}, contains an algorithm that was later generalized by Zywina in the published article \cite{zywina25}, which was still unpublished at the time of the first draft of this paper.
	
	\paragraph{Acknowledgments}
	
	I am grateful to my Ph.D. supervisor Davide Lombardo for the \textit{uncountably many} helpful discussions without which this work would not have been the same, and for all the time he spent reviewing preliminary versions of this paper. I would also like to thank Samuel Le Fourn for interesting discussions on this topic, for his careful reading of this article and for his valuable comments. Special thanks also to \'Eric Gaudron and Ga\"el R\'emond for their crucial help in finding some important inaccuracies in a preliminary version of this paper. Finally, I thank the anonymous referee for a thorough reading of the article and for numerous valuable suggestions that significantly improved the exposition.
	
	During the preparation of this article, the author was affiliated with the University of Pisa and was a member of the INdAM group GNSAGA.

	\section{Preliminaries}

	In this section, we prove some auxiliary results that will be used in the next sections. In particular, we focus on the properties of subgroups of $\GL_2(\Z_p)$ for $p$ prime, on the Faltings height of elliptic curves defined over $\Q$, on the rational points $P$ on the modular curves $X_{ns}^+(N)$ such that $j(P) \in \Z$, and on an effective variant of Mertens's theorems.

	\subsection{Schur--Zassenhaus for $p$-adic matrices}

	Let $p$ be an odd prime and let $K$ be a finite extension of $\Q_p$ with ring of integers $\OK$, uniformiser $\pi_K$ and residue field $\F_q=\F_{p^k}$.
	
	The following proposition is a profinite version of the Schur-Zassenhaus theorem, and can be found in \cite[Proposition 2.3.3]{wilson98}.
	
	\begin{proposition}\label{prop:schurzassenhaus}
		Let $G$ be a profinite group and let $N$ be a normal subgroup such that $|N|$ and $\left| \faktor{G}{N} \right|$ are coprime (where the cardinalities are supernatural numbers defined as in \cite[Definition 2.1.1]{wilson98}). Then $G$ has subgroups $H$ such that $G=NH$ and $H \cap N = 1$; moreover, all such subgroups $H$ are conjugate in $G$. In particular, for any such $H$ we have an isomorphism $G \cong H \ltimes N$, with the action given by conjugation.
	\end{proposition}
	
	\begin{proposition}\label{prop:groupteichmuller}
		Let $\mathcal{G}' < \GL_n(\OK)$ be a subgroup and let $\pi : \GL_n(\OK) \to \GL_n\left(\faktor{\OK}{\pi_K}\right) = \GL_n(\F_q)$ be the canonical projection. Let $G < G' := \pi(\mathcal{G}')$ be a subgroup of order prime to $p$.
		\begin{itemize}
			\item There exists a finite subgroup $\mathcal{G} < \mathcal{G}'$ such that $\pi$ induces an isomorphism $\mathcal{G} \cong G$. Moreover, $\mathcal{G}$ is unique up to conjugation in $\mathcal{G}'$.
			\item Suppose $G = G'$, let $\mathcal{G}$ be as above, and let $\mathcal{N} := \ker \pi \cap \mathcal{G}'$. We have
			$$\mathcal{G'} = \mathcal{G} \mathcal{N} \cong \mathcal{G} \ltimes \mathcal{N} \cong G \ltimes \mathcal{N},$$
			where the action is given by conjugation of $\mathcal{G}$ on $\mathcal{N}$.
		\end{itemize}
	\end{proposition}
	
	\begin{proof}
		Consider the groups
		\begin{align*}
			\mathcal{G}'':= \{A \in \mathcal{G}' \mid \pi(A) \in G \} \qquad \text{and} \qquad \mathcal{N}:= \{A \in \mathcal{G}' \mid A \equiv I \pmod {\pi_K} \}. 
		\end{align*}
		It is not difficult to notice that $\mathcal{N}$ is a pro $p$-group, $\mathcal{N} \lhd \mathcal{G}''$ and $\faktor{\mathcal{G}''}{\mathcal{N}} = G$. By Proposition \ref{prop:schurzassenhaus} there exists $\mathcal{G} < \mathcal{G}''$ such that $\mathcal{G}\mathcal{N} = \mathcal{G}''$ and $\mathcal{G} \cap \mathcal{N} = 1$, hence
		$$G = \frac{\mathcal{G}''}{\mathcal{N}} = \frac{\mathcal{G}\mathcal{N}}{\mathcal{N}} \cong \frac{\mathcal{G}}{\mathcal{G} \cap \mathcal{N}} = \mathcal{G}.$$
		This implies that $|\mathcal{G}| = |G|$, and since $\pi$ surjects $\mathcal{G}$ onto $G$, as $\pi(\mathcal{G}'') = \pi(\mathcal{G}\mathcal{N}) = \pi(\mathcal{G})$, it is an isomorphism. Suppose now that $\widehat{\mathcal{G}} < \mathcal{G}'$ is another group with the same property. Since $\pi(\widehat{\mathcal{G}}) = G$, we have $\widehat{\mathcal{G}} < \mathcal{G}''$ and $\widehat{\mathcal{G}}\mathcal{N} = \mathcal{G}'' = \mathcal{G}\mathcal{N}$. Moreover, the homomorphism $\pi|_{\widehat{\mathcal{G}}}$ is injective, so $\widehat{\mathcal{G}} \cap \mathcal{N} = 1$.
		By Proposition \ref{prop:schurzassenhaus} we conclude that $\mathcal{G}$ and $\widehat{\mathcal{G}}$ are conjugate in $\mathcal{G}''$ (and hence in $\mathcal{G}'$). The second part follows immediately from Proposition \ref{prop:schurzassenhaus}.
	\end{proof}

	\subsection{Heights}

	In this section, we give a relation between the stable Faltings height of an elliptic curve defined over $\Q$ and the Weil height of its $j$-invariant. In particular, we extend \cite[Theorem 2.8]{furiolombardo23} to also cover small values of $|j|$.
	
	In the next result, as well as in the rest of the paper, we denote by $\Fheight(E)$ the stable Faltings height of an elliptic curve $E$, with the normalisation of \cite[Section 1.2]{deligne85}, and by $\operatorname{h}(x)$ the logarithmic Weil height of a number $x \in \overline{\Q}$. For a definition of $\Fheight$, we refer the reader to \cite{deligne85}. In the case where $x = \frac{a}{b} \in \Q$, with $(a,b) =1$, we have $\operatorname{h}(x) = \log\max\{|a|,|b|\}$.
	
	\begin{theorem}\label{thm:heights}
		Let $\faktor{E}{\Q}$ be an elliptic curve with stable Faltings height $\Fheight(E)$.
		\begin{enumerate}
			\item If $|j(E)| > 3500$, then 
			\begin{equation*}\label{eq:heights-ineq-1}
				\frac{\operatorname{h}(j(E))}{12} - \frac{1}{2}\log\log|j(E)| -0.406 < \Fheight(E) < \frac{\operatorname{h}(j(E))}{12} - \frac{1}{2}\log\log|j(E)| +0.159,
			\end{equation*}
			\item If $|j(E)| \le 3500$, then
			\begin{equation*}
				\frac{1}{12}\operatorname{h}(j(E)) - 1.429 < \Fheight(E) < \frac{1}{12}\operatorname{h}(j(E)) - 0.135.
			\end{equation*}
		\end{enumerate}
	\end{theorem}
	
	\begin{proof}
		The first part is the same as in \cite[Theorem 2.8]{furiolombardo23}, hence we just need to prove the second part.
		Let $\tau \in \uhp$ be the point in the standard fundamental domain $\mathcal{F}$ such that $E(\C) \cong \faktor{\C}{\Z \oplus \tau\Z}$, and set $q=e^{2\pi i \tau}$. As in \cite[equation (2.5)]{furiolombardo23} we have
		\begin{equation*}\label{eq:height}
			12\Fheight(E) = \sum_{p \text{ prime}}\log\max\{1,\|j\|_p\} -\log|q| -6\log2 -6\log|\log|q|| -24\sum_{n=1}^\infty \log|1-q^n|,
		\end{equation*}
		and as shown in the proof of \cite[Theorem 2.8]{furiolombardo23} we have
		$$ -\frac{24|q|}{1-|q|} < -24\sum_{n=1}^\infty \log|1-q^n| < - \frac{4\pi^2}{\log|q|}.$$
		As we are assuming that $j(E) \le 3500$, by \cite[Lemma 2.5]{pazuki19} we know that $\pi \sqrt{3} \le 2\pi\Im\{\tau\} = |\log|q|| < \log(3500+970.8) < 8.41$, and so we have
		\begin{align*}
			|\log|q|| -6\log|\log|q|| +\frac{4\pi^2}{|\log|q||} < 2.533, \\
			|\log|q|| -6\log|\log|q|| -\frac{24|q|}{1-|q|} > -4.828.
		\end{align*}
		Using the inequality $0 \le \log\max\{1,|j|\}$, we can then write
		\begin{align*}
			12\Fheight(E) < \sum_{p \text{ prime}}\log\max\{1,\|j\|_p\} + \log\max\{1,|j|\} -6\log2 +2.533,
		\end{align*}
		which gives the desired upper bound.
		For the lower bound, we have that $\log\max\{1,|j|\} \le \log3500$, and then we conclude by writing
		\begin{equation*}
			12\Fheight(E) > \sum_{p \text{ prime}}\log\max\{1,\|j\|_p\} + \log\max\{1,|j|\} -\log3500 -6\log2 -4.828. \qedhere
		\end{equation*}
	\end{proof}
	
	\begin{remark}\label{minimalheight}
		As noticed by Deligne in \cite[page 29]{deligne85}, for every number field $K$ and every elliptic curve $\faktor{E}{K}$ we have that $\Fheight(E) > -0.75$.
	\end{remark}

	\subsection{Integral points on modular curves}

	The aim of this section is to study integral points on some modular curves associated with normalisers of non-split Cartan subgroups.
	Recently, Bajolet, Bilu and Matschke \cite{bilu21} stated the following.
	
	\begin{claim}\label{claim:bilu}
		Let $7 < p < 100$ be a prime. Every $P \in X_{ns}^+(p)(\Q)$ such that $j(P) \in \Z$ is a CM point.
	\end{claim}
	
	Unfortunately, there is a mistake in the article where they prove this statement, so this may not necessarily be true. Actually, there is evidence that the claim actually holds. If we assume it, some of our estimates in Section \ref{sec:effiso} and Section \ref{sec:adelicbound} can be slightly improved, however, we will avoid this assumption to lighten the proofs and give a unique unconditional result.
	\begin{remark}
		We give a quick explanation about the mistake in the proof of Claim \ref{claim:bilu}.
		In \cite[Section 9]{bilu21}, the authors define $\delta=\frac{\delta_2}{\delta_1}$ and $\lambda=\frac{\delta_2\theta_1 - \delta_1\theta_2}{\delta_1}$, and they assume that there exists an integer $r$ such that the number $r\delta$ is close to an integer, while $r\lambda$ is not. However, one can show that when $\frac{p-1}{2}$ is odd $\lambda$ belongs to $\Z[\delta]$, and so $r\lambda$ will be close to an integer too. In particular, one can assume (possibly changing their definition of $\eta_0$) that $\theta_1 = \ldots = \theta_{d-1} = 0$, and so $\lambda = 0$.
	\end{remark}
	
	We will gather most of the results we need about integral points on non-split Cartan modular curves in the following lemma.
	
	\begin{lemma}\label{lemma:cartan15e7e9}
		Consider the modular curves $X_{ns}^+(7)$, $X_{ns}^+(9)$, $X_{ns}^+(11)$, and $X_{ns}^+(3) \times_{X(1)} X_{ns}^+(5)$.
		\begin{enumerate}
			\item If $P \in X_{ns}^+(7)(\Q)$ is such that $j(P) \in \Z$, then either $P$ is a CM point or
			\begin{align}\label{eq:intjcartan7}
				j(P) \in \{2^3 \cdot 5^3 \cdot 7^5, \quad 2^{15} \cdot 7^5, \quad 2^9 \cdot 17^6 \cdot 19^3 \cdot 29^3 \cdot 149^3, \quad 2^6 \cdot 11^3 \cdot 23^3 \cdot 149^3 \cdot 269^3\}
			\end{align}
			and for the corresponding elliptic curves $E_P$ we have $[\GL_2(\widehat{\Z}) : \operatorname{Im}\rho_{E_P}] \in \{84, 504\}$.
			\item If $P \in X_{ns}^+(9)(\Q)$ is such that $j(P) \in \Z$, then either $P$ is a CM point or $j(P) = 3^3 \cdot 41^3 \cdot 61^3 \cdot 149^3$ and for the corresponding elliptic curve $E_P$ we have $[\GL_2(\widehat{\Z}) : \operatorname{Im}\rho_{E_P}] = 108$.
			\item If $P \in X_{ns}^+(11)(\Q)$ is such that $j(P) \in \Z$, then $P$ is a CM point.
			\item If $P \in (X_{ns}^+(3) \times_{X(1)} X_{ns}^+(5))(\Q)$ is such that $j(P) \in \Z$, then $P$ is a CM point.
		\end{enumerate}
	\end{lemma}
	
	\begin{proof}
		If $P \in X_{ns}^+(7)(\Q)$ and $j(P) \in \Z$, by \cite{kenku85} we know that either $P$ is a CM point or $j(P)$ belongs to the list \eqref{eq:intjcartan7}.
		Actually, Kenku's list in \cite{kenku85} contains some typos, i.e. he writes $2^2 \cdot 5^3 \cdot 7^5$ and $7^5 \cdot 2^5$ instead of $2^3 \cdot 5^3 \cdot 7^5$ and $2^{15} \cdot 7^5$ in the $j$-invariants column. The correct $j$-invariants are computed, for example, after equation (4.37) in \cite[p. 93]{elkies99}.
		By using the algorithm \verb|FindOpenImage| developed by Zywina in \cite{zywina22} (See the \href{https://github.com/davidzywina/OpenImage}{GitHub repository} accompaining the paper) we can compute the index of the image of the adelic representations attached to elliptic curves with $j$-invariant in the list \eqref{eq:intjcartan7}. Indeed, the index of $\operatorname{Im}\rho_E$ only depends on $j(E)$, as shown in \cite[Corollary 2.3]{zywina15index}.
		The first two $j$-invariants of the list give rise to elliptic curves with $[\GL_2(\widehat{\Z}) : \operatorname{Im}\rho_{E}] = 84$, while for the last two $j$-invariants we have $[\GL_2(\widehat{\Z}) : \operatorname{Im}\rho_{E}] = 504$.
		For the level $9$ case, we know by \cite[Table 5.2]{baran09} that $P$ is either a CM point or $j(P) = 3^3 \cdot 41^3 \cdot 61^3 \cdot 149^3$. We can then compute the index $[\GL_2(\widehat{\Z}) : \operatorname{Im}\rho_{E_P}]$ by using again the algorithm \verb|FindOpenImage|.
		The proofs of part 3 and part 4 can be found in \cite{schoof12} and \cite[Corollary 6.5]{chen99} respectively.
	\end{proof}

	\subsection{An effective variant of Mertens's theorems}

	In Section \ref{subsec:conductor}, we will need the following lemma, which is an effective variant of Mertens's theorems.
	
	\begin{lemma}\label{lemma: refined Kraus lemma}
		Let $N>6$ be a positive integer. We have
		\begin{equation*}\label{eq: refined Kraus lemma}
			\prod_{\substack{{p \mid N} \\ {p \text{ prime}}}} \left(1+\frac{1}{p}\right) < \frac{6e^\gamma}{\pi^2} (1+\log\log N),
		\end{equation*}
		where $\gamma$ is the Euler--Mascheroni constant.
	\end{lemma}
	
	A first result in this direction was given by Kraus in \cite{kraus95}, which bounds the product above with $4(1+\log\log N)$. In the same article \cite{kraus95}, in a note at the end of the paper, he wrote that Serre remarked that one can improve the bound by replacing the constant $4$ with $2.4$. More recently, in \cite[Corollary 2]{dedekindPsi} the authors proved that the product above is bounded by $e^\gamma \log\log N$, where $\gamma$ is the Euler--Mascheroni constant and $e^\gamma \approx 1.78$. Exploiting the results contained in \cite{dedekindPsi}, we show that we can actually replace the constant in Kraus's lemma with $\frac{6e^\gamma}{\pi^2} \approx 1.081$, which is asymptotically optimal (as shown in \cite[Proposition 3]{dedekindPsi}).
	
	\begin{proof}
		First of all, we notice that if the statement holds for a number $N$, then it holds for all the numbers $N'>N$ divisible by the same primes as $N$, hence it suffices to verify the inequality for squarefree numbers. Moreover, if $N=\prod_{i=1}^k p_i$ and $N'=\prod_{i=1}^k q_i$ with $p_i \le q_i$ for every $i$, then
		$$\frac{\prod_{p \mid N} \left(1+\frac{1}{p}\right)}{1 + \log\log N} \ge \frac{\prod_{q \mid N'} \left(1+\frac{1}{q}\right)}{1 + \log\log N'}:$$
		indeed, this is true if and only if
		$$\prod_{i=1}^k \frac{1+1/p_i}{1+1/q_i} \ge \frac{1 + \log\left(\sum_{i=1}^k \log p_i\right)}{1 + \log\left(\sum_{i=1}^k \log q_i\right)},$$
		which is true because $LHS \ge 1$ and $RHS \le 1$. 
		Therefore, it suffices to consider the primorials $N_k = \prod_{i=1}^k p_i$, where $p_i$ is the $i$-th prime number and $k \ge 3$, and the numbers $N$ whose radical is smaller than $7$. In the latter case, it suffices to notice that either $N$ is a power of a prime, and in this case we have
		$$\frac{\prod_{p \mid N} \left(1 + \frac{1}{p}\right)}{1+\log\log N} \le \frac{3}{2(1 + \log\log 7)} < 1 < \frac{6e^\gamma}{\pi^2},$$
		or $N$ has radical equal to $6$, and so we have $N \ge 12$ and
		$$\frac{\prod_{p \mid N} \left(1 + \frac{1}{p}\right)}{1+\log\log N} \le \frac{\frac{3}{2} \cdot \frac{4}{3}}{1 + \log\log 12} < 1.05 < \frac{6e^\gamma}{\pi^2}.$$
		If instead $N$ is the primorial $N_k$, suppose first that $k \ge 2263$ (or equivalently $p_k > 20000$). By \cite[Proposition 4]{dedekindPsi} we have
		\begin{align*}
			\prod_{p \mid N_k} \left(1 + \frac{1}{p}\right) &= \prod_{i=1}^k \left(1 + \frac{1}{p_i}\right) \le \frac{6\exp\left(\gamma + \frac{2}{p_k}\right)}{\pi^2} \left( \log\log N_k + \frac{1.125}{\log p_k} \right) \\
			&< \frac{6e^\gamma}{\pi^2} \left( \log\log N_k + \frac{1.125}{\log p_k} \right) + \frac{6e^\gamma}{\pi^2} \cdot \frac{3}{p_k} \left( \log\log N_k + \frac{1.125}{\log p_k} \right),
		\end{align*}
		where the last inequality comes from the fact that $e^\frac{2}{x} < 1 + \frac{3}{x}$ for $x > 20000$.
		Using the trivial inequality
		$$\log\log N_k < \log k + \log\log p_k < 2\log p_k$$
		and the fact that $p_k > 20000$ we obtain
		\begin{align}
			\begin{split}
				\prod_{p \mid N_k} \left(1 + \frac{1}{p}\right) &< \frac{6e^\gamma}{\pi^2} \left( \log\log N_k + \frac{1.125}{\log p_k} + \frac{6\log p_k}{p_k} + \frac{4}{p_k\log p_k} \right) \\
				&< \frac{6e^\gamma}{\pi^2} \left( \log\log N_k + 0.12 \right),
			\end{split}
		\end{align}
		which is better than the statement of the lemma. We can then test the remaining cases with $2 < k < 2263$; the computation takes less than one second in MAGMA.
	\end{proof}

	\section{Group theory and the level $p^n$}\label{sec:grouptheory}

	This is one of the crucial sections of the article. The aim of this section is to improve some of the results of \cite{zywina11}. We will prove that given an elliptic curve $\faktor{E}{\Q}$ without CM and a prime $p$ such that $\operatorname{Im}\rho_{E,p} \subseteq C_{ns}^+(p)$, in most cases, there exists $n \ge 1$ such that $\operatorname{Im}\rho_{E,p^n} = C_{ns}^+(p^n)$ and $\operatorname{Im}\rho_{E,p^\infty} \supset I + p^n M_{2 \times 2}(\Z_p)$. This classification result will allow us to compute the index $[\GL_2(\Z_p) : \operatorname{Im}\rho_{E,p^\infty}]$ with quite good precision.

	\subsection{Cartan lifts}

	To study the possible images of a $p$-adic Galois representation attached to an elliptic curve, we start by considering a generic subgroup of $\GL_2(\Z_p)$ satisfying some of the usual properties of these images. In particular, we will focus on the Cartan case.
	
	\begin{definition}\label{def: group kers and projs}
		Given a prime $p$ and a subgroup $G < \GL_2(\Z_p)$, for every $n \ge 1$ we define:
		\begin{itemize}
			\item $G(p^n):= G \pmod {p^n} \subseteq \GL_2\left(\faktor{\Z}{p^n\Z}\right)$;
			\item $G_n := \left\lbrace A \in G \mid A \equiv I \pmod {p^n} \right\rbrace$.
		\end{itemize}
	\end{definition}
	
	\begin{remark}
		It is not difficult to notice that $G(p^n) = \faktor{G}{G_n}$.
	\end{remark}
	
	Let $\mathfrak{gl}_2(\F_p)$ be the additive group of $2 \times 2$ matrices with coefficients in $\F_p$ and let $\mathfrak{sl}_2(\F_p)$ be the subgroup of trace $0$ matrices. They are Lie algebras over $\F_p$ when equipped with the usual bracket $[A,B]=AB-BA$.
	
	\begin{definition}\label{def:liealgebra}
		For every $n \ge 1$, we have an injective group homomorphism $\faktor{G_n}{G_{n+1}} \hookrightarrow \mathfrak{gl}_2(\F_p)$, sending $I + p^n A$ to the class of $A$ modulo $p$. We call $\mathfrak{g}_n$ the image of this homomorphism, and $\mathfrak{s}_n := \mathfrak{g}_n \cap \mathfrak{sl}_2(\F_p)$. 
	\end{definition}
	
	Given a group $G < \GL_2(\Z_p)$, throughout this section, we will call $S:=G \cap \operatorname{SL}_2(\Z_p)$. Recall that $\det(I+xA) \equiv 1 + x \operatorname{tr}A \pmod {x^2}$ as polynomials in $x$. Therefore, if $\det(G_n) \subseteq 1 + p^{n+1}\Z_p$, then $\mathfrak{g}_n \subseteq \mathfrak{sl}_2(\F_p)$. In particular, if $G=S$, then $\mathfrak{g}_n \subseteq \mathfrak{sl}_2(\F_p)$ for all $n \ge 1$, and so $\mathfrak{s}_n$ is the image of $\faktor{S_n}{S_{n+1}}$ in $\mathfrak{gl}_2(\F_p)$.
	
	As shown in \cite[Lemma 2.2]{zywina11}, the groups $\mathfrak{g}_n$ have the following properties.
	
	\begin{lemma}[Zywina]\label{lemma:gncontainment}
		Let $G < \GL_2(\Z_p)$ be a closed subgroup and let $n \ge 1$ be an integer.
		\begin{enumerate}
			\item If $p>2$, $\mathfrak{g}_n \subseteq \mathfrak{g}_{n+1}$. If $p=2$, the same statement holds for $n \ge 2$.
			\item If $\mathfrak{g}_n = \mathfrak{gl}_2(\F_p)$, then $G \supseteq I + p^nM_2(\Z_p)$.
			\item If $\mathfrak{g}_n = \mathfrak{g}_{2n}$, then $\mathfrak{g}_n$ is a Lie subalgebra of $\mathfrak{gl}_2(\F_p)$.
			\item Let $H:=[G,G]$ be the commutator subgroup of $G$, and call $\mathfrak{h}_n$ its filtration. If $\mathfrak{g}_n = \mathfrak{gl}_2(\F_p)$, then $\mathfrak{h}_{2n+e} = \mathfrak{sl}_2(\F_p)$, where $e$ equals $0$ or $1$ if $p$ is odd or even respectively.
		\end{enumerate}
	\end{lemma}
	
	We write here a short proof of the lemma for convenience.
	
	\begin{proof}
		Take any $B_1, B_2, B_3 \in \mathfrak{g}_n$. There exist $A_i \in M_2(\Z_p)$ such that $A_i = B_i \pmod p$ and $g_i := I + p^nA_i \in G$ for $i=1,2,3$.
		\begin{enumerate}
			\item We have $$g_1^p = (I+p^nA_1)^p \equiv I + p^{n+1}A_1 + \binom{p}{2} p^{2n} A^2 \pmod {p^{3n}}.$$
			When $p$ is odd, we this implies that $I + p^{n+1}A_1$ is congruent to $(I+p^nA_1)^p$ modulo $p^{n+2}$, and it therefore belongs to $G(p^{n+2})$. In particular, we have $B_1 \in \mathfrak{g}_{n+1}$. The case $p=2$ follows the same argument, noting that $2n \ge n+2$ for $n \ge 2$.
			\item By part (1), for every $i \ge n$ we have $\mathfrak{g}_i = \mathfrak{gl}_2(\F_p)$. In particular, for all $i \ge n$ we have $|G(p^i)| = |G(p^n)| p^{4(i-n)}$, and since $G$ is closed this is equivalent to say that $G \supseteq I + p^nM_2(\Z_p)$.
			\item Consider the commutator $g = g_1 g_2 g_1^{-1} g_2^{-1}$. This belongs to $[G,G] \subseteq G$, and we have
			\begin{align*}
				g &= (I+p^nA_1)(I+p^nA_2)(I+p^nA_1)^{-1}(I+p^nA_2)^{-1} \\
				&= \Big((I+p^nA_2)(I+p^nA_1) + p^{2n}(A_1A_2 - A_2A_1)\Big)(I+p^nA_1)^{-1}(I+p^nA_2)^{-1} \\
				&= I + p^{2n}(A_1A_2 - A_2A_1)(I+p^nA_1)^{-1}(I+p^nA_2)^{-1} \\
				&\equiv I + p^{2n}(A_1A_2 - A_2A_1) \mod {p^{3n}}.
			\end{align*}
			Therefore, $[B_1, B_2] = B_1B_2 - B_2B_1 \equiv A_1A_2 - A_2A_1 \pmod p$ is an element of $\mathfrak{g}_{2n}$. Since $\mathfrak{g}_n = \mathfrak{g}_{2n}$, we deduce that $\mathfrak{g}_n$ is closed under the Lie bracket $[ \cdot , \cdot]$. 
			\item Define $B_1 = \footnotesize \begin{pmatrix} 0 & 1 \\ 0 & 0 \end{pmatrix}$, $B_2 = \footnotesize \begin{pmatrix} 0 & 0 \\ 1 & 0 \end{pmatrix}$, and $B_3 = \footnotesize \begin{pmatrix} 1 & 0 \\ 0 & -1 \end{pmatrix}$: they generate $\mathfrak{sl}_2(\F_p)$. As in part 3 of the proof, the commutator $h = g_ig_jg_i^{-1}g_j^{-1}$ belongs to $H \cap (I + p^{2n}M_2(\Z_p))$ and $h \equiv I + p^{2n} [B_i, B_j] \pmod {p^{2n+1+e}}$. Moreover, we have $[B_1, B_2] = B_3$, $[B_2, B_3] = 2B_2$, and $[B_3, B_1] = 2B_1$. The conclusion follows.
			\qedhere
		\end{enumerate}
	\end{proof}
	
	\begin{definition}\label{def:cartan}
		Let $p$ be an odd prime and let $\varepsilon$ be a fixed integer whose reduction modulo $p$ represents a quadratic non-residue in $\F_p^\times$. We define the following subgroups of $\GL_2(\Z_p)$:
		\begin{align*}
			\text{Borel:} \qquad &B:= \left\lbrace \begin{pmatrix} a & b \\ 0 & c \end{pmatrix} \,\middle|\, a,b,c \in \Z_p, \ a,c \in \Z_p^\times \right\rbrace, \\
			\text{split Cartan:} \qquad &C_{sp}:= \left\lbrace \begin{pmatrix} a & 0\\ 0 & b \end{pmatrix} \,\middle|\, a,b \in \Z_p^\times \right\rbrace, \\
			\text{non-split Cartan:}  \qquad &C_{ns}:= \left\lbrace \begin{pmatrix} a & \varepsilon b \\ b & a \end{pmatrix} \,\middle|\, a,b \in \Z_p, \ (a,b) \not\equiv (0,0) \mod p \right\rbrace.
		\end{align*}
		Define also $C_{sp}^+ := C_{sp} \cup \begin{pmatrix} 0 & 1 \\ 1 & 0 \end{pmatrix} C_{sp}$ and $C_{ns}^+:= C_{ns} \cup \begin{pmatrix} 1 & 0 \\ 0 & -1 \end{pmatrix} C_{ns}$. These are the normalisers of $C_{sp}$ and $C_{ns}$ respectively.
	\end{definition}
	
	Throughout this section, we will indicate with $C$ a generic Cartan subgroup, which is either $C_{sp}$ or $C_{ns}$. Moreover, we will assume that $p$ is an odd prime.
	
	\begin{definition}
		Let $p$ be an odd prime and let $G < \GL_2(\Z_p)$ be a subgroup. We will say that $G$ is an \emph{N-Cartan lift} if it satisfies the following properties:
		\begin{itemize}
			\item $G$ is closed;
			\item $\det(G) = \Z_p^\times$;
			\item $G(p)$ is contained in the normaliser of a Cartan subgroup, but it is not contained in the Cartan subgroup itself;
			\item $G(p)$ contains an element of the Cartan subgroup which is not a multiple of the identity.
		\end{itemize}
		We will say that $G$ is a \emph{split N-Cartan lift} or a \emph{non-split N-Cartan lift} if $G(p)$ is contained in the normaliser of a split or non-split Cartan respectively.
	\end{definition}
	
	Let $G < \GL_2(\Z_p)$ be a subgroup such that $p \nmid |G(p)|$. We notice that for every $n \ge 1$, the group $G$ acts on $G_n$ by conjugation, and hence also on the quotient $\faktor{G_n}{G_{n+1}}$. This action factors through the group $G(p)$. In particular, this implies that $\mathfrak{g}_n$ is an $\F_p[G(p)]$-module, with $G(p)$ acting by conjugation, and hence it is a submodule of $\mathfrak{gl}_2(\F_p)$. The following result by Zywina (\cite[Lemma 2.4]{zywina11}) describes the decomposition of $\mathfrak{gl}_2(\F_p)$ in irreducible $\F_p[G(p)]$-submodules.
	
	\begin{lemma}[Zywina]\label{lemma:subrepdecomposition}
		Let $G < \GL_2(\Z_p)$ be an N-Cartan lift with respect to the Cartan group $C$. Suppose that there exists an element $\alpha \in G(p) \cap C(p)$ whose image in $\operatorname{PGL}_2(\F_p)$ has order greater than $2$. 
		Then, the $\F_p[G(p)]$-module $\mathfrak{gl}_2(\F_p)$ decomposes in irreducible non-isomorphic submodules as $\mathfrak{gl}_2(\F_p) = V_1 \oplus V_2 \oplus V_3$, where $V_1 = \F_p \cdot I$, $V_2 = \{ A \in C(p) \mid \operatorname{tr}(A) = 0 \} \cup \{0\}$, and $V_3$ has dimension $2$ over $\F_p$ and it is not a Lie subalgebra of $\mathfrak{gl}_2(\F_p)$.
		Moreover:
		\begin{itemize}
			\item If $G$ is a non-split N-Cartan lift, then
			$$V_1=\F_p\cdot I, \quad V_2=\F_p \begin{pmatrix} 0 & \varepsilon \\ 1 & 0 \end{pmatrix}, \quad V_3=\F_p \begin{pmatrix} 1 & 0 \\ 0 & -1 \end{pmatrix} \oplus \F_p \begin{pmatrix} 0 & \varepsilon \\ -1 & 0 \end{pmatrix}.$$
			\item If $G$ is a split N-Cartan lift, then
			$$V_1=\F_p\cdot I, \quad V_2=\F_p \begin{pmatrix} 1 & 0 \\ 0 & -1 \end{pmatrix}, \quad V_3=\F_p \begin{pmatrix} 0 & 1 \\ 1 & 0 \end{pmatrix} \oplus \F_p \begin{pmatrix} 0 & 1 \\ -1 & 0 \end{pmatrix}.$$
		\end{itemize}
	\end{lemma}
	
	We write here Zywina's proof for convenience.
	
	\begin{proof}
		We notice that $V_1 := \F_p \cdot I$ and $V_2 := \{ A \in C(p) \mid \operatorname{tr}(A) = 0 \}$ are stable under conjugation by every matrix in $C^+(p)$, and hence by $G(p)$. In particular, there exists a $2$-dimensional $\F_p[G(p)]$-module $V_3$ such that $\mathfrak{gl}_2(\F_p) = V_1 \oplus V_2 \oplus V_3$. Notice that $\dim V_3$ is strictly greater than $\dim V_1$ and $\dim V_2$, and that $G(p)$ acts trivially on $V_1$ and non-trivially on $V_2$ (because $G(p) \subseteq C^+(p)$ but $G(p) \subsetneq C(p)$), hence $V_1, V_2, V_3$ are pairwise non-isomorphic.
		We now show that $V_3$ is irreducible. We notice that the action of $G(p)$ on $\mathfrak{gl}_2(\F_p)$ factors through a faithful action of $H := \faktor{G(p)}{G(p) \cap \F_p^\times}$, hence, to show that $V_3$ is irreducible, it suffices to show that $H$ is not abelian.
		Since the order of $\alpha$ in $H$ is greater than $2$, we notice that $\operatorname{tr}(\alpha) \ne 0$. On the other hand, if $H$ were abelian, we would have an element $m \in G(p) \setminus C(p)$ such that $m \alpha m^{-1} \alpha^{-1} = \zeta \in \F_p^\times$. In particular, $\alpha$ and $m \alpha m^{-1} = \zeta\alpha$ are both roots of the polynomial $t^2 - \operatorname{tr}(\alpha) t + \det(\alpha)$, and so $0 \ne \operatorname{tr}(\alpha) \cdot I = (1+\zeta) \alpha$. Taking the trace we obtain $2\operatorname{tr}(\alpha) = (1+\zeta) \operatorname{tr}(\alpha)$, and since $\operatorname{tr}(\alpha) \ne 0$ we obtain that $\zeta = 1$, and therefore $\alpha$ is a scalar matrix. However, this is impossible, as $\alpha$ would have order $1$ in $H$. \\
		By Definition \ref{def:cartan}, we can verify that the descriptions of $V_3$ given in the split and non-split cases are stable under the action of $C^+(p)$. Moreover, $V_3$ is never a Lie subalgebra of $\mathfrak{gl}_2(\F_p)$, since in the split case we have $\footnotesize \left[\begin{pmatrix} 0 & 1 \\ -1 & 0 \end{pmatrix}, \begin{pmatrix} 0 & 1 \\ 1 & 0 \end{pmatrix}\right] = \begin{pmatrix} 2 & 0 \\ 0 & -2 \end{pmatrix}$, while in the non-split case we have $\footnotesize \left[\begin{pmatrix} 1 & 0 \\ 0 & -1 \end{pmatrix}, \begin{pmatrix} 0 & \varepsilon \\ -1 & 0 \end{pmatrix}\right] = \begin{pmatrix} 0 & \varepsilon \\ 1 & 0 \end{pmatrix}$.
	\end{proof}
	
	\begin{remark}\label{rmk:cartanliealg}
		If $G=C$ is a Cartan subgroup, by a direct computation it is easy to check that for every $n \ge 1$ we have $\mathfrak{g}_n = V_1 \oplus V_2$. Similarly, if $G=C^+$, since $[C^+:C]=2$, we have $\mathfrak{g}_n = V_1 \oplus V_2$.
	\end{remark}
	
	\begin{remark}\label{rmk:V3decomposes}
		When every element in the image of $G(p) \cap C(p)$ in $\operatorname{PGL}_2(\F_p)$ has order $1$ or $2$, one can verify that $V_3$ decomposes into two irreducible submodules. In the split case, $V_3$ decomposes as $\footnotesize\F_p \begin{pmatrix} 0 & 1 \\ 1 & 0 \end{pmatrix} \oplus \F_p \begin{pmatrix} 0 & 1 \\ -1 & 0 \end{pmatrix}$. In the non-split case, $V_3$ decomposes as $\footnotesize\F_p \begin{pmatrix} 1 & 0 \\ 0 & -1 \end{pmatrix} \oplus \F_p \begin{pmatrix} 0 & \varepsilon \\ -1 & 0 \end{pmatrix}$.
	\end{remark}
	
	The following lemma is proved by Zywina in \cite[Lemma 2.5]{zywina11}.
	
	\begin{lemma}[Zywina]\label{lemma:ovvissimoSL2}\label{cor:v1inliealg}
		Suppose $G < \GL_2(\Z_p)$ is a closed subgroup such that $\det G \supseteq 1+p\Z_p$ and $p \nmid |G(p)|$. Then, for every $n \ge 1$ we have $\operatorname{tr}(\mathfrak{g}_n) = \F_p$. In particular, if $G < \GL_2(\Z_p)$ is an N-Cartan lift, then $V_1 \subseteq \mathfrak{g}_n$ for every $n \ge 1$.
	\end{lemma}
	
	\begin{proof}
		Since for every $A \in M_2(\Z_p)$ we have $\det(I + p^nA) \equiv I + p^n \operatorname{tr}A \pmod {p^{n+1}}$, we notice that $g \in G_{n}$ has $\det(g) \equiv 1 \pmod {p^{n+1}}$ if and only if $g = I + p^n A$ with $\operatorname{tr}A \equiv 0 \pmod p$. Since $\det G \supseteq 1 + p\Z_p$, there exists $g_0 \in G$ with $\det g_0 \not\equiv 1 \pmod {p^2}$. Using the fact that $p \nmid |G(p)|$, we notice that $h = g_0^{p^{n-1}|G(p)|}$ is an element of $G_n$ such that $\det h \not\equiv 1 \pmod {p^{n+1}}$. In particular, this implies that $\operatorname{tr}(\mathfrak{g}_n) \ne 0$, and so $\operatorname{tr}(\mathfrak{g}_n) = \F_p$.
		The last statement follows from Lemma \ref{lemma:subrepdecomposition}, noting that $\operatorname{tr}(V_2 \oplus V_3) = 0$, and hence $V_1 \subseteq \mathfrak{g}_n$.
	\end{proof}
	
	\hspace{-2pt}The following lemma shows that if $G$ is an N-Cartan lift, then $\dim\mathfrak{g}_n$ is almost never equal to $3$.
	
	\begin{lemma}\label{lemma:dimg}
		Let $G$ be an N-Cartan lift. Suppose that the image of $G(p) \cap C(p)$ in $\operatorname{PGL}_2(\F_p)$ contains an element of order greater than $2$.
		\begin{enumerate}
			\item If $\dim\mathfrak{g}_1=2$, then for every $n>1$ we have $\dim\mathfrak{g}_n \in \{2,4\}$, and if $\dim\mathfrak{g}_n=4$ for some $n$, then for every $m>n$ the equality $\dim\mathfrak{g}_m=4$ holds.
			\item If $\dim\mathfrak{g}_1=3$, then $\dim\mathfrak{g}_n=4$ for every $n>1$.
		\end{enumerate}
	\end{lemma}
	
	\begin{proof}
		To prove the first part, it suffices to notice that by Lemma \ref{lemma:subrepdecomposition} and Lemma \ref{cor:v1inliealg} we have $\mathfrak{g}_1 = V_1 \oplus V_2$, where $V_1$, $V_2$ and $V_3$ are defined in Lemma \ref{lemma:subrepdecomposition}. Using Lemma \ref{lemma:gncontainment}(1) we see that for every $n \ge 1$ we have either $\mathfrak{g}_n = V_1 \oplus V_2$ or $\mathfrak{g}_m = \mathfrak{gl}_2$ for every $m \ge n$, and hence the conclusion follows. \\
		To prove the second part, we similarly prove that $\mathfrak{g}_1 = V_1 \oplus V_3$ and if $\dim\mathfrak{g}_2<4$, then $\mathfrak{g}_2=\mathfrak{g}_1$ (by Lemma \ref{lemma:gncontainment}(1)). By Lemma \ref{lemma:gncontainment}(3), this implies that $\mathfrak{g}_1$ is a Lie subalgebra of $\mathfrak{gl}_2(\F_p)$, hence also $\mathfrak{s}_1=\mathfrak{g}_1 \cap \mathfrak{sl}_2(\F_p) = V_3$ is a Lie subalgebra of $\mathfrak{gl}_2(\F_p)$, which contradicts Lemma \ref{lemma:subrepdecomposition}.
	\end{proof}
	
	The following proposition is a stronger version of \cite[Proposition 1.2]{zywina11}.
	
	\begin{proposition}\label{prop:cartanzywina}
		Let $G < \GL_2(\Z_p)$ be an N-Cartan lift with respect to the Cartan group $C$ such that $\dim \mathfrak{g}_1 > 1$, and suppose that there exists an element in  $G(p) \cap C(p)$ whose image in $\operatorname{PGL}_2(\F_p)$ has order greater than $2$. For every integer $n \ge 1$ we have the following.
		\begin{enumerate}
			\item If $\dim \mathfrak{g}_n = 2$, then $G(p^n)$ is contained in $C^+(p^n)$ up to conjugation and $[C^+(p^n) : G(p^n)] = [C^+(p) : G(p)]$;
			\item If $\dim\mathfrak{g}_n=3$, then $n=1$ and $G \supset I + p^2 M_{2 \times 2}(\Z_p)$;
			\item If $\dim \mathfrak{g}_n = 4$, then $G \supset I + p^{n} M_{2 \times 2}(\Z_p)$.
		\end{enumerate}
	\end{proposition}
	
	Before proving it, we need the following lemma, which is a simple version \cite[Lemma 2.1]{zywina11}.
	
	\begin{lemma}\label{lemma: zywina 2.1(iv)}
		Let $\alpha \in \GL_2\left(\Z_p\right)$ be an element such that $\operatorname{tr}(\alpha)^2 - 4\det(\alpha) \not\equiv 0 \pmod p$.
		Let $n$ be a positive integer, and let $H$ be a cyclic subgroup of $\GL_2\left(\faktor{\Z}{p^n\Z}\right)$ generated by a matrix of the form $h = I + p^iA$, with $1 \le i < n$.
		\begin{itemize}
			\item there exists $\beta \in \GL_2\left(\Z_p\right)$ such that $\beta\alpha\beta^{-1} \in C$ for some Cartan group $C$.
			\item Suppose that $\alpha \in C$ and $A \mod p \in C(p)$. If $H$ is stable under conjugation by $\alpha \mod p^n$, then $H \subseteq C(p^n)$.
		\end{itemize}
	\end{lemma}
	
	\begin{proof}
		Set $\Delta := \operatorname{tr}(\alpha)^2 - 4\det(\alpha)$. This is the discriminant of the characteristic polynomial of $\alpha$. Since it is not $0$ modulo $p$, the reduction of $\alpha$ modulo $p$ has distinct eigenvalues in $\overline{\F}_p$, and hence it is diagonalisable. The same must hold in $\Z_p$. Notice that if $\Delta = b^2$ is a square, then the diagonal form of $\alpha$ is $\footnotesize \beta = \begin{pmatrix} (\operatorname{tr}(\alpha)+b)/2 & 0 \\ 0 & (\operatorname{tr}(\alpha)-b)/2 \end{pmatrix} \in C_{sp}$, because it has the same eigenvalues. If instead $\Delta = \varepsilon b^2$ is not a square, $\alpha$ has the same diagonal form as $\footnotesize \beta = \begin{pmatrix} \operatorname{tr}(\alpha)/2 & \varepsilon b \\ b & \operatorname{tr}(\alpha)/2 \end{pmatrix} \in C_{ns}$. In particular, these means that $\alpha$ and $\beta$ are conjugated over $\Q_p$. We now show that they are actually conjugated over $\Z_p$. We prove it in the split case; the non-split case is proved in the same way after passing to the quadratic unramified extension of $\Q_p$. Consider two eigenvectors $e_1, e_2$ for the two eigenvalues of the matrix $\alpha$. We can assume that $e_1, e_2 \in \Z_p^2$ with at least one invertible coordinate. If we take the matrix $M=( e_1 , e_2 )$, we have that $\alpha = M \beta M^{-1}$. Reducing modulo $p$, we see that $e_1, e_2$ are eigenvectors for $\alpha \mod p$, and since $\alpha \mod p$ is diagonalisable we obtain that $M \mod p$ must be invertible. This means that $\det M \in \Z_p^\times$, and hence that $M^{-1}$ also has coefficients in $\Z_p$. \\
		We now prove the second part. Since $(I+p^iA)^k \equiv 0 \pmod {p^{n}}$ if and only if $k \equiv 0 \pmod {p^{n-i}}$, the matrix $h$ has order $p^{i}$ in $G(p^{n})$, and hence $|H| = p^{n-i}$. If we identify $\alpha$ with its reduction modulo $p^n$, by hypothesis, $\alpha h \alpha^{-1}$ is another generator of $H$, hence there exists a unique $m \in \left(\faktor{\Z}{p^{n-i}\Z}\right)^\times$ such that $\alpha h \alpha^{-1} = h^m$. Since $A \pmod p \in C(p)$, we have $\alpha A \alpha^{-1} \equiv A \pmod p$, and $A$ is non-zero. In particular, we have
		$$I + p^iA \equiv \alpha h \alpha^{-1} = h^m \equiv I + mp^i A \pmod {p^{i+1}},$$
		and so $m \equiv 1 \pmod p$. This means that the conjugation by $\alpha$ gives a group automorphism of $H$ with order dividing $p^{n-i-1}$, hence $\gamma := \alpha^{p-i-1}$ commutes with every element of $H$. Since the order of $\alpha \pmod p$ is coprime with $p$, we still have $\operatorname{tr}(\beta)^2 - 4 \det(\beta) \not\equiv 0 \pmod p$, and one can verify that every element in $C(p^n)$ with this property has centraliser equal to $C(p^n)$. We then conclude that $H \subseteq C(p^n)$.
	\end{proof}
	
	\begin{proof}[\textbf{Proof of Proposition \ref{prop:cartanzywina}}]
		Set $S=G \cap \operatorname{SL}_2(\Z_p)$. The proof of this proposition follows that of \cite[Proposition 2.3]{zywina11}.
		If $\dim \mathfrak{g}_n = 3$, by Lemma \ref{lemma:dimg} we know that $n=1$ and $\dim \mathfrak{g}_2 = 4$, so by Lemma \ref{lemma:gncontainment}(2) we have $G \supset I + p^2 M_{2 \times 2}(\Z_p)$. If $\dim\mathfrak{g}_n=4$, by Lemma \ref{lemma:gncontainment}(2) we have $G \supset I + p^n M_{2 \times 2}(\Z_p)$. We now focus on the case $\dim \mathfrak{g}_n = 2$. By Lemma \ref{cor:v1inliealg} we have that $\dim \mathfrak{s}_i = 1$, so by immediate induction, $\faktor{S_1}{S_{i+1}}$ is of order $p^i$ for every $i \in \{1, \dots, n\}$. In particular, lifting to $S_1$ a
		non-zero element of $\mathfrak{s}_1$ and projecting it to $S(p^{n+1})$, we find an element $h= I + pA \in S(p^{n+1})$ such that $A \not\equiv 0 \pmod p$. Since $(I+pA)^k \equiv 0 \pmod {p^{n+1}}$ if and only if $k \equiv 0 \pmod {p^n}$, the matrix $h$ has order $p^{n}$ in $G(p^{n+1})$, and hence, by cardinality arguments, the group $H:= S_1 \pmod {p^{n+1}}$ is generated by $h$. As $\dim \mathfrak{s}_1 = 1$, by Lemma \ref{lemma:subrepdecomposition} we have $\mathfrak{s}_1 = V_2$, and hence the matrix $A \pmod p$ is a non-zero element of $V_2$. In particular, we have $A \pmod p \in C(p)$.
		Let $\alpha \in G(p^{n+1})$ be an element whose reduction modulo $p$ lies in $C(p)$ and its image in $\operatorname{PGL}_2(\F_p)$ has order grater than $2$. It is easy to verify that such an element always satisfies $\operatorname{tr}(\alpha)^2 - 4\det(\alpha) \not\equiv 0 \pmod p$. In particular, since $H$ is stable under conjugation by $\alpha \in G(p^{n+1})$ and $A \pmod p$ is an element of $C(p)$, by Lemma \ref{lemma: zywina 2.1(iv)} we can assume (up to conjugation) that $\alpha \in C(p^{n+1})$ and hence that $H \subset C(p^{n+1})$. We have
		$$H = \left\{ g \in C(p^{n+1}) \cap \operatorname{SL}_2\left(\faktor{\Z}{p^{n+1}\Z}\right) : g \equiv I \pmod p \right\},$$
		since the inclusion ``$\subseteq$" is trivial and the equality follows by cardinality. 
		Consider the group $C_1(p^{n+1}) := \{ M \in C(p^{n+1}) \mid M \equiv I \pmod p \}$: this is generated by the subgroups $H$ and $\{(1+p\alpha)I\}$; indeed, they are disjoint and the product of their cardinalities equals $|C_1(p^{n+1})|$. As $G(p^{n+1})$ normalises $H$, it also normalises the group $C_1(p^{n+1})$, since every matrix in this group can be written as $M=(1+p\alpha)h^k$, for some $k \in \N$ and $\alpha \in \faktor{\Z}{p^{n+1}\Z}$.
		Consider an element $I+pA \in \GL_2\left(\faktor{\Z}{p^{n+1}\Z}\right)$: this is in $C_1(p^{n+1})$ if and only if $A \pmod {p^{n}}$ is contained in the algebra $R(p^n)$ generated by $C(p^n)$ inside $M_2\left(\faktor{\Z}{p^n\Z}\right)$ 
		(which is a rank-$2$ module over $\faktor{\Z}{p^n\Z}$).
		For every $g \in G$ and $A \in R(p^n)$ we have $g^{-1}(I+pA)g = I + pg^{-1}Ag \in C_1(p^{n+1})$, and so $g^{-1}Ag \pmod {p^n} \in R(p^n)$. This implies that $G(p^n)$ normalises $R(p^n)$, and so $G(p^n) \subseteq C^+(p^n)$. Moreover, $$|G(p^n)| = |G(p)| \cdot \prod_{i=1}^{n-1} |\mathfrak{g}_i| = |G(p)| \cdot p^{2n-2} = |G(p)| \cdot \frac{|C^+(p^n)|}{|C^+(p)|},$$
		and hence we have $[C^+(p^n) : G(p^n)] = [C^+(p) : G(p)]$.
	\end{proof}
	
	If $G$ is an N-Cartan lift such that $\dim \mathfrak{g}_n = 4$ for sufficiently large $n$, a statement equivalent to the proposition above (if we are not in the case $\dim \mathfrak{g}_1 = 3$) is that if $n$ is the largest positive integer such that $G(p^n) \subseteq C_{ns}^+(p^n)$, then $G \supset I + p^{n+1} M_{2 \times 2}(\Z_p)$. However, if we add the hypothesis that $G$ contains many scalar matrices, we can prove a stronger result.
	
	\begin{theorem}\label{thm:cartantower}
		Let $G < \GL_2(\Z_p)$ be an N-Cartan lift as in Proposition \ref{prop:cartanzywina} and such that $G \supset (1+p\Z_p) I$.
		One of the following holds:
		\begin{itemize}
			\item $G < C^+$ up to conjugation and $[C^+ : G] = [C^+(p) : G(p)]$;
			\item There exists $n \ge 1$ such that $G \supseteq I + p^n M_{2 \times 2}(\Z_p)$ and $G(p^n) \subseteq C^+(p^n)$ up to conjugation, with $[C^+(p^n) : G(p^n)] = [C^+(p) : G(p)]$;
			\item $G \supseteq I + p^2 M_{2 \times 2}(\Z_p)$ and
			$$G(p^2) \cong G(p) \ltimes (V_1 \oplus V_3),$$
			with $V_i$ defined as in Lemma \ref{lemma:subrepdecomposition} and the semidirect product defined by the conjugation action.
		\end{itemize}
	\end{theorem}
	
	Notice that the main improvement of Theorem \ref{thm:cartantower} with respect to Proposition \ref{prop:cartanzywina} is that when $\dim \mathfrak{g}_{n} = 2$ we have $G(p^{n+1}) \subseteq C^+(p^{n+1})$, which is better than $G(p^{n}) \subseteq C^+(p^{n})$.
	
	\begin{proof}
		Suppose that $\dim \mathfrak{g}_1 =3$. By Proposition \ref{prop:cartanzywina}(1) we know that $G \supset I + p^2M_{2 \times 2}(\Z_p)$. We can then apply Proposition \ref{prop:groupteichmuller} to obtain $G \cong G(p) \ltimes G_1$, and projecting modulo $p^2$ we have $G(p^2) \cong G(p) \ltimes \mathfrak{g}_1$, where by Lemma \ref{lemma:subrepdecomposition} we have $\mathfrak{g}_1 = V_1 \oplus V_3$.
		Suppose now $\dim \mathfrak{g}_1 \ne 3$. We prove that, given $n \ge 2$, if $\dim \mathfrak{g}_{n-1} < 4$ we have $G(p^n) \subseteq C^+(p^n)$ up to conjugation and $[C^+(p^n) : G(p^n)] = [C^+(p) : G(p)]$. Moreover, we will show that this is sufficient to conclude.
		We divide the proof in 5 steps. \\
		\textbf{1.} Since $2 \le \dim \mathfrak{g}_{n-1} < 4$ and $\dim \mathfrak{g}_{n-1} \ne 3$, we have $\dim \mathfrak{g}_{n-1} = 2$, and so by Proposition \ref{prop:cartanzywina}(1) we know that $G(p^{n-1}) \subseteq C^+(p^{n-1})$, with $[C^+(p^{n-1}) : G(p^{n-1})] = [C^+(p) : G(p)]$. \\
		\textbf{2.}  We now prove that the subgroup $G_1(p^n) \subseteq G(p^n)$ coincides with the group $H = C_1(p^n) := \{g \in C(p^n) : g \equiv I \pmod p \}$.
		By the proof of Proposition \ref{prop:cartanzywina} we know that $$H_2 := \left\{ g \in C(p^{n}) \cap \operatorname{SL}_2\left(\faktor{\Z}{p^{n}\Z}\right) : g \equiv I \pmod p \right\} \subset G(p^n),$$
		moreover, by hypothesis we have that the group $H_1 = \{(1+pk)I \mod p^n\}$ is also contained in $G(p^n)$. We notice that $|H_1|=|H_2|=p^{n-1}$ and $|H|=p^{2n-2}$. Moreover, $H_1$ is normal in $H$, hence $H_1H_2$ is a subgroup of $H \cap G(p^n)$. It is easy to notice that $\det (1+kp)I \equiv 1 \pmod {p^n}$ if and only if $k \equiv 0 \pmod {p^{n-1}}$ and so if and only if $(1+kp)I \equiv I \pmod {p^n}$. This implies that $H_1 \cap H_2 = \{I\}$, and so $|H_1H_2| = |H_1| \cdot |H_2| = |H|$, in particular $H = H_1H_2 \subseteq G(p^n)$. As by Lemma \ref{lemma:dimg} we know that $|G_1(p^n)| = \prod_{i=1}^{n-1} |\mathfrak{g}_i| = p^{2n-2}$ and $H \subseteq G_1(p^n)$, we have that $G_1(p^n) = H$. \\
		\textbf{3.}	Since $p \nmid |G(p)|$, by Proposition \ref{prop:groupteichmuller} there exists a subgroup $\widetilde{G(p)} < G$ such that the projection modulo $p$ induces an isomorphism $\widetilde{G(p)} \cong G(p)$, and modulo $p^n$ we have $G(p^n) = \widetilde{G(p)} \cdot G_1(p^n) = \widetilde{G(p)} \cdot H$, where we identified $\widetilde{G(p)}$ with its projection modulo $p^n$.
		Consider
		$$\Gamma := \{A \in \GL_2(\Z_p) \mid A \mkern-10mu \pmod {p^{n-1}} \in C^+(p^{n-1})\} < \GL_2(\Z_p).$$
		By Proposition \ref{prop:cartanzywina} we know that $G < \Gamma$, and obviously $C^+ < \Gamma$. By Proposition \ref{prop:groupteichmuller}, there is a group $\widetilde{C^+(p)} < C^+$ isomorphic to $C^+(p)$ via the projection modulo $p$ such that $C^+ = \widetilde{C^+(p)} \cdot C^+_1 \cong \widetilde{C^+(p)} \ltimes C^+_1$. We can then consider the unique subgroup $G' < C^+$ such that $G'_1 = C^+_1$ and $G'(p) = G(p)$. By Proposition \ref{prop:groupteichmuller} we have $\widetilde{G'(p)} < \widetilde{C^+(p)}$, and since $G, G' < \Gamma$ we have $\widetilde{G(p)} \equiv \widetilde{G'(p)} \pmod {p^{n-1}}$. Moreover, $\widetilde{G(p)}$ and $\widetilde{G'(p)}$ are conjugate in $\Gamma$, i.e. there exists $\gamma \in \Gamma$ such that $\gamma^{-1}\widetilde{G(p)}\gamma = \widetilde{G'(p)}$. \\
		\textbf{4.} We notice that $G_1(p^n) = G'_1(p^n) = C^+_1(p^n) = H$. If we identify $\widetilde{G(p)}$, $\widetilde{G'(p)}$ and $\widetilde{C^+(p)}$ with their projections modulo $p^n$, we have
		\begin{align*}
			G(p^n) = \widetilde{G(p)} \cdot H, \qquad G'(p^n) = \widetilde{G'(p)} \cdot H, \qquad C^+(p^n) = \widetilde{C^+(p)} \cdot H.
		\end{align*}
		Notice that $\gamma^{-1}H\gamma = H$: indeed, given $I+pA \in H$ we have $\gamma^{-1} (I+pA) \gamma = I + p\gamma^{-1}A\gamma \in H$, as $\gamma \pmod {p^{n-1}} \in C^+(p^{n-1})$. Therefore, we have
		\begin{equation*}
			\gamma^{-1} G(p^n) \gamma = \gamma^{-1} \widetilde{G(p)} \gamma \cdot \gamma^{-1}H\gamma = \widetilde{G'(p)} \cdot H = G'(p^n)
		\end{equation*}
		as desired. Finally, it is easy to check that
		\begin{equation*}
			[C^+(p^n) : G(p^n)] = [C^+(p^n) : G'(p^n)] = [C^+(p) : G(p)].
		\end{equation*}
		\textbf{5.} If there exists $n$ such that $\dim \mathfrak{g}_{n-1} = 4$, and $n$ is minimal, then we proved that $G(p^n) \subseteq C^+(p^n)$ and by Proposition \ref{prop:cartanzywina}(3) we have that $G \supseteq I + p^nM_{2 \times 2}(\Z_p)$.
		If instead $\dim \mathfrak{g}_n =2$ for every $n$, we notice that the element $\gamma$ at the end of step 3 can be taken to be $\gamma \equiv I \pmod {p^{n-1}}$: indeed, since $\widetilde{G(p)} \pmod {p^{n-1}}$ is contained in $C^+(p^{n-1})$, we can take $\delta \in \widetilde{G(p)}$ such that $\delta\gamma \equiv I \pmod {p^{n-1}}$, and hence $\gamma^{-1} \widetilde{G(p)} \gamma = (\delta\gamma)^{-1} \widetilde{G(p)} \delta\gamma$.
		In particular, this means that for every $n$ we can choose $\gamma_n \equiv I \pmod {p^{n-1}}$ such that $\gamma_n^{-1} G \gamma_n \pmod {p^n}$ is contained in $C^+(p^n)$. If we take the product $\gamma := \cdots \gamma_2 \gamma_1$, this converges (as $\gamma_n$ tends to $I$) and we have $\gamma^{-1} G \gamma \subseteq C^+$.
	\end{proof}

	\section{Local properties}

	In this section, we investigate some properties of the reduction modulo a prime $\mathfrak{p} \mid p$ of elliptic curves $E$ such that $\operatorname{Im}\rho_{E,p}$ is contained in the normaliser of a non-split Cartan subgroup.
	
	\begin{proposition}\label{prop:goodreductionforcartan}
		Let $E$ be an elliptic curve defined over a number field $K$, $n$ a positive integer, and $p$ an odd prime such that $\operatorname{Im}\rho_{E,p^n} \subseteq C_{ns}^+(p^n)$ up to conjugation.
		For any prime $\lambda \subseteq \OK$ that does not divide $p$ and such that $N_{K/\Q}(\lambda) \not\equiv \pm 1 \pmod {p^n}$, the elliptic curve $E$ has potentially good reduction at $\lambda$. Moreover, given $\mathfrak{p} \mid p$ in $K$, if $p^{n-1}(p-1) \nmid 2e(\mathfrak{p}|p)$, then the elliptic curve $E$ has potentially good reduction at $\mathfrak{p}$.
	\end{proposition}
	
	\begin{proof}
		The proof follows and generalises those of \cite[Proposition 3.3]{lemos19split} and \cite[Proposition 2.2]{lemos19borel}. We can assume that $E$ does not have CM, as CM curves have potentially good reduction everywhere. Let $\lambda \mid \ell$ be a prime of potentially multiplicative reduction, let $E_{K_\lambda}$ be the base change of $E$ to $K_\lambda$, and let $E_q$ be the Tate curve with parameter $q \in K_\lambda^\times$, isomorphic to $E$ over a quadratic extension of $K_\lambda$. Rename $E=E_{K_\lambda}$. There is a quadratic character $\psi$ such that $\rho_{E_q,p^n} \cong \rho_{E,p^n} \otimes \psi$, and up to conjugation we have
		\begin{equation}
			\rho_{E,p^n} \cong \psi \otimes \begin{pmatrix} \chi_{p^n} & \ast \\ 0 & 1 \end{pmatrix} = \begin{pmatrix} \psi\chi_{p^n} & \ast \\ 0 & \psi \end{pmatrix},
		\end{equation}
		where $\chi_{p^n}$ is the cyclotomic character modulo $p^n$. Consider an automorphism $\sigma \in \Gal\left(\faktor{\overline{K}_\lambda}{K_\lambda}\right)$, and set $A:= \rho_{E_q, p^n}(\sigma) = \begin{pmatrix} \chi_{p^n}(\sigma) & \ast \\ 0 & 1 \end{pmatrix}$. By our hypothesis on $\operatorname{Im}\rho_{E,p^n}$ there exists an element of $C_{ns}^+(p^n)$ conjugate to $A$ (up to changing sign by multiplying by $-I$). We now divide cases according to whether $\chi_{p^n}(\sigma) \equiv 1 \pmod p$ or $\chi_{p^n}(\sigma) \not\equiv 1 \pmod p$.
		\begin{enumerate}[(i)]
			\item Suppose first that $\chi_{p^n}(\sigma) \not\equiv 1 \pmod p$. We know that the roots of the characteristic polynomial of $A$ are $1$ and $\chi_{p^n}(\sigma)$. In particular, there exists an element $g$ in $C_{ns}^+(p^n)$ satisfying the polynomial equation $(g-1)(g - \chi_{p^n}(\sigma))=0$. If $g \in C_{ns}(p^n)$, then $\overline{g} = g \pmod p$ must be a scalar matrix, because non-scalar matrices in $C_{ns}(p)$ have eigenvalues in $\F_{p^2} \setminus \F_p$. As $1$ is an eigenvalue of $\overline{g}$, the matrix $\overline{g}$ is equal to the identity, contradicting the fact that $\chi_{p^n}(\sigma) \not\equiv 1 \pmod p$. This implies that $g \in C_{ns}^+(p^n) \setminus C_{ns}(p^n)$, and in particular, $\operatorname{tr} g = 0$, which implies that $\chi_{p^n}(\sigma) = -1$.
			\item If instead $\chi_{p^n}(\sigma) \equiv 1 \pmod p$, then we can write $A = I + p^r\begin{pmatrix} \ast & \ast \\ 0 & 0 \end{pmatrix}$, with $r \le n$ as large as possible. As for matrices in $M_{2 \times 2}(\F_p)$ the rank is invariant under conjugation, if $r<n$ there would be a matrix in $\mathfrak{g}_r$ of rank $1$ (with $\mathfrak{g}_r$ defined as in Definition \ref{def:liealgebra} for the group $G=C_{ns}^+$), which is impossible by Remark \ref{rmk:cartanliealg}. This implies that $A = I$, and so that $\chi_{p^n}(\sigma) \equiv 1 \pmod {p^n}$.
		\end{enumerate}
		We have then proved that $\chi_{p^n}(\sigma) \in \{ \pm 1 \}$ for every $\sigma$.
		Suppose first that $p \ne \ell$. If $\operatorname{Frob}_\lambda$ is a Frobenius element in $\Gal\left(\faktor{\overline{K}}{K}\right)$ with respect to $\lambda$, we have that $N_{K/\Q}(\lambda) = \chi_{p^n}(\operatorname{Frob}_\lambda) \equiv \pm 1 \pmod {p^n}$.
		If instead $\ell = p$, as the character $\chi_{p^n}$ is surjective from $\Gal\left(\faktor{\Q_p(\zeta_{p^n})}{\Q_p}\right)$ to $\left(\faktor{\Z}{p^n\Z}\right)^\times$, we must have $[K_\mathfrak{p}(\zeta_{p^n}) : K_\mathfrak{p}] \le 2$. However, this implies that $e(\mathfrak{p}|p)$ is a multiple of $\frac{\varphi(p^n)}{2}$, because $\faktor{\Q_p(\zeta_{p^n})}{\Q_p}$ is totally ramified.
	\end{proof}
	
	\begin{corollary}\label{cor:goodreductionforcartanQ}
		Let $E$ be an elliptic curve defined over $\Q$, $n$ a positive integer, and $p$ an odd prime such that $p^n \ne 3$ and $\operatorname{Im}\rho_{E,p^n} \subseteq C_{ns}^+(p^n)$ up to conjugation. For any prime $\ell \not\equiv \pm 1 \pmod {p^n}$ the elliptic curve $E$ has potentially good reduction at $\ell$.
	\end{corollary}
	
	The following lemma is a known fact which generalises \cite[Section 1, Proposition 1]{serre72}.
	
	\begin{lemma}\label{lemma:ordinarydecomposition}
		Let $E$ be an elliptic curve over a $p$-adic field $K$ and let $\mathfrak{p}$ be the prime of $K$ above $p$. Suppose that $E$ has good ordinary reduction at $\mathfrak{p}$. For every positive integer $n$, the inertia group $I_K$ of $K$ acts on $E[p^n]$ as $\begin{pmatrix} \chi_{p^n} & \ast \\ 0 & 1 \end{pmatrix}$, where $\chi_{p^n}$ is the cyclotomic character modulo $p^n$.
	\end{lemma}
	
	\begin{proof}
		Let $\mu$ be the $p$-adic cyclotomic character. Since $E$ has ordinary reduction, we know that there is an exact sequence of $\Z[I_K]$-modules
		\begin{equation*}
			\begin{tikzcd}
				0 \arrow[r] \arrow[r] & T_p(\mu) \arrow[r] & T_pE \arrow[r] & T_p\widetilde{E} \arrow[r] & 0,
			\end{tikzcd}
		\end{equation*}
		where $T_p\widetilde{E} \cong \Z_p$ has trivial action by $I_K$. Indeed, $I_K$ acts trivially on $T_p\widetilde{E}$, and since the determinant is cyclotomic, the other character must be the cyclotomic character.
		In particular, modulo $p^n$ we have	
		\begin{equation*}
			\begin{tikzcd}
				0 \arrow[r] \arrow[r] & \mu_{p^n} \arrow[r] & E[p^n] \arrow[r] & \faktor{\Z}{p^n\Z} \arrow[r] & 0,
			\end{tikzcd}
		\end{equation*}
		and the group $I_K$ acts on $E[p^n]$ as $\begin{pmatrix} \chi_{p^n} & \ast \\ 0 & 1 \end{pmatrix}$, where $\chi_{p^n}$ is the cyclotomic character modulo $p^n$.
	\end{proof}
	
	Given an elliptic curve $E$ over a $p$-adic field $K$, define the \emph{canonical subgroup of order $p$} of $E$ as in \cite[Definition 3.5]{furiolombardo23}. Notice that this definition can be extended to elliptic curves defined over number fields (see \cite[Definition 3.8]{furiolombardo23}).
	
	\begin{lemma}\label{lemma:locallemmaSamuel}
		Let $p$ be an odd prime and let $E$ be an elliptic curve over a $p$-adic field $K$. Let $\mathfrak{p} \subseteq K$ be the prime above $p$ with ramification index $e:=e(\mathfrak{p}|p)$ and let $I_K$ be the inertia group of $K$. Suppose that $E$ has good reduction at $\mathfrak{p}$.
		\begin{enumerate}
			\item If $E$ has ordinary reduction, $E$ admits a canonical subgroup and for every positive integer $n$ then the group $\rho_{E,p^n}(I_K)$ contains an element of order $\frac{p^n-p^{n-1}}{\gcd(p^n-p^{n-1}, e)}$ when projected in $\operatorname{PGL}_2\left(\faktor{\Z}{p^n\Z}\right)$.
			\item If $E$ has supersingular reduction and does not have a canonical subgroup, then the group $\rho_{E,p^n}(I_K)$ contains an element of order $\frac{p^{n+1}-p^{n-1}}{\gcd(p^{n+1}-p^{n-1}, e)}$.
		\end{enumerate}
	\end{lemma}
	
	\begin{proof}
		If $E$ has potentially good ordinary reduction, by Lemma \ref{lemma:ordinarydecomposition} the image of $I_K$ in $\operatorname{PGL}_2$ injects in the subgroup $\footnotesize \begin{pmatrix} \ast & \ast \\ 0 & 1 \end{pmatrix}$ and contains a subgroup isomorphic to $\chi_{p^n}(I_K)$. Since $[I_{\Q_p} : I_K] = e$ and $|\chi_{p^n}(I_{\Q_p})| = p^n-p^{n-1}$, the order of $\chi_{p^n}(I_K)$ must be divisible by $\frac{p^n-p^{n-1}}{\gcd(p^n-p^{n-1}, e)}$, and noting that the image of $\chi_{p^n}$ is cyclic we obtain the desired property.
		Moreover, in this case $E$ always has a canonical subgroup, which is the kernel of the reduction modulo $\mathfrak{p}$.
		Assume now that $E$ has potentially good supersingular reduction and that $E$ does not have a canonical subgroup. By \cite[Theorem 4.6]{smith23} (which is stated over number fields, but proved over $p$-adic fields) we know that $K(E[p]) \subseteq K(E[p^n])$ contains elements with valuation $\frac{1}{p^{2} - 1}$, and hence the ramification degree of $K(E[p^n])$ over $K$ is divisible by $$\frac{p^{2}-1}{\gcd(p^{2}-1, [K: \ \Q_p^{nr} \cap K])} = \frac{p^{2}-1}{\gcd(p^{2}-1,e)}.$$ 
		Moreover, since tame extensions of $K^{nr}$ are cyclic, there must be an element in the inertia subgroup $I(K(E[p^n])/K) \cong \rho_{E,p^n}(I_K)$ of order $\frac{p^2-1}{\gcd(e,p^2-1)}$. If $n>1$, since $\det \circ \rho_{E_L,p^\infty} (I_K) = (\Z_p^\times)^{e}$, there is also an element of $\rho_{E,p^n}(I_K)$ with determinant of order $\frac{\varphi(p^n)}{(\varphi(p^n), e)}$, and so $\rho_{E,p^n}(I_K)$ contains an element of order $\frac{p^n-p^{n-1}}{(p^n-p^{n-1}, e)}$. In particular, we have an element in $\rho_{E,p^n}(I_K)$ whose order is the less common multiple of $\frac{p^2-1}{\gcd(e,p^2-1)}$ and $\frac{p^n-p^{n-1}}{(p^n-p^{n-1}, e)}$, and so with order $\frac{p^{n+1}-p^{n-1}}{\gcd(e,p^{n+1}-p^{n-1})}$.
	\end{proof}
	
	The following theorem is a generalisation of \cite[Theorem 3.11]{furiolombardo23}.
	
	\begin{theorem}\label{thm:canonicalsbg}
		Let $p$ be a prime and let $E$ be an elliptic curve over a $p$-adic field $K$. Let $\mathfrak{p} \subseteq K$ be the prime above $p$ with ramification index $e:=e(\mathfrak{p}|p)$. Suppose that $E$ has potentially good reduction at $\mathfrak{p}$ and let $L$ be the minimal extension of $K^{nr}$ over which $E$ acquires good reduction, with degree $d=[L:K^{nr}]$. Suppose also that $\operatorname{Im}\rho_{E,p} \subseteq C_{ns}^+(p)$.
		\begin{enumerate}
			\item If $p>de+1$ and $p \ne 2de+1$, then $E$ does not have a canonical subgroup of order $p$.
			\item If $E$ has potentially good supersingular reduction modulo $\mathfrak{p}$ and $p \ge \max\{de-1,3\}$, then $E$ does not have a canonical subgroup of order $p$.
		\end{enumerate}
	\end{theorem}
	
	\begin{proof}
		We start by proving part 1.
		Consider the subgroup $I < \operatorname{Im}\rho_{E,p}$ obtained as the image of the inertia group of $L$.
		If $E$ has potentially good supersingular reduction, part 2 of the theorem supersedes part 1, hence we may assume that $E$ has potentially good ordinary reduction.
		By Lemma \ref{lemma:locallemmaSamuel} we know that the image of $I$ in $\operatorname{PGL}_2(\F_p)$ contains an element of order $\frac{p-1}{\gcd(de, p-1)}$.
		Since the square of any element of $C_{ns}^+(p) \setminus C_{ns}(p)$ is a scalar matrix and hence has order $2$ in $\operatorname{PGL}_2(\F_p)$, and every element in $C_{ns}(p)$ has order $p+1$ in $\operatorname{PGL}_2(\F_p)$, we have that $\frac{p-1}{(de, p-1)} \mid p+1$, and so $p-1 \mid 2de$. However, this is impossible because $p-1 \ne 2de$ and $p-1 > de$. \\
		Assume now that $E$ has potentially good supersingular reduction and suppose that $\faktor{E}{L}$ admits a canonical subgroup. By \cite[Theorem 3.10]{furiolombardo23} (which is stated over number fields, but its proof holds over $p$-adic fields) we know that its Hasse invariant $A$ is a number in $L$ with valuation $0 < v_p(A) < \frac{p}{p+1}$. However, $v_p(A)$ is a rational number with denominator dividing $de$. Suppose that $v_p(v_p(A)) > 0$: we can write $v_p(A)=\frac{\alpha p}{de}$, for some positive integer $\alpha$. On the other hand, we have $\frac{\alpha p}{de} < \frac{p}{p+1}$, which gives $1 \le \alpha < \frac{de}{p+1}$, contradicting the hypothesis that $p \ge de-1$. This implies that $v_p(v_p(A)) = 0$. Let $c$ be the coefficient of $x^\frac{p^2-p}{2}$ in the division polynomial $\Psi_p(x)$ and let $\mu$ be its valuation. By \cite[Theorem 1]{debry14} we know that $c \equiv A \pmod p$, and so $v_p(A) = \mu$. By \cite[Theorem 4.6]{smith23} (which is stated over number fields, but its proof holds over $p$-adic fields) we know that $L(E[p])$ contains elements of valuation $\frac{\mu}{p^2-p}$. However, $v_p(\mu) = 0$, and so $p$ must divide the degree $[L(E[p]) : L]$, which is a divisor of $2(p^2-1)$, giving a contradiction.
	\end{proof}
	
	\begin{corollary}\label{cor:canonicalsbgeasy}
		Let $E$ be an elliptic curve over a number field $K$ and let $p$ be a prime. Let $\mathfrak{p} \subseteq K$ be a prime above $p$ with ramification index $e:=e(\mathfrak{p}|p)$. Suppose that $\operatorname{Im}\rho_{E,p} \subseteq C_{ns}^+(p)$. 
		\begin{enumerate}
			\item If $p>6e+1$ and $p \ne 8e+1, 12e+1$, then $E$ does not have a canonical subgroup of order $p$.
			\item If $E$ has potentially good supersingular reduction modulo $\mathfrak{p}$ and $p \ge 6e-1$, then $E$ does not have a canonical subgroup of order $p$.
		\end{enumerate}
	\end{corollary}
	
	\begin{proof}
		We notice that $p \ge 6e-1$, and since $6e-1 > 2e+1$ we have $p-1 \nmid 2e$. By Proposition \ref{prop:goodreductionforcartan} this implies that $E$ has potentially good reduction at $\mathfrak{p}$. Moreover, using that $p \ge 6e-1 \ge 5$, by \cite[Proposition 1]{kraus90} we see that the degree $d=[L:K_{\mathfrak{p}}^{nr}]$ of the minimal extension of $K_{\mathfrak{p}}^{nr}$ over which $E$ acquires good reduction is at most $6$. We then conclude by applying Theorem \ref{thm:canonicalsbg}.
	\end{proof}
	
	\begin{corollary}\label{cor:canonicalsbgQ}
		Let $E$ be an elliptic curve over $\Q$. If $p$ is a prime such that $p > 7$ and $p \ne 13$, and $\operatorname{Im}\rho_{E,p} \subseteq C_{ns}^+(p)$, then $E$ does not have a canonical subgroup of order $p$.
	\end{corollary}
	
	\begin{corollary}\label{cor:supersingular}
		Let $E$ be an elliptic curve over a number field $K$ and let $p$ be a prime. Let $\mathfrak{p} \subseteq K$ be a prime above $p$ with ramification index $e:=e(\mathfrak{p}|p)$. Suppose that $\operatorname{Im}\rho_{E,p} \subseteq C_{ns}^+(p)$.
		If $p > 6e+1$ and $p \ne 8e+1, 12e+1$, then $E$ has potentially good supersingular reduction modulo $\mathfrak{p}$.
	\end{corollary}
	
	\begin{proof}
		We notice that by Proposition \ref{prop:goodreductionforcartan} the curve $E$ has potentially good reduction at $\mathfrak{p}$. It then suffices to combine Corollary \ref{cor:canonicalsbgeasy} and \cite[Theorem 3.10]{furiolombardo23}, using the fact that if $A$ is the Hasse invariant of $E$, then $v_p(A)$ is equal to $0$ if and only if $E$ has ordinary reduction modulo $p$.
	\end{proof}
	
	The corollary above generalises \cite[Proposition 3.1]{ejder22}, written below, to arbitrary number fields.
	
	\begin{corollary}\label{cor:supersingularQ}
		Let $E$ be an elliptic curve over $\Q$. If $p$ is a prime such that $p>7$, $p \ne 13$ and $\operatorname{Im}\rho_{E,p} \subseteq C_{ns}^+(p)$, then $E$ has potentially good supersingular reduction modulo $p$.
	\end{corollary}

	\section{Effective surjectivity theorem}\label{sec:effiso}

	In this section, we give a generalised version of the effective surjectivity theorem of Le Fourn \cite[Theorem 5.2]{lefourn16}, obtaining a bound on the product of the prime powers $p^n$ for which the image of the representation $\rho_{E,p^n}$ is contained either in a Borel subgroup or in the normaliser of a Cartan subgroup of $\GL(E[p^n])$. There are two main differences with respect to \cite[Theorem 5.2]{lefourn16}: the first is that we are able to bound the product of prime powers and not just the product of primes, the second is that in the non-split Cartan case our version also applies non-trivially to curves of small height (similarly to \cite[Theorem 4.1]{furiolombardo23}).
	
	\begin{theorem}\label{thm:effiso}\label{effiso}
		Let $E$ be an elliptic curve without CM defined over a number field $K$. We denote by $\Fheight(E)$ the stable Faltings height of $E$ (with the normalisation of \cite[Section 1.2]{deligne85}). Let $\mathcal{B}, \mathcal{C}_{sp}, \mathcal{C}_{ns}$ be sets of odd primes $p$ such that $\operatorname{Im}\rho_{E,p} \subseteq G(p)$ up to conjugacy for $G = B, C_{sp}^+, C_{ns}^+$ respectively. For every $p \in \mathcal{B} \cup \mathcal{C}_{sp} \cup \mathcal{C}_{ns}$, let $n_p$ be the largest positive integer such that $\operatorname{Im}\rho_{E,p^{n_p}} \subseteq G(p^n)$, and let $\Lambda:= \prod_{p \in \mathcal{B}} p^{\frac{n_p}{2}} \prod_{p \in \mathcal{C}} p^{n_p}$, where $\mathcal{C} := \mathcal{C}_{sp} \cup \mathcal{C}_{ns}$.
		\begin{enumerate}
			\item We have
			$$\Lambda < 1454 \cdot 2^{|\mathcal{C}|} [K:\Q] \left( \Fheight(E) + \frac{7}{2}\log(\Fheight(E) + 2.72) + 4\log\Lambda + 5\right).$$
			\item If $\mathcal{B} = \mathcal{C}_{sp} = \emptyset$ we have
			\begin{equation*}
				\Lambda < 1454 \cdot 2^{|\mathcal{C}|} [K:\Q] \left( \Fheight(E) + \frac{3}{2}\log(\Fheight(E) + 2.72) + 2\log\Lambda + 2.6 \right).
			\end{equation*}
			\item If $K=\Q$ and $\mathcal{B} = \mathcal{C}_{sp} = \emptyset$, we have
			$$\Lambda < 1454 \cdot 2^{|\mathcal{C}|} \left( \Fheight(E) + 2\log\Lambda + \frac{3}{2}\max\{0, \log(\Im\{\tau\})\} + 1.38 \right),$$
			where $\tau$ is the point in the standard fundamental domain $\mathcal{F}$ of $\uhp$ such that $E(\C) \cong \faktor{\C}{\Z \oplus \tau\Z}$.
		\end{enumerate}
		Furthermore, if $\tau_{\sigma} \in \mathcal{F}$ corresponds to the curve $\sigma(E)$, for some $\sigma: K \hookrightarrow \C$, and if we assume that $\Im\{\tau_\sigma\} \ge \frac{15}{\pi}$ for every $\sigma$, we can replace the number $1454$ with $1266.4$ in all the inequalities.
	\end{theorem}
	
	To prove Theorem \ref{thm:effiso}, we follow closely the approach of \cite[Theorem 5.2]{lefourn16} and \cite[Theorem 1.4]{gaudron-remond} like in \cite{furiolombardo23}. We will implement and generalise the improvements introduced in \cite[Section 4]{furiolombardo23} to obtain sharper bounds.
	
	From Theorem \ref{thm:effiso}, we then deduce the following easier inequality.
	
	\begin{theorem}\label{thm:totaleffiso}
		Let $\faktor{E}{\Q}$ be an elliptic curve without CM, let $\mathcal{C}$ be the set of all primes $p>2$ such that $\operatorname{Im}\rho_{E,p} \subseteq C_{ns}^+(p)$ up to conjugation, and let $\Lambda$ be as in Theorem \ref{thm:effiso}. We have
		\begin{equation*}
			\Lambda < 21000 \left( \Fheight(E) + 40 \right)^{1.308}.
		\end{equation*}
		Moreover, if we define
		\begin{equation*}
			\delta(x) := \frac{1}{\log(\log(x+40) + 7.6) - 0.903}
		\end{equation*}
		for every $x>-0.75$, we have
		\begin{align*}
			\Lambda &< 14400 \cdot (\Fheight(E)+40)^{0.907 \cdot \delta(\Fheight(E))} \left(\Fheight(E) + 22.5 \right).
		\end{align*}
	\end{theorem}
	
	We divide the rest of this section in two subsections: in the first, we will prove Theorem \ref{thm:effiso}; in the second, we will give some better estimates in the case where the $j$-invariant of $E$ is not an integer, and then we will exploit them to prove Theorem \ref{thm:totaleffiso}.

	\subsection{Proof of Theorem \ref{thm:effiso}}

	This section is really similar to \cite[Section 4]{furiolombardo23}, however we will repeat all the steps for convenience of the reader. We first prove Theorem \ref{thm:effiso}(1), and then we will focus on the cases $\mathcal{B} = \mathcal{C}_{sp} = \emptyset$ and $K=\Q$.
	
	We begin by recalling some crucial definitions from \cite{gaudron-remond}.
	\begin{definition}\label{def: x}
		Let $A$ be a complex abelian variety, let $B \subset A$ be an abelian subvariety of codimension $t \ge 1$, and let $L$ be a polarisation on $A$. We define
		\begin{align*}
			x(B) := \left(\frac{\deg_LB}{\deg_LA}\right)^\frac{1}{t} \qquad \text{and} \qquad x := \min_{B \subsetneq A} x(B),
		\end{align*}
		where $\deg_LA$ is the top self-intersection number of the line bundle $L$ on $A$, and similarly $\deg_L B$.
	\end{definition}
	Let $(A,L)$ be a polarised abelian variety defined over a number field $K$. Fix an embedding $\sigma : K \hookrightarrow \C$ and let $(A_\sigma,L_\sigma)$ be the base-change of $(A,L)$ to $\C$ via $\sigma$. We will denote by $B[\sigma]$ a proper abelian subvariety of $A_\sigma$ such that $x(B[\sigma])=x$.
	\begin{definition}
		Let $A$ be a complex abelian variety and let $L$ be a polarisation on $A$. Let $\| \cdot \|_L$ be the norm induced by $L$ on the tangent space $t_A$, and let $\Omega_A$ be the period lattice. We define
		$$\rho(A,L) := \min\{\|\omega\|_L \mid \omega \in \Omega_A \setminus \{0\}\}.$$
	\end{definition}
	\begin{remark}\label{rmk: princ pol on ell curves}
		Let $E$ be an elliptic curve defined over a number field $K$ and let $L$ be its canonical principal polarisation. As explained in \cite[Remark 3.3]{gaudron-remond}, given an embedding $\sigma : K \hookrightarrow \C$ we have $\rho(E_\sigma,L_\sigma)^{-2} = \Im\{\tau_\sigma\}$, where $\tau_\sigma$ is the element in the standard fundamental domain $\mathcal{F}$ that corresponds to $E_\sigma$ and $L_\sigma$ is the base-change of the polarisation $L$ via $\sigma$.
	\end{remark}
	\begin{definition}\label{def: delta sigma}
		Let $A$ be an abelian variety defined over a number field $K$ and let $\sigma: K \hookrightarrow \C$ be an embedding. Let $L$ be a polarisation on $A$ and let $\operatorname{d}_\sigma$ be the distance induced by $L_\sigma$ on $t_{A_\sigma}$. We define
		$$\delta_\sigma = \min\{\operatorname{d}_\sigma(\omega,t_{B[\sigma]}) \mid \omega \in \Omega_{A_\sigma} \setminus t_{B[\sigma]} \},$$
		where $B[\sigma]$ is as in Definition \ref{def: x}.
	\end{definition}

	We now begin the proof of Theorem \ref{effiso} introducing the general setting; then we will split the proof in different parts, distinguishing the case $\mathcal{B}, \mathcal{C}_{sp} \ne \emptyset$, the case $\mathcal{B}, \mathcal{C}_{sp} = \emptyset$ and $K=\Q$, and the case $\mathcal{B}, \mathcal{C}_{sp} = \emptyset$ and $K \ne \Q$. Similarly to \cite{lefourn16}, we start by giving the construction of a particular quotient of the abelian surface $E \times E$. However, we will use a slightly different quotient, which is more natural.

	Choose an extension $\faktor{K'}{K}$ of degree $2^{|\mathcal{C}|}$ such that for every prime $p \in \mathcal{C}$ we have
	\begin{equation}\label{eq:kappaprimo}
		\rho_{E,p}\left(\Gal\left(\faktor{\overline{K}}{K'}\right)\right) \subseteq C(p)
	\end{equation}
	up to conjugation, where $C(p)$ is either $C_{sp}(p)$ or $C_{ns}(p)$. Note that, if $\rho_{E,p}\left(\Gal\left(\faktor{\overline{K}}{K}\right)\right)$ is already contained in $C(p)$, we choose $K'$ to be an arbitrary complex quadratic extension of $K$. Since the image of the complex conjugation is not contained in $C(p)$ for every $p$, the field $K'$ is always a complex field.
	We now construct a subgroup $G_p$ of $E[p^{n_p}]^2$ for every $p \in \mathcal{B} \cup \mathcal{C}$.
	\begin{itemize}
		\item \underline{$\mathbf{p \in \mathcal{B}}$}: We define the group $G_p$ as $\Gamma_{p^{n_p}} \times E[p^{n_p}] \subseteq E \times E$, where $\Gamma_{p^{n_p}}$ is a cyclic subgroup of order $p^{n_p}$ fixed by $\rho_{E,p^{n_p}}$. We have $|G_p| = p^{3n_p}$.
		\item \underline{$\mathbf{p \in \mathcal{C}_{sp}}$}: We define the group $G_p$ as $\Gamma_1 \times \Gamma_2$, where $\Gamma_1, \Gamma_2 \subset E[p^{n_p}]$ are two independent cyclic subgroups of order $p^{n_p}$ stabilised by $\rho_{E,p^{n_p}}$ over $K'$. We have $|G_p| = p^{2n_p}$.
		\item \underline{$\mathbf{p \in \mathcal{C}_{ns}}$}: Choose an element $g_p \in C_{ns}(p^{n_p})$ such that $g_p \pmod p \notin \F_p^\times \cdot \operatorname{Id}$. We define the subgroup $G_p$ as $\{(x,g_p \cdot x) \mid x \in E[p^{n_p}]\}$. We have $|G_p| = p^{2n_p}$.
	\end{itemize}
	
	It is not difficult to notice that all the groups $G_p$ we defined are stable under the action of the absolute Galois group of $K'$: indeed, this is clear by definition in the case of the Borel and split Cartan subgroups, and it is true in the case of the non-split Cartan as $C_{ns}(p^{n_p})$ is abelian and for every $\gamma \in C_{ns}(p^{n_p})$ we have $\gamma(x,g_px) = (\gamma x, \gamma g_p x) = (\gamma x, g_p (\gamma x))$.
	We now consider the group
	$$G:= \bigoplus_{p \in \mathcal{B} \cup \mathcal{C}} G_p \subset E \times E.$$
	
	Define
	$$\Lambda_\mathcal{B} := \prod_{p \in \mathcal{B}} p^{n_p} \qquad \text{and} \qquad \Lambda_\mathcal{C} := \prod_{p \in \mathcal{C}} p^{n_p}.$$
	By taking the quotient $A$ of $E \times E$ by the subgroup $G$, we have an isogeny $\varphi: E\times E \to A$ defined over $K'$ such that $\deg \varphi = \Lambda_{\mathcal{B}}^3 \Lambda_{\mathcal{C}}^2$. There exists $\psi: A \to E \times E$ such that $\psi \circ \varphi = [\Lambda_{\mathcal{B}} \Lambda_{\mathcal{C}}]_{E \times E}$, so $\deg\psi= \Lambda_{\mathcal{B}} \Lambda_{\mathcal{C}}^2 = \Lambda^2$. As explained in the proof of \cite[Proposition 5.1]{lefourn16}, for every embedding $\sigma : K' \hookrightarrow \C$, there is a canonical norm $\| \cdot \|_\sigma$ on the tangent space of $E_\sigma$, which contains the period lattice $\Omega_{E, \sigma}$. As in \cite[Part 7.3]{gaudron-remond} and in the proof of \cite[Proposition 5.1]{lefourn16}, we choose an embedding $\sigma_0$ such that there exists a basis $(\omega_0,\tau_{\sigma_0}\omega_0)$ of $\Omega_{E,\sigma_0}$ for which $\tau_{\sigma_0}$ is as in Remark \ref{rmk: princ pol on ell curves} and
	$$\| \omega_0 \|_{\sigma_0} = \max_\sigma \min_{\omega \in \Omega_{E,\sigma} \setminus \{0\}} \| \omega \|_\sigma.$$
	By Remark \ref{rmk: princ pol on ell curves}, this choice of $\sigma_0$ minimizes $\Im\{\tau_{\sigma}\}$ among all $\sigma$, as in \cite[Part 7.3]{gaudron-remond}. Let $\Omega_{A,\sigma_0}$ be the period lattice of $A_{\sigma_0}$. We want to show that there exists an element $\chi \in \Omega_{A,\sigma_0}$ such that $\mathrm{d}\psi(\chi) = (\omega_0,\tau_{\sigma_0}\omega_0)$. To do this, we prove the following lemma.
	
	\begin{lemma}\label{lemma:hyp*}
		For every embedding $\sigma : K' \to \C$, if $\Omega_{E,\sigma}$ and $\Omega_{A,\sigma}$ are the period lattices of $E$ and $A$ with respect to $\sigma$, then $\operatorname{d}\psi(\Omega_{A,\sigma}) \subseteq \Omega_{E,\sigma}^2$ contains an element $(\omega_1, \omega_2)$ such that $\langle \omega_1, \omega_2 \rangle_\Z = \Omega_{E,\sigma}$.
	\end{lemma}
	
	\begin{proof}
		The proof is similar to the second part of the proof of \cite[Theorem 5.2]{lefourn16}, however \cite[Lemma 5.3]{lefourn16} can no longer be applied. 
		Let $t_E \times t_E$ be the tangent space of $E \times E$ with respect to the embedding $\sigma$, and let $\pi : t_E \times t_E \to E \times E$ be the projection. The lattice $\Omega := \pi^{-1}(G) \subset t_E \times t_E$ defines a quotient abelian variety $\faktor{t_E \times t_E}{\Omega}$ isomorphic to $A$.
		\[\begin{tikzcd}
			{t_E \times t_E} & {t_E \times t_E} & {t_E \times t_E} \\
			{E \times E} & A & {E \times E}
			\arrow["{\operatorname{id}}", from=1-1, to=1-2]
			\arrow["\pi"', from=1-1, to=2-1]
			\arrow["{\Lambda_\mathcal{B} \Lambda_\mathcal{C}}", from=1-2, to=1-3]
			\arrow[from=1-2, to=2-2]
			\arrow["\pi", from=1-3, to=2-3]
			\arrow["\varphi"', from=2-1, to=2-2]
			\arrow["\psi"', from=2-2, to=2-3]
		\end{tikzcd}\]
		Let $\Omega' := \Lambda_{\mathcal{B}} \Lambda_{\mathcal{C}} \Omega \subseteq \Omega_E \times \Omega_E$ be the image of the lattice $\Omega$ under the homothety $\Lambda_{\mathcal{B}} \Lambda_{\mathcal{C}}$. This is equal to the image of $\Omega_A$ via $\operatorname{d}\psi$, i.e. $\Omega' = \operatorname{d}\psi(\Omega_A)$. We want to show that $\Omega'$ contains a basis of $\Omega_E$. Fix a basis $(\widetilde{e_1}, \widetilde{e_2})$ of $\Omega_E$. Let $p$ be a prime in $\mathcal{B} \cup \mathcal{C}$ and consider the image of $\Omega'$ in $\left(\faktor{\Omega_E}{p^{n_p}\Omega_E}\right)^2$. Multiplying it by $\frac{1}{p^{n_p}}$ we can identify it with a subgroup of $\left(\faktor{\Omega_E/p^{n_p}}{\Omega_E}\right)^2 = E[p^{n_p}]^2$. By definition of $\Omega'$, the image of $\frac{1}{p^{n_p}}\Omega'$ in $E[p^{n_p}]^2$ is exactly $\frac{\Lambda_{\mathcal{B}} \Lambda_{\mathcal{C}}}{p^{n_p}} G_p = G_p$. Identify $E[p^{n_p}]$ with $\F_p^2$ choosing the basis $\pi\left(\frac{\widetilde{e_1}}{p^{n_p}}, \frac{\widetilde{e_2}}{p^{n_p}}\right) = (e_1, e_2)$. We now prove that for every $G_p$ there is an element $(x,y) \in G_p$ such that $\det_{e_1, e_2}(x,y) = 1$. \\
		If $p \in \mathcal{B}$, let $(a,b) \in \Gamma_{p^{n_p}}$ be an element of order $p^{n_p}$. We can choose $(c,d)$ such that $ad-bc =1 \pmod {p^{n_p}}$, and hence the element $((a,b),(c,d)) \in G_p$ has determinant $1$. \\
		If $p \in \mathcal{C}_{sp}$, similarly to the Borel case, given two elements $(a,b) \in \Gamma_1$ and $(c,d) \in \Gamma_2$ of order $p^{n_p}$, we have $ad-bc = k \not\equiv 0 \pmod p$, because $\Gamma_1 \pmod p \ne \Gamma_2 \pmod p$. We can then take $(a',b') = (k^{-1}a, k^{-1}b) \in \Gamma_1$ such that $a'd-b'c = 1$. \\
		If $p \in \mathcal{C}_{ns}$, let $e_1',e_2'$ be another basis of $E[p^n]$. We have $$\det_{e_1,e_2} (x,y) = \det_{e_1,e_2}(e_1',e_2') \cdot \det_{e_1',e_2'}(x,y),$$ hence it suffices to show that for a particular choice of a basis $e_1',e_2'$, the group $\det_{e_1',e_2'}G_p$ contains the whole $\left(\faktor{\Z}{p^n\Z}\right)^\times$. Fix $e_1', e_2'$ such that $g_p = \begin{pmatrix} a & \varepsilon b \\ b & a \end{pmatrix}$ with $p \nmid b$. We have
		$$\det(x,g_p \cdot x) = \det \begin{pmatrix} x_1 & ax_1 + \varepsilon bx_2 \\ x_2 & bx_1+ax_2 \end{pmatrix} = b(x_1^2-\varepsilon x_2^2) = b \cdot \det \begin{pmatrix} x_1 & \varepsilon x_2 \\ x_2 & x_1 \end{pmatrix}.$$
		As $\det C_{ns}(p^n) = \left(\faktor{\Z}{p^n\Z}\right)^\times$, the proof of the claim is obtained by varying $x$. \\
		We showed that for every $p \in \mathcal{B} \cup \mathcal{C}$ there exists $\gamma_p \in \SL_2\left(\faktor{\Z}{p^{n_p}}\right)$ such that $\gamma_p \begin{pmatrix} e_1 \\ e_2 \end{pmatrix} \in \faktor{\Omega'}{p^{n_p}\Omega_E^2}$. Since the projection $\SL_2(\Z) \to \prod_{p \in \mathcal{B} \cup \mathcal{C}} \SL_2\left(\faktor{\Z}{p^{n_p}\Z}\right)$ is surjective, there exists an element $\gamma \in \SL_2(\Z)$ such that
		\begin{equation*}
			\gamma \begin{pmatrix} \widetilde{e_1} \\ \widetilde{e_2} \end{pmatrix} \in \Omega' + \Lambda_{\mathcal{B}}\Lambda_{\mathcal{C}} \Omega_E^2 \subseteq \Omega'.
		\end{equation*}
		Since $(\widetilde{e_1}, \widetilde{e_2})$ is a basis of $\Omega_E$, also $\gamma (\widetilde{e_1}, \widetilde{e_2})$ is a basis, and it is contained in $\Omega'$ as desired.
	\end{proof}
	
	Using Lemma \ref{lemma:hyp*}, composing $\psi$ with an isomorphism of $E \times E$, we can assume that there exists an element $\chi \in \Omega_{A,\sigma_0}$ such that $\mathrm{d}\psi(\chi) = (\omega_0,\tau_{\sigma_0}\omega_0)$. Setting $\omega = (\omega_0, \tau_{\sigma_0}\omega_0, \chi) \in \Omega_{E \times E \times A , \sigma_0}$, we define 
	$A_\omega$ as the minimal abelian subvariety of $(E \times E \times A)_{\sigma_0}$ containing $\omega = (\omega_0, \tau_{\sigma_0}\omega_0, \chi)$ in its tangent space. As in the proof of \cite[Proposition 5.1]{lefourn16}, one shows that $$A_\omega := \{(\psi(z),z) \mid z \in A_{\sigma_0}\} \subset (E \times E \times A)_{\sigma_0}.$$
	Indeed, the inclusion $A_\omega \subseteq \{(\psi(z),z) \mid z \in A_{\sigma_0}\}$ is clear, and the projection from $A_\omega$ to $E \times E$ is a subvariety of $(E \times E)_{\sigma_0}$ containing $(\omega_0, \tau_{\sigma_0} \omega_0)$ in its period lattice. As $E$ is an elliptic curve without complex multiplication, the endomorphism ring of $E \times E$ is $M_{2 \times 2}(\Z)$, therefore no strict abelian subvariety of $(E \times E)_{\sigma_0}$ contains $(\omega_0, \tau_{\sigma_0} \omega_0)$ in its tangent space. This proves that the dimension of $A_\omega$ is at least $2$, hence the equality above.
	We then see that the complex abelian variety $A_\omega$ can be defined over $K'$. When we consider $A_{\omega}$ as being defined over $K'$, we will write $(A_\omega)_\sigma$ for its base-change to $\mathbb{C}$ along a given embedding $\sigma: K' \hookrightarrow \mathbb{C}$.  The abelian variety $A_\omega$ falls within the context of \cite[Part 7.3]{gaudron-remond}. 
	
	As explained in \cite[Proposition 5.1]{lefourn16}, one can repeat the proof of Gaudron and R\'emond to obtain a bound on $\Lambda$ similar to that of \cite[Theorem 5.2]{lefourn16}. However, we will change some details to improve the final result. 
	
	We choose a polarisation on $A_\omega$ as in \cite[Part 7.3]{gaudron-remond}, namely, in the following way. Set $n = \lfloor |\tau_{\sigma_0}|^2 \rfloor$, let $L_E$ be the canonical principal polarisation on $E_{\sigma_0}$ and let $\pi_1,\pi_2$ be the projections from $(E \times E)_{\sigma_0}$ on the two copies of $E_{\sigma_0}$. We consider the polarisation $L'=\pi_1^*L_E^{\otimes n} \otimes \pi_2^*L_E$ on $(E \times E)_{\sigma_0}$ and the isogeny $f$ defined as the composition $A_\omega \xrightarrow{\sim} A_{\sigma_0} \xrightarrow{\psi} (E \times E)_{\sigma_0}$, where the first isomorphism is given by the projection $(\psi(z),z) \mapsto z$. We define the polarisation $L := f^*L'$ on $A_\omega$, and as in \cite[Part 7.3]{gaudron-remond} we compute
	\begin{equation}\label{eq:degLAomega}
		\deg_L A_\omega = (\deg f)\deg_{L'} E^2 = 2n\Lambda^2.
	\end{equation}
	
	\begin{lemma}\label{slopes}
		Let $\hat{\mu}_{max}(\overline{t_{A_\omega}^\vee})$ be the quantity defined in \cite[Part 6.8]{gaudron-remond}. The inequality $$\hat{\mu}_{max}(\overline{t_{A_\omega}^\vee}) \le \Fheight(E) + 2\log \Lambda + \frac{1}{2}\log\frac{n}{\pi}$$ holds.
	\end{lemma}
	\begin{proof}
		It suffices to combine the proof of \cite[Lemma 7.6]{gaudron-remond} with the remark at the end of the proof of \cite[Proposition 5.1]{lefourn16}, which gives $\Fheight(A_\omega) \le 2\Fheight(E) + \log \Lambda$.
	\end{proof}
	
	The following definition collects the notations that will be needed in the rest of the proof.
	\begin{definition}\label{def: notation for the proof of the effective surjectivity theorem}
		Following \cite[Parts 6.2 and 6.3]{gaudron-remond}, we set
		$\varepsilon = \frac{3\sqrt{2}-4}{2}$, $\theta = \frac{\log2}{\pi}$ and $S_\sigma := \left\lfloor \frac{\theta\varepsilon}{x\delta_\sigma^2} \right\rfloor$, where $x$ is as in Definition \ref{def: x} and $\delta_\sigma$ is as in Definition \ref{def: delta sigma}. We further define $\mathcal{V} = \{\sigma:K' \hookrightarrow \C \mid S_\sigma \ge 1\}$.
		For every $\sigma: K' \hookrightarrow \C$ choose $B[\sigma] \subset (A_\omega)_\sigma$ as in Definition \ref{def: x}. We introduce the following quantities.
		\begin{align*}
			\aleph_1 := & 2\max\{0,\hat{\mu}_{max}(\overline{t_{A_\omega}^\vee})\} + 5\log 2 + \frac{2}{[K':\Q]} \sum_{\sigma \in \mathcal{V}} \log \max \left\lbrace 1, \frac{1}{\rho((A_\omega)_\sigma,L_\sigma)} \right\rbrace \\
			& + \frac{4}{[K':\Q]} \sum_{\sigma \in \mathcal{V}} \log\deg_{L_\sigma}B[\sigma] + \varepsilon\log12; \\
			m := & \ \frac{1}{[K':\Q]} \sum_{\sigma \in \mathcal{V}} \frac{1}{\delta_\sigma^2}.
		\end{align*}
	\end{definition}
	
	By \cite[Part 6]{gaudron-remond},	and in particular \cite[Part 6.8]{gaudron-remond}, we have
	$$\varepsilon\log2 \left( \frac{\varepsilon\theta}{x}m -1
	\right) < \frac{\varepsilon\log2}{[K':\Q]}\sum_{\sigma \in \mathcal{V}} S_\sigma \le \aleph_1 + \frac{\pi x}{2}\left(\frac{3}{2} + \frac{3\theta\varepsilon}{x}\sqrt{m} + \left(\frac{\theta\varepsilon}{x} \right)^2 m\right),$$
	where the inequality on the left is obtained by the definition of $S_\sigma$ together with the inequality $\lfloor x \rfloor > x-1$, while the inequality on the right is that of \cite[Part 6.8, equation (14)]{gaudron-remond} together with the estimate on $\aleph_2$ obtained by the Cauchy-Schwarz inequality on the same page of \cite{gaudron-remond}.
	Solving the inequality in $\sqrt{m}$, using that $\theta = \frac{\log2}{\pi}$, we obtain
	\begin{equation}\label{sqrtm-ineq}
		\sqrt{m} < \frac{3\pi x}{2\varepsilon\log2}\left(1 + \sqrt{1+ \frac{8}{9\pi x}\left(\aleph_1 + \frac{3\pi}{4}x + \varepsilon \log2\right)}\right).
	\end{equation}
	
	We now want to find a bound on $x$ in terms of $n$. Recall that $\omega = (\omega_0, \tau_{\sigma_0}\omega_0, \chi)$. We have
	\begin{align*}
		\| \omega \|_{L, \sigma_0}^2 &= \| (\omega_0, \tau_{\sigma_0} \omega_0) \|_{L',\sigma_0}^2 = n \|\omega_0 \|_{L_E, \sigma_0}^2 + \| \tau_{\sigma_0} \omega_0 \|_{L_E, \sigma_0}^2.
	\end{align*}
	Using Remark \ref{rmk: princ pol on ell curves} we obtain
	\begin{align*}
		\| \omega \|_{L, \sigma_0}^2 &= (n + |\tau_{\sigma_0}|^2) \| \omega_0 \|_{L_E, \sigma_0}^2 = (n+|\tau_{\sigma_0}|^2) \rho(E_{\sigma_0}, (L_E)_{\sigma_0})^2 \\
		&= \frac{n+|\tau_{\sigma_0}|^2}{\Im\{\tau_{\sigma_0}\}} \le \frac{n+|\tau_{\sigma_0}|^2}{\sqrt{|\tau_{\sigma_0}|^2- \frac{1}{4}}} \le \frac{2n}{\sqrt{n-\frac{1}{4}}},
	\end{align*}
	where the last inequality follows from the fact that the function $\frac{n+t}{\sqrt{t-1/4}}$ for $t \in \left[n, n+1\right]$ attains its maximum at $t=n=\lfloor |\tau_{\sigma_0}| \rfloor$.
	
	Suppose now that $\sigma_0 \notin \mathcal{V}$, and hence that $S_{\sigma_0} = 0$. We notice that $x \le x(0) = \frac{1}{\sqrt{\deg_L A_\omega}} = \frac{1}{\Lambda \sqrt{2n}}$ by equation \eqref{eq:degLAomega}.
	By definition of $S_{\sigma_0}$ we have
	$$1 > \frac{\theta\varepsilon}{x\delta_{\sigma_0}^2} \ge \frac{\theta\varepsilon \Lambda \sqrt{2n}}{\| \omega \|_{L,\sigma_0}^2} \ge \left( \frac{1}{2} - \frac{1}{8n} \right)^\frac{1}{2} \theta\varepsilon \Lambda > \sqrt{\frac{3}{8}} \theta\varepsilon \Lambda,$$
	that gives $\Lambda < \frac{\sqrt{8}}{\theta\varepsilon\sqrt{3}} < 62$, which is better than Theorem \ref{effiso} (since $\Fheight(E)> -0.75$ by Remark \ref{minimalheight}). Thus, we can assume $S_{\sigma_0} \ge 1$, and in particular, we can assume that $\sigma_0 \in \mathcal{V}$.
	
	Since $K'$ is a complex field, there exists an embedding $\overline{\sigma_0}$ of $K'$ different from $\sigma_0$, which is its complex conjugate, inducing the same norm $\| \cdot \|_{L_{\overline{\sigma_0}}} = \| \cdot \|_{L_{\sigma_0}}$ and such that $\delta_{\overline{\sigma_0}} = \delta_{\sigma_0}$.
	Combining the fact that $\delta_{\sigma_0}^2 \le \frac{2n}{\sqrt{n - \frac{1}{4}}}$ and that $\sigma_0, \overline{\sigma_0} \in \mathcal{V}$ we obtain
	\begin{equation}\label{eq:lowerboundonm}
		m = \frac{1}{[K':\Q]} \sum_{\sigma \in \mathcal{V}} \frac{1}{\delta_\sigma^2} \ge \frac{1}{[K':\Q]} \cdot \frac{2}{\delta_{\sigma_0}^2} \ge \frac{2}{[K':\Q]} \cdot \frac{\sqrt{n - 1/4}}{2n}.
	\end{equation}
	
	We now notice that for every embedding $\sigma$ we have
	\begin{equation}
		\rho((A_\omega)_\sigma,L_\sigma) \ge \rho(E_\sigma,(L_E)_\sigma).
	\end{equation}
	Indeed, for every period $\overline{\omega} = (\overline{\omega}_1, \overline{\omega}_2, \overline{\chi}) \in \Omega_{A_\omega, \sigma} \subseteq \Omega_{E, \sigma}^2 \times \Omega_{A, \sigma}$ we have
	\begin{equation*}
		\|\overline{\omega}\|^2_{L, \sigma} = n \|\overline{\omega}_1 \|^2_{L_E,\sigma} + \|\overline{\omega}_2 \|^2_{L_E,\sigma} \ge \max\{\|\overline{\omega}_1 \|^2_{L_E,\sigma}, \|\overline{\omega}_2 \|^2_{L_E,\sigma}\}.
	\end{equation*}
	We then have $\log \max \left\lbrace 1, \frac{1}{\rho((A_\omega)_\sigma,L_\sigma)} \right\rbrace \le \log \max \left\lbrace 1, \frac{1}{\rho(E_\sigma,(L_E)_\sigma)} \right\rbrace$.
	If we define
	\begin{align}
		\begin{split}\label{eq:aleph1bar}
			\! \overline{\aleph}_1 := \, &2\Fheight(E) + 4\log\Lambda + \log\frac{n}{\pi} + 5\log 2 + \frac{4}{[K':\Q]} \sum_{\sigma \in \mathcal{V}} \log\deg_{L_\sigma}B[\sigma]\\ 
			& + \frac{2}{[K':\Q]} \sum_{\sigma} \log \max \left\lbrace 1, \frac{1}{\rho(E_\sigma,(L_E)_\sigma)} \right\rbrace + \varepsilon\log12,
		\end{split}
	\end{align}
	by Lemma \ref{slopes} we have $\aleph_1 \le \overline{\aleph}_1$. We can then replace $\aleph_1$ by $\overline{\aleph}_1$ in equation \eqref{sqrtm-ineq}, and by inequality \eqref{eq:lowerboundonm} we obtain
	\begin{equation}\label{eq:tosquareineq}
		\frac{(n-1/4)^{\frac{1}{4}}}{\sqrt{n[K':\Q]}} < \frac{3\pi x}{2\varepsilon\log2}\left(1 + \sqrt{1+ \frac{8}{9\pi x}\left(\overline{\aleph}_1 + \frac{3\pi}{4}x + \varepsilon \log2\right)}\right).
	\end{equation}
	
	By Remark \ref{minimalheight}, we have $\Fheight(E) \ge -0.75$ for every elliptic curve $E$. Since $\Lambda - 4 \cdot 1266.4\log\Lambda < 2 \cdot 1266.4(-0.75 + 1.38)$ holds for every $\Lambda \le 57000$, we can assume $\Lambda>57000$, otherwise Theorem \ref{effiso} would trivially hold. Then we have
	$$\overline{\aleph}_1 > -1.5 + 4\log57000 + 5\log2 + \varepsilon\log12 > 44.$$
	Using that $x \le x(0) = \frac{1}{\Lambda\sqrt{2n}} \le \frac{1}{\Lambda\sqrt{2}}$ (Lemma \ref{B[sigma]=0}), this gives
	$$\frac{8}{9\pi x}\left(\overline{\aleph}_1 + \frac{3\pi}{4}x + \varepsilon \log2\right) > \frac{8 \sqrt{2} \cdot 57000}{9\pi} \cdot 44 > 10^6.$$
	Since the function $\frac{1+\sqrt{z}}{\sqrt{1+z}}$ is decreasing for $z>1$, in equation \eqref{eq:tosquareineq} we obtain
	\begin{align*}
		\frac{(n-1/4)^{\frac{1}{4}}}{\sqrt{n[K':\Q]}} &< \frac{3\pi x}{2\varepsilon\log2} \cdot \frac{1+1000}{\sqrt{10^6 + 1}} \sqrt{2+ \frac{8}{9\pi x}\left(\overline{\aleph}_1 + \frac{3\pi}{4}x + \varepsilon \log2\right)}.
	\end{align*}
	Squaring both sides and bounding $x$ with $x(0) =\frac{1}{\sqrt{2n} \, \Lambda}$ we have
	\begin{align*}
		\frac{\sqrt{n-1/4}}{n[K':\Q]} &< \frac{9\pi^2 x^2}{4\varepsilon^2 (\log2)^2} \cdot 1.002 \cdot \left(2+ \frac{8}{9\pi x}\left(\overline{\aleph}_1 + \frac{3\pi}{4}x + \varepsilon \log2\right)\right) \\
		&= \frac{2\pi x}{\varepsilon^2(\log2)^2} \cdot 1.002 \left( \overline{\aleph}_1 + 3\pi x + \varepsilon \log2 \right) \\
		&< \frac{2.004\pi}{\varepsilon^2(\log2)^2 \Lambda \sqrt{2n}} \left( \overline{\aleph}_1 + 0.085 \right),
	\end{align*}
	where in the last inequality we bounded $x$ with $\frac{1}{57000\sqrt{2}}$, since we are assuming that $\Lambda>57000$. We then obtained that
	\begin{align}
		\begin{split}\label{daquidueopzioni}
			\Lambda < [K':\Q] \left(2- \frac{1}{2n}\right)^{-\frac{1}{2}} \frac{2.004\pi}{\varepsilon^2(\log2)^2} \ (\overline{\aleph}_1 + 0.085),
		\end{split}
	\end{align}
	and writing $(2-\frac{1}{2n})^{-\frac{1}{2}} \le \sqrt{\frac{2}{3}}$, we have
	\begin{align}\label{eq:unapertutti}
		\Lambda < 727 \, [K':\Q] \left(\overline{\aleph}_1 + 0.085\right).
	\end{align}
	
	We now want to bound $\overline{\aleph}_1$. We will first do it in general, and then specialise to the case $\mathcal{B} = \mathcal{C}_{sp} = \emptyset$.
	We will use the following lemma.
	
	\begin{lemma}\label{lemma:matriciel}
		Let $E$ be an elliptic curve defined over a number field $K$. We have
		$$\frac{1}{[K:\Q]} \sum_{\sigma : K \hookrightarrow \C} \frac{1}{\rho(E_\sigma, (L_E)_\sigma)^2} < 2.29 \Fheight(E) + 6.21.$$
	\end{lemma}
	
	\begin{proof}
		If we call $T:= \frac{1}{[K:\Q]} \sum_\sigma \frac{1}{\rho(E_\sigma, (L_E)_\sigma)^2}$, by \cite[Proposition 3.2]{gaudron-remond} we have that either $T < \frac{3}{\pi}$, which is better than the statement of the lemma by Remark \ref{minimalheight}, or $\pi T \le 3\log T + 6 \Fheight(E) + 8.66$. In the latter case, we can apply \cite[Lemma B.1]{smart98} and obtain
		\begin{align*}
			T &< \frac{3}{\pi}\log\left( \frac{12}{\pi}\Fheight(E) + 5.52 + 4e^2 \right) + \frac{6}{\pi} \Fheight(E) + 2.76 < 2.29 \Fheight(E) + 6.21,
		\end{align*}
		where in the last inequality we used that $\frac{12}{\pi}\Fheight(E) + 5.52 + 4e^2 > 32.2$ for $\Fheight(E)>-0.75$, and that $\frac{\log x}{x} \le \frac{\log 32.2}{32.2}$ for $x \ge 32.2$.
	\end{proof}
	
	We can apply Lemma \ref{lemma:matriciel} and use the concavity of the logarithm to obtain
	\begin{align}
		\begin{split}\label{eq:logmatriciel}
			&\frac{2}{[K':\Q]} \sum_{\sigma : K' \hookrightarrow \C} \log\max \left\lbrace 1, \frac{1}{\rho(E_\sigma, (L_E)_\sigma)} \right\rbrace \\
			\le& \max\left\lbrace 0, \log\left( \frac{1}{[K':\Q]} \sum_{\sigma : K' \hookrightarrow \C} \frac{1}{\rho(E_\sigma, (L_E)_\sigma)^2} \right) \right\rbrace \\
			<& \log(2.29 \Fheight(E) + 6.21) < \log(\Fheight(E) + 2.72) + 0.829.
		\end{split}
	\end{align}
	
	By the definition of $B[\sigma]$ we have $x=x(B[\sigma]) \le x(0)$, and so either $B[\sigma]=0$, which gives $\deg_{L_\sigma} B[\sigma] = 1$, or $\frac{\deg_{L_\sigma}B[\sigma]}{\deg_{L_\sigma} A_\omega} \le \frac{1}{\sqrt{\deg_{L_\sigma} A_\omega}}$, which gives $\deg_{L_\sigma} B[\sigma] \le \sqrt{\deg_{L_\sigma} A_\omega}$. In particular, using equation \eqref{eq:degLAomega} we have the bound
	\begin{align}\label{eq:boundsuiBsigma}
		\frac{4}{[K':\Q]} \sum_{\sigma \in \mathcal{V}} \log\deg_{L_\sigma}B[\sigma] &\le 2\log\deg_{L_\sigma} A_\omega \le 4\log\Lambda + 2\log n + 2\log2.
	\end{align}
	
	If we use inequalities \eqref{eq:logmatriciel} and \eqref{eq:boundsuiBsigma} to bound $\overline{\aleph}_1$ and we replace it in equation \eqref{eq:unapertutti} we obtain
	\begin{align*}
		\Lambda < 727 [K':\Q] \left( 2\Fheight(E) \! + \! \log(\Fheight(E) \! + \! 2.72) \! + \! 8\log\Lambda \! + \! 3\log n \! - \! \log\pi \! + \! 7\log 2 \! + \! 0.914 \! + \! \varepsilon\log12 \right) \! .
	\end{align*}
	By our choice of $\sigma_0$, we know that $\frac{1}{\rho(E_{\sigma_0}, (L_E)_{\sigma_0})^2} = \Im\{\tau_{\sigma_0}\}$ is smaller than or equal to the mean of the values $\Im\{\tau_\sigma\}$. Using Lemma \ref{lemma:matriciel} we then have
	$$n \le |\tau_{\sigma_0}|^2 \le \Im\{\tau_{\sigma_0}\}^2 + \frac{1}{4} \le (2.29\Fheight(E) + 6.21)^2 + \frac{1}{4},$$
	and so
	\begin{equation}\label{eq:logn}
		\log n \le 2\log(\Fheight(E) + 2.72) + 1.662.
	\end{equation}
	We then obtain
	\begin{align*}
		\Lambda &< 1454 \, [K':\Q] \left( \Fheight(E) + \frac{7}{2}\log(\Fheight(E) + 2.72) + 4\log\Lambda + 5\right) \\
		&= 1454 \cdot 2^{|\mathcal{C}|} [K:\Q] \left( \Fheight(E) + \frac{7}{2}\log(\Fheight(E) + 2.72) + 4\log\Lambda + 5\right),
	\end{align*}
	concluding the proof of the first part of Theorem \ref{thm:effiso}.

	\subsubsection*{The non-split Cartan case}

	When $\mathcal{B} = \mathcal{C}_{sp} = \emptyset$ we are able to obtain a better bound. To do this, we show that the subvarieties $B[\sigma]$ are equal to $0$ for every embedding $\sigma$.
	
	We distinguish cases according to whether $\Lambda \le \sqrt{2n}$ or $\Lambda > \sqrt{2n}$.
	
	\begin{lemma}
		If $\Lambda \le \sqrt{2n}$, then Theorem \ref{effiso} holds for $E$ and $\Lambda$.
	\end{lemma}
	
	\begin{proof}
		If $\Lambda \le \sqrt{2n}$, we can write $$\Lambda \le \sqrt{2\lfloor |\tau_{\sigma_0}|^2 \rfloor} \le \sqrt{2|\tau_{\sigma_0}|^2} \le \sqrt{2\left((\Im\{\tau_{\sigma_0}\})^2 + \frac{1}{4}\right)} \le \sqrt{2} \, \Im\{\tau_{\sigma_0}\} + \frac{1}{\sqrt{2}}.$$
		Remark \ref{rmk: princ pol on ell curves} gives $\Im\{\tau_{\sigma}\} = \rho(E_\sigma, L_\sigma)^{-2}$, so by Lemma \ref{lemma:matriciel} we have $$\Im\{\tau_{\sigma_0}\} \le \frac{1}{[K':\Q]}\sum_{\sigma} \Im\{\tau_\sigma\} \le 3\Fheight(E) + 6.5,$$
		and therefore $\Lambda < 5\Fheight(E) + 10$, which is largely better than Theorem \ref{effiso} (taking into account that $\Fheight(E)>-0.75$ by Remark \ref{minimalheight}).
	\end{proof}
	
	\begin{lemma}\label{B[sigma]=0}
		Assume $\Lambda > \sqrt{2n}$ and $\mathcal{B} = \mathcal{C}_{sp} = \emptyset$. Given $A_\omega$ considered as an abelian variety over $K'$ as above, for every $\sigma: K' \hookrightarrow \C$ we have $B[\sigma]=0$, and hence $x=\frac{1}{\Lambda\sqrt{2n}}$.
	\end{lemma}
	
	\begin{proof}
		Since $A_\omega \cong A$, it is sufficient to prove the statement for $A$ and $L=\psi^*L'$. First, we notice that $x(0)=\left(\frac{1}{2n\Lambda^2}\right)^\frac{1}{2}= \frac{1}{\Lambda\sqrt{2n}}$. Let us now consider an arbitrary proper abelian subvariety $B$ such that $\dim B >0$. The subgroups of $A$ correspond to those of $E \times E$ that contain the group $$G = \prod_{p \in \mathcal{C}_{ns}} \{(x,g_p \cdot x) \mid x \in E[p^{n_p}]\}.$$
		Since $B$ is a proper subvariety of the abelian surface $A$, we have $\dim B = 1$. Hence, given the isogeny $\varphi : E \times E \to A$, the group $\varphi^{-1}(B) \subset E \times E$ is an algebraic subgroup of dimension $1$ containing $G$. In particular, there exists an elliptic curve $C \subset E \times E$ such that the algebraic group $\varphi^{-1}(B)$ is $\tilde{C} := \langle C , G \rangle$, and $C$ is the connected component of $\tilde{C}$ that contains $0$. Since $\ker \varphi = G$, we have $\varphi(C) = \varphi(\tilde{C}) = B$, and $C=[\Lambda](C)=\psi \circ \varphi (C)= \psi(B)$.
		By assumption, $E$ does not have CM, hence there exist two relatively prime integers $a,b$ such that $$C=\{(P,Q) \in E \times E \mid aP=bQ\}.$$
		Therefore, we have $\deg\varphi|_C = |\ker\varphi|_C| = |C \cap G|$. We now notice that $C \cap G = \prod_{p \in \mathcal{C}_{ns}} (C \cap G_p)$: indeed, the groups $G_p$ are $p$-groups and hence have pairwise coprime orders. The same holds for the subgroups $C \cap G_p$. The group $C \cap G$ is generated by the groups $C \cap G_p$, and for every pair of primes $p,q$ the groups $C \cap G_p$ and $C \cap G_q$ intersect trivially: this implies that $C \cap G$ is the direct product of the groups $C \cap G_p$. For every $p$ we have
		$$C \cap G_p = \left\lbrace  (x,g_p \cdot x) \mid x \in E[p^{n_p}] \text{ such that } (a-bg_p)x=0 \right\rbrace.$$
		However, given
		\begin{align*}
			g_p = \begin{pmatrix} \alpha & \varepsilon \beta \\ \beta & \alpha \end{pmatrix} \qquad \text{we have} \qquad
			a-bg_p = \begin{pmatrix} a-b\alpha & -\varepsilon b\beta \\ -b\beta & a-b\alpha \end{pmatrix}.
		\end{align*}
		By assumption we have $p \nmid \beta$, hence if $p \nmid b$, we have that $a-bg_p$ is invertible (since it is an element of $C_{ns}(p^{n_p})$). If instead $p \mid b$, then $p \nmid a$ and so $p \nmid a-b\alpha$, and $a-bg_p$ is again invertible. This implies that $C \cap G_p = 0$, and so $C \cap G = 0$. This shows that $\deg\varphi|_C = 1$. We then have 
		$$\Lambda^2 = \deg[\Lambda]|_C = (\deg\psi|_B)(\deg\varphi|_C) = \deg\psi|_B,$$
		and therefore $\deg_LB = \deg_{\psi^*L'}B = (\deg\psi|_B)\deg_{L'}C \ge \Lambda^2$.
		We can now estimate
		$$x(B) = \frac{\deg_LB}{\deg_LA} \ge \frac{\Lambda^2}{2n\Lambda^2} = \frac{1}{2n},$$
		and so $x(B) > x(0) = \frac{1}{\Lambda \sqrt{2n}}$ since $\Lambda > \sqrt{2n}$.
	\end{proof}
	
	\begin{remark}
		As $B[\sigma]=0$ for every $\sigma$, we have $\rho((A_\omega)_\sigma,L_\sigma) = \delta_\sigma$.
	\end{remark}
	
	Given $\overline{\aleph}_1$ as in \eqref{eq:aleph1bar}, we can use Lemma \ref{B[sigma]=0} and inequality \eqref{eq:logmatriciel} to obtain
	\begin{align*}
		\overline{\aleph}_1 < 2\Fheight(E) + \log(\Fheight(E) + 2.72) + 4\log \Lambda + \log\frac{n}{\pi} + 5\log2 + \varepsilon \log12 + 0.829.
	\end{align*}
	Combined with \eqref{eq:logn} and \eqref{eq:unapertutti}, this yields
	\begin{align*}
		\Lambda < 1454 \cdot 2^{|\mathcal{C}|} [K:\Q] \left( \Fheight(E) + \frac{3}{2}\log(\Fheight(E) + 2.72) + 2\log\Lambda + 2.6 \right),
	\end{align*}
	proving the second part of Theorem \ref{effiso}.

	\subsubsection*{The case $K=\Q$}

	Suppose now that $K=\Q$. Since $E$ is defined over $\Q$, we have that $\tau_{\sigma} = \tau_{\sigma_0} = \tau$ for every embedding $\sigma$.
	We then have
	\begin{align*}
		\frac{2}{[K':\Q]} \sum_{\sigma \in \mathcal{V}} \log \max \left\lbrace 1, \frac{1}{\rho(E_\sigma,(L_E)_\sigma)} \right\rbrace = \max\{0, \log\Im\{\tau\}\}
	\end{align*}
	and $\log n \le \log|\tau|^2 \le \log\left(\Im\{\tau\}^2 + \frac{1}{4}\right)$.
	If $|\tau|^2<2$, we have $\log n = 0$, if instead $|\tau|^2 \ge 2$, then $\Im\{\tau\}^2 \ge \frac{7}{4}$ and so $\log\left(\Im\{\tau\}^2 + \frac{1}{4}\right) \le \log\left(\frac{8}{7}\Im\{\tau\}^2\right) = 2\log(\Im\{\tau\}) + \log(8/7)$. We can combine the two cases by writing $\log n \le \max\{0,2\log(\Im\{\tau\})\} + \log(8/7)$.
	Bounding $\overline{\aleph}_1$ with
	$$2\Fheight(E) + 4\log\Lambda -\log\pi + 5\log2 + \varepsilon\log12 + 3\max\{0, \log\Im\{\tau\}\} + \log\left(\frac{8}{7}\right),$$
	by equation \eqref{eq:unapertutti} we obtain
	$$\Lambda < 1454 \cdot 2^{|\mathcal{C}|} \left( \Fheight(E) + 2\log\Lambda + \frac{3}{2}\max\{0, \log(\Im\{\tau\})\} + 1.38 \right),$$
	which proves part 3 of Theorem \ref{effiso}.
	
	\medskip
	
	To conclude the proof of Theorem \ref{thm:effiso}, we notice that for $\Im\{\tau_{\sigma_0}\} \ge \frac{15}{\pi}$, we can write $n \geq |\tau_{\sigma_0}|^2-1 \ge \Im\{\tau_{\sigma_0}\}^2-1 > 21.7$. We can then estimate $\left(2- \frac{1}{2n}\right)^{-\frac{1}{2}} \le \sqrt{\frac{44}{87}}$. Hence, in equation \eqref{daquidueopzioni} we can use the estimate
	$$\left(2- \frac{1}{2n}\right)^{-\frac{1}{2}} \frac{2.004\pi}{\varepsilon^2(\log2)^2} \le 633.2.$$
	Repeating the rest of the proof in the same way we obtain the desired inequalities.

	\subsection{Bounds for non-integral $j$-invariants}\label{sec:jnonintero}

	We now show that in the setting of Theorem \ref{thm:effiso}, if we assume that $j(E) \notin \OK$, we can obtain stronger bounds on $\Lambda$ whenever $\mathcal{B} = \emptyset$. This will allow to prove Theorem \ref{thm:totaleffiso}, reducing a lot the size of the constants involved.
	The approach here is completely different: instead of studying the complex structure of $E$ and the periods of auxiliary complex abelian varieties, we rely on local arguments at primes $\mathfrak{p} \mid p$ for which $\rho_{E,p}$ is not surjective.
	
	\begin{proposition}\label{prop:jnoninteroesponente}
		Let $E$ be an elliptic curve defined over a number field $K$. Let $\mathcal{C}_{sp}, \mathcal{C}_{ns}$ be disjoint sets of odd primes $p$ such that $\operatorname{Im}\rho_{E,p} \subseteq H(p)$ up to conjugacy for $H = C_{sp}^+, C_{ns}^+$ respectively. Let $n_p$ be the largest positive integer such that $\operatorname{Im}\rho_{E,p^{n_p}} \subseteq H(p^n)$ up to conjugacy, and let $\Lambda:= \prod_{p \in \mathcal{C}_{sp} \cup \mathcal{C}_{ns}} p^{n_p}$. Then $\Lambda$ divides
		$$\gcd_{\lambda \subseteq \OK \text{ prime}}(\max\{0,-v_\lambda(j(E))\}).$$
	\end{proposition}
	
	\begin{proof}
		If $j(E) \in \OK$ the statement is trivial. Let $\lambda$ be a prime of $K$ such that $e := -v_\lambda(j(E)) > 0$ and let $p^n$ be a prime power such that $p^n \mid \Lambda$. We want to show that $p^n \mid e$. We can assume that $\zeta_p \in K$: indeed, $p$ does not divide $[K(\zeta_p):K]$ and so the power of $p$ that divides the valuation of $j(E)$ at primes of $K(\zeta_p)$ above $\lambda$ is the same as that of the $\lambda$-adic valuation of $j(E)$.
		Consider $E$ to be defined over $K_\lambda$, and let $E_q$ be the Tate curve with parameter $q \in K_\lambda^\times$, isomorphic to $E$ over a quadratic extension of $K_\lambda$. We know that $v_\lambda(q) = e$. Suppose first that $p \in \mathcal{C}_{ns}$: if $\chi_{p^n}$ is the cyclotomic character modulo $p^n$, there is a quadratic character $\psi$ such that $\rho_{E_q,p^n} \cong \rho_{E,p^n} \otimes \psi$, and we have
		$$\rho_{E,p^n} \otimes \psi \cong \begin{pmatrix} \chi_{p^n} & k \\ 0 & 1 \end{pmatrix} = \begin{pmatrix} 1 & k \\ 0 & 1 \end{pmatrix},$$
		where $k(\sigma)$ is such that $\sigma\left(q^\frac{1}{p^n}\right) = q^\frac{1}{p^n} \zeta_{p^n}^{k(\sigma)}$. Indeed, as shown in case (ii) of the proof of Proposition \ref{prop:goodreductionforcartan}, the image of an automorphism $\sigma$ via $\chi_{p^n}$ must be $\pm 1$, however it cannot be $-1$ because $\chi_p = \chi_{p^n} \pmod p$ is trivial, as $\zeta_p \in K$.
		By the definition of $\mathcal{C}_{ns}$, for every $\sigma \in \Gal\left( \faktor{\overline{K_\lambda}}{K_\lambda} \right)$ we know that $M_\sigma := (\rho_{E,p^n} \otimes \psi)(\sigma)$ is conjugate to an element of $C_{ns}^+(p^n)$. We call $G:=\operatorname{Im}(\rho_{E,p^\infty} \otimes \psi)$ and $G(p^n):=\operatorname{Im}(\rho_{E,p^n} \otimes \psi)$. Let $0 \le i \le n$ be such that $M_\sigma = \begin{pmatrix} 1 & u p^i \\ 0 & 1 \end{pmatrix}$, for $u \not\equiv 0 \pmod p$. If $i=0$, then $M_\sigma \pmod p$ has non-diagonal Jordan form over $\overline{\F}_p$, and hence it cannot be an element of $C_{ns}^+(p)$. If instead $0<i<n$, we can write $M_\sigma = I + p^i \begin{pmatrix} 0 & u \\ 0 & 0 \end{pmatrix}$, and so there is an element in $\mathfrak{g}_i$ of rank $1$ (where $\mathfrak{g}_i$ is defined in Definition \ref{def:liealgebra}). However, since $G(p^n) \subseteq C_{ns}^+(p^n)$ up to conjugation, by Remark \ref{rmk:cartanliealg} the group $\mathfrak{g}_i$ is conjugate to a subgroup of $V_1 \oplus V_2$ defined as in Lemma \ref{lemma:subrepdecomposition}, which contains no matrices of rank $1$. Since the rank is invariant under conjugation, we have that $i=n$, and in particular $M_\sigma = I$. This implies that $k(\sigma) = 0$ for every $\sigma$, and in particular that $\sigma\left(q^\frac{1}{p^n}\right) = q^\frac{1}{p^n}$. Hence $q^\frac{1}{p^n} \in K_\lambda$, and so $p^n \mid e$. \\
		If instead we have $p \in \mathcal{C}_{sp}$, there exists again a quadratic character $\psi$ such that
		$$\rho_{E,p^n} \otimes \psi \cong \begin{pmatrix} \chi_{p^n} & k \\ 0 & 1 \end{pmatrix} \equiv \begin{pmatrix} 1 & k \\ 0 & 1 \end{pmatrix} \pmod p,$$
		with $\sigma\left(q^\frac{1}{p^n}\right) = q^\frac{1}{p^n} \zeta_{p^n}^{k(\sigma)}$. In particular, every element in $\operatorname{Im}\rho_{E,p^n}$ can be written as $I+pA$, with $A$ of the form $\begin{pmatrix} \ast & \ast \\ 0 & 0 \end{pmatrix}$. Since $C_{sp}^+(p)$ does not contain elements of order $p$, we must have $k \equiv 0 \pmod p$. If we call $G$ a group conjugate to $\operatorname{Im}\rho_{E_q,p^\infty}$ such that $\pm G(p^n) \subseteq C_{sp}^+(p^n)$, we must have that for every $1 \le i < n$, every element in $\mathfrak{g}_i$ has rank at most $1$. By Remark \ref{rmk:cartanliealg}, these elements must lie in $V \oplus W$, with $V:= \F_p \begin{pmatrix} 1 & 0 \\ 0 & 0 \end{pmatrix}$ and $W:=\F_p \begin{pmatrix} 0 & 0 \\ 0 & 1 \end{pmatrix}$. However, we must have that either $\mathfrak{g}_i \subseteq V$ or $\mathfrak{g}_i \subseteq W$: indeed, if we had $0 \ne x \in \mathfrak{g}_i \cap V$ and $0 \ne y \in \mathfrak{g}_i \cap W$, then the matrix $x+y$ would lie in $\mathfrak{g}_i$, which is impossible as $x+y$ has rank $2$. By Lemma \ref{lemma:gncontainment}(1), we have $\mathfrak{g}_1 \subseteq \ldots \subseteq \mathfrak{g}_{n-1}$, hence they are all contained in the same subspace $V$ or $W$. 
		Let $i$ be the smallest integer for which $\mathfrak{g}_i \ne 0$. If we take a non-zero element of $\mathfrak{g}_i$, this is the image of an element in $G(p^n)$ of order $p^{n-i}$. On the other hand, as $G(p) = \{I\}$, we have
		$$|G(p^n)| = \prod_{t=1}^{n-1} |\mathfrak{g}_t| = \prod_{t=i}^{n-1} |\mathfrak{g}_t| = p^{n-i}.$$
		In particular, this implies that $G(p^n)$ is a cyclic $p$-group, as is $H(p^n):= \operatorname{Im}(\rho_{E,p^n} \otimes \psi)$.
		Let
		\begin{equation*}
			M_\sigma := (\rho_{E,p^n} \otimes \psi)(\sigma) = I + p^i \begin{pmatrix} a & b \\ 0 & 0 \end{pmatrix}
		\end{equation*}
		be a generator of $H(p^n)$, where $i$ is as large as possible (i.e. either $a$ or $b$ are coprime with $p$). If $i \ge n$ then $M_\sigma = I$, and so $H(p^n) = I$: we conclude as in the non-split case that $p^n \mid e$.
		If $i < n$, we can assume that $p \nmid a$: indeed, if we had $p \mid a$, modulo $p^{i+1}$ we would have $M_\sigma \equiv I + p^i \begin{pmatrix} 0 & b \\ 0 & 0 \end{pmatrix} \pmod {p^{i+1}}$, where $b \not\equiv 0 \pmod {p^{i+1}}$. In particular, we would have a non-zero element $x \in \mathfrak{g}_i$ conjugate to $\begin{pmatrix} 0 & b \\ 0 & 0 \end{pmatrix}$, and so such that $\operatorname{tr} x = \det x = 0$. However, this is impossible because such an element cannot lie in $V \oplus W$.
		Consider the element $q^\frac{1}{p^n}\zeta_{p^n}^{-b/a}$: in the basis $( \zeta_{p^n}, q^\frac{1}{p^n} )$ this is expressed as the vector $\left(-b/a , 1\right)$. We then have
		\begin{align*}
			\sigma(q^\frac{1}{p^n}\zeta_{p^n}^{-b/a}) \longleftrightarrow \begin{pmatrix} 1+p^i a & p^i b \\ 0 & 1 \end{pmatrix} \begin{pmatrix} -b/a \\ 1 \end{pmatrix} = \begin{pmatrix} -b/a \\ 1 \end{pmatrix} \longleftrightarrow q^\frac{1}{p^n}\zeta_{p^n}^{-b/a},
		\end{align*}
		and so $\sigma(q^\frac{1}{p^n}\zeta_{p^n}^{-b/a}) = q^\frac{1}{p^n}\zeta_{p^n}^{-b/a}$. Since $M_\sigma$ generates $H(p^n)$, this implies that $q^\frac{1}{p^n}\zeta_{p^n}^{-b/a}$ is fixed by every automorphism, and so $q^\frac{1}{p^n}\zeta_{p^n}^{-b/a} \in K_\lambda$. We conclude as in the non-split case that $v_\lambda(q^\frac{1}{p^n}\zeta_{p^n}^{-b/a}) = \frac{e}{p^n} \in \Z$, and hence $p^n \mid e$.
	\end{proof}
	
	\begin{theorem}\label{thm:cartanboundwithdenominator}
		Let $E$ be an elliptic curve defined over a number field $K$ of degree $d=[K:\Q]$, $n$ a positive integer, and $p$ an odd prime such that $\operatorname{Im}\rho_{E,p^n} \subseteq C_{ns}^+(p^n)$ up to conjugation. Suppose that $j = j(E) \notin \OK$, then
		\begin{equation*}
			p^n \le \frac{1.6 dn\operatorname{h}(j)}{\log(1.6 dn\operatorname{h}(j)) - \log\log(1.6 dn\operatorname{h}(j))} < \frac{2.2 dn \operatorname{h}(j)}{\log(dn \operatorname{h}(j))}.
		\end{equation*}
		Moreover, if $p^{n-1}(p-1) \nmid 2d$ we have
		\begin{equation*}
			p^n < \frac{d\operatorname{h}(j) + 1.116}{\log(d\operatorname{h}(j) + 1.116) - \log\log(d\operatorname{h}(j) + 1.116)} < 1.68 \cdot \frac{d\operatorname{h}(j)}{\log(d\operatorname{h}(j))}.
		\end{equation*}
	\end{theorem}
	
	\begin{proof}
		Let $M_K$ be the set of all places of $K$. We have
		\begin{align}\label{eq:buttovialogj}
			\operatorname{h}(j) = \frac{1}{d} \sum_{\nu \in M_K} n_\nu \log\max\{1,\|j\|_\nu\} \ge \frac{1}{d} \sum_{\substack{{\lambda \subseteq \OK \text{ prime}} \\ {v_\lambda(j) < 0}}} n_\lambda \log\|j\|_\lambda,
		\end{align}
		where $n_\nu$ are the local degrees, and the inequality is obtained by taking the sum only over the finite places. We remark that, since $j \notin \OK$, the sum on the RHS of equation \eqref{eq:buttovialogj} is non-zero. By Proposition \ref{prop:jnoninteroesponente} we know that for every prime $\lambda \subseteq \OK$ such that $v_\lambda(j) < 0$, we have $p^n \le -v_\lambda(j)$. Moreover, we have $\log\|j\|_\lambda = -v_\lambda(j)\log N_{K/\Q}(\lambda) \ge p^n\log N_{K/\Q}(\lambda)$. By Proposition \ref{prop:goodreductionforcartan} we know that either $\lambda \mid p$ or $N_{K/\Q}(\lambda) \equiv \pm 1 \pmod {p^n}$.
		This implies that either $p$ divides $N_{K/\Q}(\lambda)$ or $N_{K/\Q}(\lambda) \ge p^n-1 \ge p-1$. In both cases, by equation \eqref{eq:buttovialogj} we have
		\begin{equation}
			d\operatorname{h}(j) \ge \sum_{\substack{{\lambda \subseteq \OK \text{ prime}} \\ {v_\lambda(j) < 0}}} n_\lambda \log\|j\|_\lambda \ge p^n\log(p-1).
		\end{equation}
		As $p \ge 3$, we have that $d\operatorname{h}(j) > 2$. Moreover, we have $\frac{\log p}{\log(p-1)} \le 1.6$, hence we can write $1.6 d\operatorname{h}(j) \ge p^n\log p$.
		The function $x^n\log x$ is strictly increasing, and its inverse function is $\sqrt[n]{\frac{nx}{W(nx)}}$, where $W(x)$ is the Lambert W function. This implies that $p^n \le \frac{1.6 dn\operatorname{h}(j)}{W(1.6 dn\operatorname{h}(j))}$, and by \cite[Theorem 2.1]{lambertW}, using $d\operatorname{h}(j) > 2 > \frac{e}{1.6 n}$, we have
		\begin{equation}
			p^n \le \frac{1.6 dn\operatorname{h}(j)}{\log(1.6 dn\operatorname{h}(j)) - \log\log(1.6 dn\operatorname{h}(j))}.
		\end{equation}
		If we now assume that $p^{n-1}(p-1) \nmid 2d$ and take $\lambda$ such that $v_\lambda(j) < 0$, by Proposition \ref{prop:goodreductionforcartan} we know that $\lambda \nmid p$, and so $N_{K/\Q}(\lambda) \equiv \pm 1 \pmod {p^n}$. Using $p^{n-1}(p-1) \nmid 2d$ we notice that $p^n \ne 3$, and hence $p^n \ge 5$. Similarly to above we obtain
		$$d\operatorname{h}(j) \ge p^n\log(p^n-1) = p^n\log(p^n) + p^n\log\left(1-\frac{1}{p^n}\right) > p^n\log(p^n) - 1.116,$$
		where we used that $\frac{\log(1-x)}{x} > -1.116$ for $x \in \left( 0, \frac{1}{5} \right)$. As before, the function $x\log(x)$ has inverse $\frac{x}{W(x)}$, and using $d\operatorname{h}(j) + 1.116 \ge 2 + 1.116 > e$, we can apply again \cite[Theorem 2.1]{lambertW}, obtaining the desired inequality.
		To conclude the proof, it suffices to notice that the functions
		\begin{gather*}
			\frac{1.6 x}{\log(1.6x)-\log\log(1.6x)} \qquad \text{and} \qquad \frac{x+1.116}{\log(x+1.116) - \log\log(x+1.116)}
		\end{gather*}
		are smaller than $\frac{cx}{\log x}$, for $c=2.2$ and $c=1.68$ respectively, whenever $x > 2$ (which is always the case for $d\operatorname{h}(j)$, as we proved above).
	\end{proof}
	
	\begin{corollary}\label{cor:cartanboundwithdenominatorQ}
		Let $E$ be an elliptic curve defined over $\Q$, $n$ a positive integer, and $p$ an odd prime such that $\operatorname{Im}\rho_{E,p^n} \subseteq C_{ns}^+(p^n)$ up to conjugation. Suppose that $j = j(E) \notin \Z$ and define $b(j) := \operatorname{h}(j) - \log\max\{1,|j|\}$. We have that either $p^n = 3$ or
		\begin{equation*}
			p^n < \frac{b(j) + 0.527}{\log(2b(j) + 1.054) - \log\log(2b(j) + 1.054)} < 1.3 \cdot \frac{b(j)}{\log{b(j)}}.
		\end{equation*}
	\end{corollary}
	
	\begin{proof}
		The proof is analogous to that of Theorem \ref{thm:cartanboundwithdenominator}. First, we note that we can replace $\operatorname{h}(j)$ with $b(j)$ in equation \eqref{eq:buttovialogj}, and that $b(j) \ge 2$, as $j \notin \Z$.
		We then note that by Corollary \ref{cor:goodreductionforcartanQ} we have $\ell \equiv \pm 1 \pmod {p^n}$, but for $p^n \ne 3$ the number $p^n \pm 1$ is even and greater than $2$. In particular, it cannot be prime, and so $\ell \ge 2p^n -1$.
		Since $p^n \ne 3$ we have $p^n \ge 5$ and $p^{n-1}(p-1) > 2$, so as in the proof of Theorem \ref{thm:cartanboundwithdenominator} we obtain
		\begin{align}
			b(j) &\ge p^n\log(2p^n - 1) = p^n\log(2p^n) + p^n\log\left(1-\frac{1}{2p^n}\right) > p^n\log(2p^n) - 0.527,
		\end{align}
		where we used that $\frac{\log(1-x)}{x} > -1.054$ for $x \in \left(0,\frac{1}{10}\right)$. 
		We notice that since $p^n \ge 5$ we have $b(j) \ge 5\log9 > 10$.
		One can verify that the function $x\log(2x)$ has inverse $\frac{x}{W(2x)}$, and since $2(b(j)+0.527) > e$ we can apply \cite[Theorem 2.1]{lambertW} to obtain the desired inequality. We conclude the proof by noting that the function $\frac{x+0.527}{\log(2x+1.054)-\log\log(2x+1.054)}$ is smaller than $\frac{1.3 x}{\log x}$ for $x>10$.
	\end{proof}
	
	\begin{theorem}\label{thm:multiplecartanboundwithdenominator}
		Let $E$ be an elliptic curve over a number field $K$ of degree $d$ over $\Q$. Let $\mathcal{C}_{sp}, \mathcal{C}_{ns}$ and $\Lambda$ be as in Proposition \ref{prop:jnoninteroesponente} and suppose that $j(E) \notin \OK$. We have
		\begin{equation*}
			\Lambda \le \frac{d}{\log2}\operatorname{h}(j(E)).
		\end{equation*}
		Moreover, if $K=\Q$ we have
		\begin{equation*}
			\Lambda \le \frac{1}{\log2}\left(\operatorname{h}(j(E)) - \log\max\{1,|j(E)|\}\right) < \frac{12}{\log2} \Fheight(E) + 25.
		\end{equation*}
	\end{theorem}
	
	\begin{proof}
		Set $j=j(E)$. As in equation \eqref{eq:buttovialogj} we have
		$$[K:\Q]\operatorname{h}(j) \ge \sum_{\substack{{\lambda \subseteq \OK \text{ prime}} \\ {v_\lambda(j) < 0}}} n_\lambda \log\|j\|_\lambda,$$
		and by Proposition \ref{prop:jnoninteroesponente} we have that for every $\lambda$ in the sum above, the inequality $\log\|j\|_\lambda \ge -v_\lambda(j)\log2 \ge \Lambda\log2$ holds, and the hypothesis $j(E) \notin \OK$ ensures there is at least one such prime ideal $\lambda$.
		If $K=\Q$, as in Corollary \ref{cor:cartanboundwithdenominatorQ}, we can replace $\operatorname{h}(j)$ with $\operatorname{h}(j) - \log\max\{1,|j|\}$ to obtain the first inequality. By Theorem \ref{thm:heights}, for $|j| \le 3500$ we have that 
		$$\frac{1}{\log2}\operatorname{h}(j) < \frac{12}{\log2}(\Fheight(E) + 1.429) < \frac{12}{\log2} \Fheight(E) +25.$$
		If instead $|j|>3500$, we have
		\begin{align*}
			\frac{1}{\log2}(\operatorname{h}(j) - \log|j|) &< \frac{1}{\log2}(12\Fheight(E) + 6\log\log|j| - \log|j| + 12 \cdot 0.406) < \frac{12}{\log2} \Fheight(E) + 14,
		\end{align*}
		which is even better.
	\end{proof}
	
	We now exploit the above theorem to prove Theorem \ref{thm:totaleffiso}.
	
	\begin{proof}[Proof of Theorem \ref{thm:totaleffiso}]
		By Theorem \ref{thm:multiplecartanboundwithdenominator} we can assume that $j(E) \in \Z$, otherwise we would have a better bound. Let $\tau \in \uhp$ be the element in the standard fundamental domain for the action of $\SL_2(\Z)$ corresponding to $E$, and let $q=e^{2\pi i \tau}$.
		By Lemma \ref{lemma:cartan15e7e9}, we can assume that $7 \nmid \Lambda$, $11 \nmid \Lambda$, and at most one among $3$ and $5$ divides $\Lambda$: indeed, if $7$ divided $\Lambda$ we would have $\Lambda \le 504$, which is better than the statement of the theorem, while the cases $11 \mid \Lambda$ and $15 \mid \Lambda$ never occur. Since it is known that there are no non-CM elliptic curves $E$ with $j(E) \in \Q$ and $\operatorname{Im}\rho_{E,p} \subseteq C_{ns}^+(p)$ for $p \in \{13, 17\}$ (see \cite[Corollary 1.3]{balakrishnan19} and \cite[Theorem 1.2]{balakrishnan23}), we know that $|\mathcal{C}| \le 1 + |\{p \ge 19 : p \mid \Lambda\}|$. This implies that
		\begin{align*}
			|\mathcal{C}| &\le \max\left\{\log_{19}\Lambda, \ 1 + \log_{19}\frac{\Lambda}{3}, \ 1 + \log_{19}\frac{\Lambda}{5}\right\} = \log_{19}\Lambda + 1 - \log_{19}3 < \log_{19}\Lambda + 0.627.
		\end{align*}
		Suppose first that $|\log|q|| \le 30$: by \cite[Theorem 2.8(2)]{furiolombardo23} we obtain that $\Fheight(E) < 0.6$.
		Using that $\Im\{\tau\} = \frac{|\log|q||}{2\pi}$ and writing $2^{|\mathcal{C}|} < 2^{0.627} \cdot \Lambda^{\log_{19}2}$, by Theorem \ref{thm:effiso}(3) we have that
		\begin{align*}
			\Lambda^{1-\log_{19}2} < 1454 \cdot 2^{0.627} \left(0.6 + 2\log\Lambda + \frac{3}{2}\log\left(\frac{15}{\pi}\right) + 1.38 \right).
		\end{align*}
		Solving the inequality numerically we obtain that $\Lambda < 2.41 \cdot 10^6$, which satisfies the first statement of the theorem: indeed, by Remark \ref{minimalheight} we have $\Fheight(E)>-0.75$, and so $20000 \cdot 39.25^{1.308} > 2.41 \cdot 10^6$.
		We can then assume that $|\log|q|| > 30$, and hence by \cite[Theorem 2.8(3)]{furiolombardo23} that $\Fheight(E) > 0.45$. By Theorem \ref{thm:effiso}(2) we have
		\begin{equation}\label{eq:effisocomesuKmasuQ}
			\Lambda < 1266.4 \cdot 2^{|\mathcal{C}|} \left( \Fheight(E) + \frac{3}{2}\log(\Fheight(E) + 2.72) + 2\log\Lambda + 2.6 \right).
		\end{equation}
		The function $(x+\frac{3}{2}\log(x+2.72)+2.6)/(x+8)$ is bounded by $\alpha := 1.0144$, so we have
		\begin{equation}
			\Lambda < 1266.4\alpha \cdot 2^{|\mathcal{C}|} \left( \Fheight(E) + \frac{2}{\alpha}\log \Lambda + 8 \right).
		\end{equation}
		Using again the inequality $2^{|\mathcal{C}|} < 2^{0.627} \cdot \Lambda^{\log_{19}2}$ we obtain
		\begin{align}\label{eq:effisointermedio}
			\Lambda^{1-\log_{19}2} &< 1266.4\alpha \cdot 2^{0.627} \left( \Fheight(E) + \frac{2}{\alpha}\log \Lambda + 8 \right) < 1984 \left( \Fheight(E) + \frac{2}{\alpha}\log \Lambda + 8 \right).
		\end{align}
		Since $\Lambda^{1-\log_{19}2} - \frac{2}{\alpha}\cdot 1984 \log\Lambda > 0.225 \, \Lambda^{1-\log_{19}2}$ for $\Lambda \ge 2.34 \cdot 10^6$, we have
		\begin{equation}\label{eq:effisosenzaloglog}
			\Lambda < \left(\frac{1984}{0.225}\right)^{\!\frac{1}{1-\log_{19}2}} \left( \Fheight(E) + 8 \right)^{\frac{1}{1-\log_{19}2}} < 145000 \left( \Fheight(E) + 8 \right)^{1.308},
		\end{equation}
		which holds also for $\Lambda < 2.34 \cdot 10^6$ (indeed, since $\Fheight(E) > 0.45$, we have $145000 \left( \Fheight(E) + 8 \right)^{1.308} \ge 145000 \cdot 8.45^{1.308} > 2.34 \cdot 10^6$), and hence for all values of $\Lambda$. Using inequality \eqref{eq:effisosenzaloglog} to bound $\log\Lambda$ in \eqref{eq:effisointermedio}, we obtain
		\begin{equation}
			\Lambda^{1-\log_{19}2} < 1984 \left( \Fheight(E) + \frac{2.616}{\alpha}\log\left(\Fheight(E) + 8 \right) + 31.5 \right).
		\end{equation}
		The function $x+\frac{2.616}{\alpha} \log(x+8) + 31.5$ is smaller than $1.1 \cdot (x+35)$ for every $x > 0.45$, hence we have
		\begin{align}\label{eq:effisobootstrap}
			\Lambda &< (1984 \cdot 1.1)^{1.308} \left( \Fheight(E) + 35 \right)^{1.308} < 23300 \cdot \left( \Fheight(E) + 35 \right)^{1.308}.
		\end{align}
		Repeating this last step once more, using equation \eqref{eq:effisobootstrap} in equation \eqref{eq:effisointermedio}, we obtain
		\begin{equation}\label{eq:firsttotaleffiso}
			20900 \cdot (\Fheight(E)+40)^{1.308},
		\end{equation}	
		concluding the proof of the first part of the theorem.
		We now focus on the second part. We start by assuming again that $|\log|q||>30$. Using the bound on $\Lambda$ given in equation \eqref{eq:firsttotaleffiso} in equation \eqref{eq:effisocomesuKmasuQ}, we obtain
		\begin{equation*}
			\Lambda < 1266.4 \cdot 2^{\omega(\Lambda)} \left( \Fheight(E) + 2.616\log(\Fheight(E) + 40) + 1.5\log(\Fheight(E) + 2.72) + 22.5 \right),
		\end{equation*}
		where $\omega(\Lambda)$ is the number of distinct prime factors of $\Lambda$, and applying the weighted AM-GM inequality we obtain
		\begin{equation}\label{eq:effisosenzalogLambda}
			\Lambda < 1266.4 \cdot 2^{\omega(\Lambda)} \left( \Fheight(E) + 4.116\log(\Fheight(E) + 26.42) + 22.5 \right).
		\end{equation}
		As we can assume that $\Lambda \ge 26$, by \cite[Th\'eor\`eme 13]{robin83} we have $\omega(\Lambda) \le \frac{\log\Lambda}{\log\log\Lambda - 1.1714}$, and by equation \eqref{eq:firsttotaleffiso} we have
		\begin{align*}
			\omega(\Lambda) &< \frac{1.308\log(\Fheight(E)+40) + \log20900}{\log(1.308\log(\Fheight(E)+40) + \log20900) - 1.1714} < \frac{1.308\log(\Fheight(E)+40) + \log20900}{\log(\log(\Fheight(E)+40) + 7.6) - 0.903}.
		\end{align*}
		Suppose that $\Fheight(E)> 1.2 \cdot 10^{15}$. We have the bounds $\delta(\Fheight(E)) < 0.352$ and $\frac{4.116\log(\Fheight(E) + 26.42)}{\Fheight(E)} < 10^{-10}$, so replacing in equation \eqref{eq:effisosenzalogLambda} and bounding $1.308 \cdot \log2 < 0.907$, we obtain
		\begin{align*}
			\Lambda &< 1266.5 \cdot 20900^{\log2 \cdot \delta(\Fheight(E))} (\Fheight(E)+40)^{0.907 \cdot \delta(\Fheight(E))} \left(\Fheight(E) + 22.5 \right) \\
			&< 14400 \cdot (\Fheight(E)+40)^{0.907 \cdot \delta(\Fheight(E))} \left(\Fheight(E) + 22.5 \right).
		\end{align*}
		To complete the proof, it suffices to notice that for $\Fheight(E) \le 1.2 \cdot 10^{15}$ we have
		\begin{equation*}
			21000 \left( \Fheight(E) + 40 \right)^{1.308} < 14400 \left(\Fheight(E)+40\right)^{0.907 \cdot \delta(\Fheight(E))} \left(\Fheight(E) + 22.5 \right),
		\end{equation*}
		which also holds for $|\log|q|| \le 30$, since in this case we have $\Fheight(E) < 0.6$ (as shown at the start of the proof).
	\end{proof}

	\section{$p$-adic Galois representations}

	In this section, we study the images of the $p$-adic Galois representations attached to a non-CM elliptic curve $E$ and we compute their indices in $\GL_2(\Z_p)$. This is the first step to obtain the bound on the index of the adelic image given in Section \ref{sec:adelicbound}.
	
	We will indicate the $p$-adic groups by the labels given in \cite{rszb22}, which we will call \textit{RSZB labels}. They will be of the form \verb*|N.i.g.m|, where $N$ is the level, $i$ is the index of the group in $\GL_2\left(\faktor{\Z}{N\Z}\right)$, $g$ is the genus of the corresponding modular curve, and $m$ is an ordinal used to distinguish groups with the same $N, i, g$.

	\subsection{$p$-adic images}

	Fix an elliptic curve $\faktor{E}{\Q}$ and an odd prime $p$, and write $G:=\operatorname{Im}\rho_{E,p^\infty}$ and $S:=G \cap \operatorname{SL}_2(\Z_p)$.
	The aim of this section is to show that $G$ is often an N-Cartan lift as defined in Section \ref{sec:grouptheory}, and that in this case many of the propositions of Section \ref{sec:grouptheory} apply.
	We start by proving the following.
	
	\begin{proposition}\label{prop:ellcurvecartanlift}
		Let $\faktor{E}{\Q}$ be an elliptic curve without CM and let $p$ be an odd prime such that $G(p) \subseteq C_{ns}^+(p)$. Then the group $G$ is a non-split N-Cartan lift.
	\end{proposition}
	
	\begin{proof}
		Since $\det \circ \rho_{E,p^\infty}$ is the $p$-adic cyclotomic character, it follows that $\det(G) = \Z_p^\times$. By \cite[Lemme 17]{serre81} we know that $G(p) \not\subset C_{ns}(p)$. By the open image theorem (\cite[Section 4.4, Th\'eor\`eme 3]{serre72}) we know that $G$ is open in $\GL_2(\Z_p)$, and hence it is closed.
		Finally, we need to show that $G(p) \cap C_{ns}(p)$ contains an element which is not a multiple of the identity. It is easy to notice that every element in $C_{ns}^+(p) \setminus C_{ns}(p)$ has order dividing $2(p-1)$, and the same holds for scalar matrices. Suppose by contradiction that $G(p) \cap C_{ns}(p)$ consists of multiples of the identity. In particular, every element of $G(p)$ has order dividing $2(p-1)$.
		Suppose now that $p> 11$. By Corollary \ref{cor:goodreductionforcartanQ} we know that $E$ has potentially good reduction at $p$. If we consider the subgroup $I < \operatorname{Im}\rho_{E,p}$ obtained as the image of a pro-$p$ inertia subgroup of $\GalQ$, by \cite[Proposition 1]{kraus90}, Theorem \ref{thm:canonicalsbg} and Lemma \ref{lemma:locallemmaSamuel} we know that there exists $e \in \{1,2,3,4,6\}$ such that either $I$ contains an element of order $\frac{p^2-1}{e}$, or the image of $I$ in $\operatorname{PGL}_2(\F_p)$ contains an element of order $\frac{p-1}{e}$.
		In the former case, we get a contradiction, because $\frac{p^2-1}{e} \nmid 2(p-1)$ for $p > 11$.
		In the latter case, since the square of any element of $C_{ns}^+(p) \setminus C_{ns}(p)$ is a scalar matrix, we have that $\frac{p-1}{e} \mid 2$. However, this can happen only for $p = 13$, which does not occur by \cite[Corollary 1.3]{balakrishnan19}.
		To conclude, it suffices to notice that for $p \in \{3,5,7,11\}$ the statement follows from \cite[Theorems 1.2, 1.4, 1.5, 1.6]{zywina15}.
	\end{proof}
	
	\begin{lemma}\label{lemma:dimg1}
		Let $\faktor{E}{\Q}$ be an elliptic curve and $p$ an odd prime such that $G(p) \subseteq C_{ns}^+(p)$. We have $\dim\mathfrak{g}_n \ge 2$ for every $n \ge 1$.
	\end{lemma}
	
	Lemma \ref{lemma:dimg1} is the same as \cite[Proposition 3.2]{ejder22}, however, our version also holds for $p \in \{3,5,7,13\}$. For every prime $p>7$ and $p \ne 13$, this is a consequence of the fact that if $G(p) \subseteq C_{ns}^+(p)$, then $E$ has potentially good supersingular reduction at $p$ (Corollary \ref{cor:supersingularQ}). Before proving Lemma \ref{lemma:dimg1}, treating the remaining primes, we need the following lemma.
	
	\begin{lemma}\label{lemma:ordinarycartan}
		Let $\faktor{E}{\Q}$ be an elliptic curve and $p$ an odd prime such that $E$ has potentially good ordinary reduction at $p$. If $p \ge 5$, we have $G(p^2) \not\subseteq C_{ns}^+(p^2)$. If $p=3$, we have $G(27) \not\subseteq C_{ns}^+(27)$; moreover, if $E$ has good ordinary reduction at $3$ we have $G(9) \not\subseteq C_{ns}^+(9)$.
	\end{lemma}
	
	\begin{proof}
		Let $\Q_p^{nr}$ be the maximal unramified extension of $\Q_p$ and let $K$ be the minimal extension of $\Q_p^{nr}$ over which $E$ acquires good reduction. By \cite[Proposition 1]{kraus90} and \cite[Th\'eor\`eme 1]{kraus90} we know that $e := [K : \Q_p^{nr}] \in \{1,2,3,4,6,12\}$. Let $I_K < \Gal\left(\faktor{\overline{K}}{K}\right)$ be the inertia subgroup. By Lemma \ref{lemma:ordinarydecomposition} we know that $I_K$ acts on $E[p^n]$ as $\begin{pmatrix} \chi_{p^n} & \ast \\ 0 & 1 \end{pmatrix}$, where $\chi_{p^n}$ is the cyclotomic character modulo $p^n$. 
		Suppose first that $p>3$. If we consider $n=2$, since $(e,p)=1$ we notice that $p+1 \in \operatorname{Im}\chi_{p^2} = \left(\faktor{\Z}{p^2\Z}\right)^{\times e}$. In particular, there exists an element $g$ in $\rho_{E,p^2}(I_K)$ conjugate to $\begin{pmatrix} p+1 & k \\ 0 & 1 \end{pmatrix}$ that satisfies the polynomial equation $(g-1)(g-p-1)=0$. Suppose by contradiction that $g \in C_{ns}^+(p^2)$. It is easy to check that $g \equiv I \pmod p$, and so also $k \equiv 0 \pmod p$. In particular, if we write $k=ph$ we have
		$$\begin{pmatrix} p+1 & k \\ 0 & 1 \end{pmatrix} = I + p\begin{pmatrix} 1 & h \\ 0 & 0 \end{pmatrix} = I+pA.$$
		By Proposition \ref{prop:ellcurvecartanlift} we know that $G$ is an N-Cartan lift, and by Remark \ref{rmk:cartanliealg} $A$ must be conjugate to an element of $V_1 \oplus V_2$ described in Lemma \ref{lemma:subrepdecomposition}. However, $A$ has rank $1$, which is impossible as elements of $V_1 \oplus V_2$ only have rank $0$ or $2$.
		Suppose now that $p=3$: then either $E$ has good reduction at $3$, so we have $e=1$ and we can repeat the same proof as for $p>3$, or $E$ has bad reduction at $3$. In the latter case, since by definition $v_3(e) \le 1$, we notice that $3^2+1 \in \operatorname{Im}\chi_{27}$, and hence there is an element $g \in \rho_{E,27}(I_K)$ conjugate to $\begin{pmatrix} 3^2+1 & k \\ 0 & 1 \end{pmatrix}$. Suppose that $g \in C_{ns}^+(27)$. We see as before that $k \equiv 0 \pmod 3$ and if $k \not\equiv 0 \pmod 9$ we would have a non-zero element of the form $\begin{pmatrix} 0 & u \\ 0 & 0 \end{pmatrix}$ in $\mathfrak{g}_1$, which is impossible. We then conclude as before that we have an element conjugate to $\begin{pmatrix} 1 & h \\ 0 & 0 \end{pmatrix}$ inside $V_1 \oplus V_2$, which is impossible.
	\end{proof}
	
	\begin{proof}[Proof of Lemma \ref{lemma:dimg1}]
		By Lemma \ref{cor:v1inliealg} we know that $V_1 \subseteq \mathfrak{g}_1$, and hence $\dim\mathfrak{g}_1 > 0$. Suppose by contradiction that $\dim \mathfrak{g}_1 = 1$, and so that $\mathfrak{g}_1 = V_1$. By Proposition \ref{prop:groupteichmuller} we know that there exists a lift $\widetilde{G(p)}<G$ of $G(p)$ isomorphic to it via the projection such that $G = \widetilde{G(p)}G_1$, and up to conjugation of $G$ in $\GL_2(\Z_p)$ we can assume that $\widetilde{G(p)} < C_{ns}^+$. Since $\mathfrak{g}_1 = V_1$, modulo $p^2$ we obtain that $G(p^2) = \widetilde{G(p)} \cdot \left\lbrace (1+p\alpha) I \right\rbrace_{\alpha \in \F_p} < C_{ns}^+(p^2)$ and $[C_{ns}^+(p^2) : G(p^2)] = p$.
		If $p=3$, the curve $E$ corresponds to a rational point on \cite[\href{https://beta.lmfdb.org/ModularCurve/Q/9.81.1.a.1/}{Modular Curve 9.81.1.a.1}]{lmfdb}, with equation $x^3 - 6x^2y + 3x^2z + 6xyz - 6xz^2 - y^3 - 6yz^2 + z^3 = 0$ in $\mathbb{P}^2$ (this is the modular curve corresponding to $G(9)$). However, this equation has no solutions modulo $27$, and hence such a curve $E$ does not exist.
		If instead $p>3$, by Corollary \ref{cor:goodreductionforcartanQ}, we know that the curve $E$ has potentially good reduction modulo $p$. If $E$ has potentially ordinary reduction at $p$, we can apply Lemma \ref{lemma:ordinarycartan} to get a contradiction. If $E$ has potentially supersingular reduction at $p$, we can use Theorem \ref{thm:canonicalsbg} to show that $E$ does not have a canonical subgroup of order $p$, and so by \cite[Theorem 1.1]{smith23} we have that if $R \in E[p^2] \setminus E[p]$, then $p^2 \mid [\Q(R):\Q] \mid [\Q(E[p^2]):\Q]$. We know that $|\mathfrak{g}_1| = |G_1(p^2)| = \frac{|G(p^2)|}{|G(p)|} = \frac{[\Q(E[p^2]):\Q]}{[\Q(E[p]):\Q]}$, and since $p \nmid [\Q(E[p]):\Q]$ we obtain that $p^2 \mid |\mathfrak{g}_1|$. The conclusion follows from Lemma \ref{lemma:gncontainment}(1). 
	\end{proof}
	
	\begin{remark}
		In the proof above, the statement about ramification in division fields that allows us to show that $p^2 \mid [\Q(R) : \Q]$ is due to Lozano-Robledo \cite[Theorem 1.2(2)]{lozanorobledo16}. However, as pointed out in \cite{smith23}, his proof is incorrect. A correct version is provided in \cite[Theorem 1.1]{smith23}, which is the same we used in the proof.
	\end{remark}
	
	\begin{theorem}\label{thm:ellipticcartantower}
		Let $\faktor{E}{\Q}$ be an elliptic curve without CM and set $G:=\operatorname{Im}\rho_{E,p^\infty}$. Let $p$ be an odd prime such that $\operatorname{Im}\rho_{E,p} \subseteq C_{ns}^+(p)$ up to conjugation and let $n \ge 1$ be the smallest integer such that $\operatorname{Im}\rho_{E,p^\infty} \supseteq I + p^n M_{2 \times 2}(\Z_p)$. One of the following holds:
		\begin{itemize}
			\item $G(p^n) = C_{ns}^+(p^n)$ up to conjugation.
			\item $n=2$ and
			$$G(p^2) \cong C_{ns}^+(p) \ltimes \left\lbrace I + p\begin{pmatrix} a & \varepsilon b \\ -b & c \end{pmatrix} \right\rbrace,$$
			with the semidirect product defined by the conjugation action.
			\item $p=5$ and $G$ corresponds to the group with RSZB label \verb*|5.30.0.2|.
			\item $p=3$ and $\pm G$ corresponds to one of the groups with RSZB labels \verb*|3.6.0.1|, \verb*|3.12.0.1|, \verb*|9.18.0.1|, \verb*|9.18.0.2|, \verb*|9.36.0.1|, \verb*|9.36.0.2|, \verb*|9.36.0.3|.
		\end{itemize}
	\end{theorem}
	
	\begin{remark}\label{rmk: conjecture on p-adic images}
		Notice that it is conjectured that $\operatorname{Im}\rho_{E,p} \subseteq C_{ns}^+(p)$ only for $p \le 11$ (see \cite[Conjecture 1.12]{zywina15} and \cite{balakrishnan19, balakrishnan23}). We give a quick explanation of what we know and what we expect for $p \le 11$:
		\begin{itemize}
			\item the case $G(p^n) = C_{ns}^+(p^n)$ is known to occur for $p^n \in \{3, 3^2, 5, 7, 11\}$ (they correspond to modular curves with genus $0$ or genus $1$ and rank $1$), and it is thought that these are the only cases;
			\item the case $G(p^2) \cong C_{ns}^+(p) \ltimes \F_p^3$ is known to occur for $p=3$ (genus 0 modular curve) and $p=5$ (see \cite[Section 5.3]{balakrishnan23} or \cite[Table 1]{rszb22}), and we think that these are the only possible cases. In \cite[Section 8.7]{rszb22}, they show that the case $p = 11$ never occurs;
			\item all the other special cases mentioned for $p=3,5$ correspond to genus-0 modular curves, and occur in infinitely many cases.
		\end{itemize}
		After the first version of this paper, the cases $C_{ns}^+(27)$ and $C_{ns}^+(49)$ where excluded in \cite{balakrishnan25} and \cite{furiolombardo25} respectively, and hence also all $C_{ns}^+(p^n)$ with $p \in \{3,7\}$ and $p^n > 9$. Moreover, the case $G(p^2) \cong C_{ns}^+(p) \ltimes \F_p^3$ for $p>7$ is being treated in an ongoing work by the author together with Matthew Bisatt and Davide Lombardo.
	\end{remark}
	
	\begin{proof}
		We show that $G$ satisfies the hypotheses of Theorem \ref{thm:cartantower}. First, we know by Proposition \ref{prop:ellcurvecartanlift} that $G$ is a non-split N-Cartan lift. By Lemma \ref{lemma:dimg1} we also know that $\dim \mathfrak{g}_1 > 1$. Moreover, by \cite[Theorem 3.16]{lombardo22} we know that for $p>3$ we have $G \supseteq (1+p\Z_p)I$. \\
		Suppose first that $p>5$. By Theorem \ref{thm:index3incartan} we know that $G(p) = C_{ns}^+(p)$, and hence the image of $G(p) \cap C_{ns}(p) = C_{ns}(p)$ in $\operatorname{PGL}_2(\F_p)$ contains an element of order greater than $2$. 
		We can then apply Theorem \ref{thm:cartantower}. As $G \supseteq I + p^n M_{2 \times 2}(\Z_p)$, we have either $G(p^n) \subseteq C_{ns}^+(p^n)$ with $[C_{ns}^+(p^n) : G(p^n)] = [C_{ns}^+(p) : G(p)] = 1$, or $n=2$ and $G(p^n) \cong G(p) \ltimes (V_1 \oplus V_3)$, and the conclusion follows. \\
		If $p=5$, then by Theorem \ref{thm:index3incartan} we have $[C_{ns}^+(p) : G(p)] \in \{1,3\}$. If $G(p) = C_{ns}^+(p)$, we can repeat the argument above. If instead $[C_{ns}^+(p) : G(p)] =3$, the argument above does not work anymore, because every element of $G(p) \cap C_{ns}(p)$ has order $2$ in $\operatorname{PGL}_2(\F_p)$. Using \cite[Theorem 1.6]{rszb22} we see that either $G$ corresponds to a modular curve with infinitely many rational points, or $G(25) \subseteq C_{ns}^+(25)$, or $G$ has RSZB label \verb*|25.50.2.1| or \verb*|25.75.2.1|. In the last case, we see that $G(5) \in \{C_{sp}^+(5), C_{ns}^+(5)\}$, and so we don't have $[C_{ns}^+(p) : G(p)] =3$. In the first case we can check in \cite[Table 2]{sutherlandzywina17} that the only possible case is the group with RSZB label \verb*|5.30.0.2|: indeed, this is the unique group with $G(5)$ contained in $C_{ns}^+(5)$ (which has SZ label 5C\textsuperscript{0}-5a) and index of the form $30 \cdot 5^k$ (this must hold because $G_1$ is contained in the $5$-Sylow of $\GL_2(\Z_5)$ and $[\GL_2(\F_5) : G(5)] = 30$). If $G(25) \subseteq C_{ns}^+(25)$, then $G$ must be contained in the group with RSZB label \verb*|25.750.46.1|, which is, in turn, contained in the group with RSZB label \verb*|25.50.2.1|. However, this last group has been ruled out in \cite[Section 5.3]{balakrishnan23}. Indeed, the modular curve associated with it has $2$ rational points: one is a CM point, and the other corresponds to an elliptic curve with $G(5) = C_{ns}^+(5)$, as we can check in \cite[Table 1]{rszb22}. \\
		If $p=3$, by \cite[Theorem 1.2]{zywina15} we can consider the three following cases: $G(3) = C_{ns}^+(3)$, or $G(3) = C_{sp}^+(3)$, or $G(3)$ is contained in $C_{sp}(3)$. In the first case, we have that $E[3]$ is an irreducible Galois module, and then by \cite[Proposition 3.12]{lombardo22} we have again that $G \supseteq (1+3\Z_3)I$. Moreover, the image of $C_{ns}^+(3)$ in $\operatorname{PGL}_2(\F_3)$ contains an element of order $4$, hence we can apply Theorem \ref{thm:cartantower} and conclude as for $p>5$. If $G(3) = C_{sp}^+(3)$, we can apply \cite[Theorem 1.6]{rszb22} to show that either $G(9) \subset C_{ns}^+(9)$ or $G$ appears in \cite[Table 1]{sutherlandzywina17}. If $G(9) \subset C_{ns}^+(9)$, then $G(9)$ is contained in the group corresponding to the modular curve with RSZB label \verb*|9.54.2.2|, which has no non-cuspidal non-CM points by \cite[Section 8.2]{rszb22}. If instead $G$ appears in \cite[Table 1]{sutherlandzywina17}, as in the case $p=5$, we can notice that since $G(3) = C_{sp}^+(3)$ (SZ label 3A\textsuperscript{0}-3a), the index of $G$ must be of the form $6 \cdot 3^k$, and the only such groups in the table are those corresponding to the modular curves with RSZB labels \verb*|3.6.0.1|, \verb*|9.18.0.1|, \verb*|9.18.0.2|. Suppose now that $G(3) \subseteq C_{sp}(3)$. In particular, this implies that $G(3)$ is contained in a Borel subgroup, so $\pm G$ must correspond to a modular curve in the finite list given in \cite[Corollary 1.1]{rszb22}. However, the only curves in the list for which $\pm G(3) \subseteq C_{sp}(3)$ are those with RSZB labels \verb*|3.12.0.1|, \verb*|9.36.0.1|, \verb*|9.36.0.2|, \verb*|9.36.0.3|.
	\end{proof}

	\subsection{$p$-adic indices}

	In this section, we provide some bounds on the indices of the images of the $p$-adic Galois representations attached to $E$. In particular, we will mainly focus on the case in which $\operatorname{Im}\rho_{E,p}$ is contained in the normaliser of a non-split Cartan subgroup.
	
	\begin{proposition}\label{prop:padicindices}
		Let $\faktor{E}{\Q}$ be an elliptic curve without complex multiplication and let $p$ be an odd prime such that $\operatorname{Im}\rho_{E,p} \subseteq C_{ns}^+(p)$ up to conjugation, with equality holding in the case $p=3$. Let $n \ge 1$ be the largest integer for which $\operatorname{Im}\rho_{E,p^n} \subseteq C_{ns}^+(p^n)$. We have
		\begin{align*}
			&[\GL_2(\Z_p) : \operatorname{Im}\rho_{E,p^\infty}] \in \begin{cases}
				\left\lbrace \frac{p^2-p}{2}, \frac{p^3-p^2}{2}, 30 \right\rbrace \quad &\text{for } n=1 \\
				\left\lbrace \frac{p-1}{2} \cdot p^{2n-1}\right\rbrace \quad &\text{for } n>1,
			\end{cases}
		\end{align*}
		where $[\GL_2(\Z_p) : \operatorname{Im}\rho_{E,p^\infty}] = 30$ for $p=5$.
	\end{proposition}
	
	\begin{proof}
		Suppose first that $\operatorname{Im}\rho_{E,p} = C_{ns}^+(p)$. This implies that we are in one of the first two cases of Theorem \ref{thm:ellipticcartantower}, and so if $n$ is the smallest integer such that $\operatorname{Im}\rho_{E,p^\infty} \supseteq I + p^nM_{2 \times 2}(\Z_p)$, then either $\operatorname{Im}\rho_{E,p^n} = C_{ns}^+(p^n)$, or $n=2$ and $\operatorname{Im}\rho_{E,p^2}$ is a group of order $2(p^2-1)p^3$. In particular, we have that $[\GL_2(\Z_p) : \operatorname{Im}\rho_{E,p^\infty}] \in \left\lbrace \frac{p-1}{2} \cdot p^{2n-1}, \frac{p^3-p^2}{2}\right\rbrace$.
		If instead $\operatorname{Im}\rho_{E,p} \subsetneq C_{ns}^+(p)$, by Theorem \ref{thm:ellipticcartantower} we know that $p^n=5$ and $[\GL_2(\Z_p) : \operatorname{Im}\rho_{E,p^\infty}] = 30$.
	\end{proof}
	
	\begin{corollary}\label{cor:padicindices}
		Let $\faktor{E}{\Q}$ be an elliptic curve without complex multiplication and let $p$ be an odd prime such that $\operatorname{Im}\rho_{E,p} \subseteq C_{ns}^+(p)$ up to conjugation, with equality holding in the case $p=3$. Let $n \ge 1$ be the largest integer for which $\operatorname{Im}\rho_{E,p^n} \subseteq C_{ns}^+(p^n)$. We have
		\begin{align*}
			&[\GL_2(\Z_p) : \operatorname{Im}\rho_{E,p^\infty}] \le \frac{p-1}{2p} \cdot p^{3n}.
		\end{align*}
	\end{corollary}
	
	\begin{proof}
		If $[\GL_2(\Z_p) : \operatorname{Im}\rho_{E,p^\infty}] \ne 30$ the statement easily follows from Proposition \ref{prop:padicindices}. If instead $[\GL_2(\Z_p) : \operatorname{Im}\rho_{E,p^\infty}] = 30$, then $p^n=5$ and $30 < \frac{5-1}{10} \cdot 5^3 = 50$.
	\end{proof}
	
	\begin{remark}
		Notice that, if one could manage to exclude the second case of Theorem \ref{thm:ellipticcartantower}, Corollary \ref{cor:padicindices} could by improved to give an inequality which depends on $p^{2n}$ instead of $p^{3n}$. Moreover, we remark that one could improve Corollary \ref{cor:padicindices} by writing an inequality in terms of $p^{2n+1}$; however, this would be pointless for our purpose, since Theorem \ref{thm:totaleffiso} allows us to bound the product $\prod_{p \mid \Lambda} p^{2n+1}$ with the same function as for $\prod_{p \mid \Lambda} p^{3n}$.
	\end{remark}
	
	We now give two propositions to bound the $p$-adic index in some cases in which $\operatorname{Im}\rho_{E,p}$ is not contained in the normaliser of a non-split Cartan. These cases will be the only ones that can occur whenever there exists a large prime $p$ for which $\operatorname{Im}\rho_{E,p} \subseteq C_{ns}^+(p)$.
	
	\begin{proposition}\label{prop:2adicimages}
		Let $\faktor{E}{\Q}$ be an elliptic curve without complex multiplication. Set $G := \operatorname{Im}\rho_{E,2^\infty}$ and define $X_G$ the corresponding modular curve.
		\begin{itemize}
			\item If $E$ does not admit any rational $2$-isogeny, then either $[\GL_2(\Z_2) : G]$ divides $32$ and $X_G$ has infinitely many rational points, or $j(E)$ is one among
			\begin{align*}
				-\frac{3 \cdot 18249920^3}{17^{16}}, \quad -\frac{7 \cdot 1723187806080^3}{79^{16}}
			\end{align*}
			and $[\GL_2(\widehat{\Z}) : \operatorname{Im}\rho_E] = 128$.
			\item If $E$ has a rational $2$-isogeny, the index $[\GL_2(\Z_2) : G]$ divides $96$ and is a multiple of $3$, and either $X_G$ has infinitely many rational points, or $j(E)$ is one among
			\begin{align*}
				2^{11}, \quad 2^4 \cdot 17^3, \quad \frac{4097^3}{2^{4}}, \quad \frac{257^3}{2^{8}}, \quad -\frac{857985^3}{62^8}, \quad \frac{919425^3}{496^4}
			\end{align*}
			and $[\GL_2(\widehat{\Z}) : \operatorname{Im}\rho_E] = 384$.
		\end{itemize}
	\end{proposition}
	
	\begin{proof}
		We start by proving the first part. By \cite[Theorem 1.1]{rousezb15}, \cite[Corollary 1.3]{rousezb15}, and \cite[Remark 1.5]{rousezb15} we know that either $j(E)$ is one among the two numbers in the statement,
		or the index $[\GL_2(\Z_2) : \operatorname{Im}\rho_{E,2^\infty}]$ divides $96$ and $X_G$ has infinitely many rational points. In the former case, we can compute the index of the adelic representation using the algorithm \verb*|FindOpenImage.m| developed in \cite{zywina22} (See the \href{https://github.com/davidzywina/OpenImage}{GitHub repository} accompaining the paper). Indeed, by \cite[Corollary 2.3]{zywina15index} we know that the index only depends on $j$-invariant. 
		In the latter case, since $E$ admits a rational $2$-isogeny if and only if $\operatorname{Im}\rho_{E,2}$ is contained in a Borel subgroup, we notice that $E$ admits a rational $2$-isogeny if and only if the index $[\GL_2(\Z_2) : \operatorname{Im}\rho_{E,2^\infty}]$ is divisible by $3$. The conclusion follows.
		The second part is proved in the same way.
	\end{proof}
	
	\begin{proposition}\label{prop:3e5speciali}
		Let $\faktor{E}{\Q}$ be an elliptic curve without complex multiplication.
		\begin{itemize}
			\item If $\operatorname{Im}\rho_{E,3} = \GL_2(\F_3)$, then $[\GL_2(\Z_3) : \operatorname{Im}\rho_{E,3^\infty}] \le 27$;
			\item If $\operatorname{Im}\rho_{E,5}$ is conjugate to the exceptional subgroup $5S4$, then $$[\GL_2(\Z_5) : \operatorname{Im}\rho_{E,5^\infty}] = [\GL_2(\F_5) : \operatorname{Im}\rho_{E,5}] = 5.$$
		\end{itemize}
	\end{proposition}
	
	\begin{proof}
		We start by considering the case $p=3$. By \cite[Theorem 1.6]{rszb22} and by surjectivity of $\rho_{E,3}$, we know that $\operatorname{Im}\rho_{E,3^\infty}$ corresponds to a group in \cite[Table 1]{sutherlandzywina17}. As $\operatorname{Im}\rho_{E,3}$ is equal to $\GL_2(\F_3)$, the index $[\GL_2(\Z_3) : \operatorname{Im}\rho_{E,3^\infty}]$ must be a power of $3$. However, the largest power of $3$ among the indices of \cite[Table 1]{sutherlandzywina17} is $27$, hence $[\GL_2(\Z_3) : \operatorname{Im}\rho_{E,3^\infty}] \le 27$. \\
		Consider now the case $p=5$. If $\operatorname{Im}\rho_{E,5} = 5S4$ we have $[\GL_2(\F_5) : \operatorname{Im}\rho_{E,5}] = 5$. Similarly to Lemma \ref{lemma:subrepdecomposition} one can easily check that the only non-trivial $\F_5[5S4]$-submodules of $\mathfrak{gl}_2(\F_5)$ are $\F_5 \cdot \operatorname{Id}$ and $\mathfrak{sl}_2(\F_5)$. However, if we set $G:=\operatorname{Im}\rho_{E,5^\infty}$, by Lemma \ref{lemma:ovvissimoSL2} we know that $\F_5 \cdot \operatorname{Id}$ is contained in $\mathfrak{g}_1$, and so we have $\mathfrak{g}_1 \in \{\F_5 \cdot \operatorname{Id}, \mathfrak{gl}_2(\F_5)\}$.
		If $\mathfrak{g}_1 = \F_5 \cdot \operatorname{Id}$, then $E$ corresponds to a rational point on the modular curve with RSZB label \verb*|25.625.36.1|, which has no rational points by \cite[Section 8.6]{rszb22}.	To conclude, we notice that if $\mathfrak{g}_1 = \mathfrak{gl}_2(\F_5)$, by Lemma \ref{lemma:dimg} we have that $[\GL_2(\Z_5) : G] = 5$.
	\end{proof}
	
	We conclude this section with the following lemma, which is an improved version of \cite[Lemma 2.6]{zywina11}, that will be crucial in the next section.
	
	\begin{lemma}\label{lemma: p-Sylow of S_n}
		Let $\faktor{E}{\Q}$ be an elliptic curve without CM and let $p$ be a prime. Set $G := \operatorname{Im}\rho_{E,p^\infty}$, $S := G \cap \SL_2(\Z_p)$, and $H := \rho_{E,p^\infty}(\overline{\Q}/\Qab) \subseteq S$. Recall the definition of $G_n, S_n, H_n$ from Definition \ref{def: group kers and projs}. Let $n \ge 1$ be the minimum integer such that $G \supseteq I + p^n M_2(\Z_p)$.
		\begin{enumerate}
			\item $H \supseteq S_{2n+e} = (I + p^{2n+e}M_2(\Z_p)) \cap \SL_2(\Z_p)$, where $e = 0,1$ if $p>2$ or $p=2$ respectively.
			\item If $p>2$ and $p \mid \#G(p)$, then $H \supseteq S_n = (I + p^{n}M_2(\Z_p)) \cap \SL_2(\Z_p)$ and $p \mid \#H(p)$.
			\item If $G$ is an $N$-Cartan lift, then $[S : H] = [S(p) : H(p)] \le 2$.
			\item If $p>2$ and $G(p) = \GL_2(\F_p)$, then $H = S$, and if $p>3$ then $H = \SL_2(\Z_p)$.
			\item If $p=5$ and $G(p) = 5S4$, then $[S : H] = [S(p) : H(p)] = 2$.
		\end{enumerate}
	\end{lemma}
	
	\begin{proof}
		Recall the definition of $\mathfrak{g}_n$, $\mathfrak{s}_n$, and $\mathfrak{h}_n$ from Definition \ref{def:liealgebra}. Since $\Qab$ is the maximal abelian extension of $\Q$, we have $H = [G,G]$. Part 1 of the lemma follows from Lemma \ref{lemma:gncontainment}(4). \\
		The commutator map $G \times G_n \to H_n$ given by $(g, I + p^nA) \mapsto g(I+p^nA)g^{-1}(I+p^nA)^{-1}$ induces a function $f: G(p) \times \mathfrak{g}_n \to \mathfrak{h}_n$, which can be written as $f(g,B) = gBg^{-1} - B$. \\
		We now focus on part 2. Notice that $S_n$ is stable under conjugation in $\GL_2(\Z_p)$, hence it suffices to show the statement for some conjugate of $G$. If $p \mid \#G(p)$, then either $G(p) = \GL_2(\F_p)$ or $G(p)$ is contained in a Borel subgroup. In both cases, by surjectivity of the determinant, up to conjugation, there are elements $g,h \in G(p)$ such that $g = \footnotesize \begin{pmatrix} 1 & 1 \\ 0 & 1 \end{pmatrix}$, and $h = \footnotesize \begin{pmatrix} a & b \\ 0 & d \end{pmatrix}$ with $a \ne d$. It is easy to verify that $\langle f(g, \mathfrak{gl}_2(\F_p)), f(h, \mathfrak{gl}_2(\F_p)) \rangle = \mathfrak{sl}_2(\F_p)$, and that $ghg^{-1}h^{-1}$ has order $p$. \\
		We now prove part 3. To do that, it suffices to prove that $H$ contains $S_1$, and then the conclusion follows from \cite[Theorem 1.7(2)]{lozanorobledo23}. Consider the elements in $\mathfrak{gl}_2(\F_p)$
		\begin{align*}
			g &= \begin{pmatrix} 0 & \varepsilon \\ 1 & 0 \end{pmatrix}, \quad h = \begin{pmatrix} 1 & 0 \\ 0 & -1 \end{pmatrix} \qquad \text{in the non-split case, and} \\
			g &= \begin{pmatrix} 1 & 0 \\ 0 & -1 \end{pmatrix}, \quad h = \begin{pmatrix} 0 & 1 \\ 1 & 0 \end{pmatrix} \qquad \text{in the split case.}
		\end{align*}
		If we set
		\begin{align*}
			V_1 = \F_p \cdot Id, \quad V_2 = \F_p \cdot g, \quad V_{3,1} = \F_p \cdot h, \quad V_{3,2} = \F_p \cdot gh, \quad V_3 = V_{3,1} \oplus V_{3,2},
		\end{align*}
		by Lemma \ref{lemma:subrepdecomposition} and Remark \ref{rmk:V3decomposes} we know that either $V_1, V_2, V_3$ are irreducible $\F_p[G(p)]$-modules, or $V_1, V_2, V_{3,1}, V_{3,2}$ are irreducible $\F_p[G(p)]$-modules and $\alpha g \in G(p)$ for some $\alpha \in \F_p^\times$. Indeed, by the definition of N-Cartan lift, either there exists an element in $G(p) \cap C(p)$ which has order greater than $2$ in $\operatorname{PGL}_2(\F_p)$, or there is an element of order exactly $2$, which is always equal to $g$ in $\operatorname{PGL}_2(\F_p)$. We also know that there exists $g' \in G(p) \cap C(p)$ such that $hg' \in G(p)$. Notice that for every $x \in \mathfrak{gl}_2(\F_p)$ we have $f(\alpha g, x) = f(g,x)$, and since $C(p)$ is abelian, we have $f(hg',g) = f(h,g)$.
		It is easy to verify that
		\begin{align*}
			f(g,h) = -2h, \qquad f(h,g) = -2g, \qquad f(g^{-1},gh) = hg-gh = -2gh.
		\end{align*}
		We then deduce that, if $V_3$ is reducible, then
		\begin{align*}
			f(hg',V_2) = V_2, \qquad f(g, V_{3,1}) = V_{3,1}, \qquad f(g^{-1}, V_{3,2}) = V_{3,2},
		\end{align*}
		and hence, for every positive integer $m$, we have $\mathfrak{s_m} \subseteq f(G(p), \mathfrak{g_m}) \subseteq \mathfrak{h}_m$.
		If instead $V_3$ is irreducible, we still have $f(hg', V_2) = V_2$, but we don't know anymore whether $\alpha g$ is in $G(p)$. Nevertheless, in this case there exists $t \in G(p)$ such that $tht^{-1} \in V_3$ is not a multiple of $h$. This implies that $f(t,h)$ is a non-zero element of $V_3$, and hence, for every positive integer $m$, if $V_3 \subseteq \mathfrak{g}_m$ we have $V_3 \subseteq \mathfrak{h}_m$. We obtain that also in this case $\mathfrak{s}_m$ is contained in $\mathfrak{h}_m$.
		To conclude, it suffices to notice that $H_1 \subseteq S_1$, and the fact that $\mathfrak{h}_m = \mathfrak{s}_m$ for every $m$ implies $H_1 = S_1$. \\
		To prove part 4, notice that $S(p) = \SL_2(\F_p)$. Suppose first that $p>3$. By \cite[IV \S3.4 Lemma 3]{serre-abrep} we have $S = \SL_2(\Z_p)$, and by part 2, we know that $H \supseteq S_1$. Modulo $p$, it suffices to notice that the commutator subgroup of $\GL_2(\F_p)$ is $\SL_2(\F_p)$. If instead $p=3$, we have again that $H(p) = S(p) = [G(p),G(p)] = \SL_2(\F_p)$, and by \cite[Theorem 1.6]{rszb22} we know that either $n=1$, or $n=2$ and $\dim \mathfrak{g}_1 = 1$. By Part 2 we deduce that in the first case $H$ contains $S_1$, concluding the proof, while in the second case $H$ contains $S_2$, hence to conclude it suffices to show that $\mathfrak{h}_1 = \mathfrak{s}_1$. However, if $\mathfrak{g}_1 = \F_p \cdot Id$ we have $\mathfrak{s}_1 = 0$, which trivially implies $\mathfrak{h}_1 = \mathfrak{s}_1$; while if $\mathfrak{g}_1 \ne \F_p \cdot Id$, since the centre of $\GL_2(\F_3)$ is $\{ \pm Id\}$, we have $\mathfrak{h}_1 \supseteq f(\GL_2(\F_p), \mathfrak{g}_1) \ne 0$, and hence by a dimension argument we obtain $\mathfrak{h}_1 = \mathfrak{s}_1 = \mathfrak{g}_1$ as desired.  \\
		The proof of part 5 follows by noting that $5S4 \supset C_{sp}^+(5)$, and hence we can repeat the proof of part 3 to show that $H_1 = S_1$. On the other hand, the index modulo $5$ follows from the fact that $A_4$ is the commutator subgroup of $S_4$.
	\end{proof}

	\section{Bound on the adelic index}\label{sec:adelicbound}

	The aim of this section is to combine the results from the previous sections to obtain a bound on the index of the image of the adelic Galois representation of an elliptic curve $\faktor{E}{\Q}$ without CM. In particular, we will combine Theorem \ref{thm:ellipticcartantower} and Theorem \ref{thm:totaleffiso} to obtain a bound on the contribution given by those primes for which $\operatorname{Im}\rho_{E,p}$ is contained in the normaliser of a non-split Cartan. We will then give a bound for the index at the other non-surjective primes and a bound for the entanglement phenomenon among all primes. To do this, we will use some results about the degree of entanglement fields given in Section \ref{subsec:entanglement}.
	
	Define the function
	\begin{equation}\label{eq: definition of delta}
		\delta(x) := \frac{1}{\log(\log(x+40) + 7.6) - 0.903}
	\end{equation}
	for every $x>-0.75$. Notice that by Remark \ref{minimalheight} we can always evaluate $\delta$ in $\Fheight(E)$, for every elliptic curve $E$ defined over a number field.
	
	\begin{theorem}\label{thm:adelicbound}
		Let $\faktor{E}{\Q}$ be an elliptic curve without CM and let $\Fheight(E)$ be its stable Faltings height, with the normalisation given in \cite{deligne85}. Set $\delta(x)$ as in equation \eqref{eq: definition of delta}. We have
		\begin{align*}
			[\GL_2(\widehat{\Z}) : \operatorname{Im}\rho_E] &< 9.5 \cdot 10^{20} (\Fheight(E) + 40)^{4.42} &\text{and} \\
			[\GL_2(\widehat{\Z}) : \operatorname{Im}\rho_E] &< 3.4 \cdot 10^{20} \cdot (\Fheight(E) + 22.5)^{3 + 4.158 \cdot \delta(\Fheight(E))}. &
		\end{align*}
		In particular, we have $[\GL_2(\widehat{\Z}) : \operatorname{Im}\rho_E] < \Fheight(E)^{3+O\left(\frac{1}{\log\log\Fheight(E)}\right)}$ as $\Fheight(E)$ tends to $\infty$.
	\end{theorem}

	\subsection{Entanglement}\label{subsec:entanglement}

	In this section, we assume that $p$ is a prime for which $\operatorname{Im}\rho_{E,p} \subseteq C_{ns}^+(p)$ and we study the entanglement between the $p$-torsion and the rest of the torsion.
	Before doing it, we introduce the following notation.
	
	\begin{notation}
		Let $E$ be an elliptic curve defined over a field $K$, and let $m,n$ be coprime supernatural numbers. We will set
		\begin{itemize}
			\item $\Ent_K(m,n) := [K(E[m]) \cap K(E[n]) : K]$ the $K$-entanglement degree of $E$ between $m$ and $n$;
			\item $\deg_K(n : m) := [K(E[n]) : K(E[n]) \cap K(E[m])] = [K(E[n], E[m]) : K(E[m])]$ the degree of $E[n]$ over $E[m]$.
		\end{itemize}
	\end{notation}
	
	Notice that with the notation above we have
	\begin{gather*}
		\deg_K(n : m) \cdot \Ent_K(m,n) = [K(E[n]) : K] \\
		\deg_K(n:m) \cdot [K(E[m]) : K] = [K(E[mn]) : K].
	\end{gather*}
	
	\begin{theorem}\label{thm:entanginnf}
		Let $p \ge 5$ be a prime and let $E$ be an elliptic curve defined over a $p$-adic field $K$. Let $\mathfrak{p} \subseteq K$ be the prime above $p$ with ramification index $e:=e(\mathfrak{p}|p)$, and suppose that $E$ has potentially good supersingular reduction at $\mathfrak{p}$, acquiring good reduction over an extension of ramification index $d \in \{1,2,3,4,6\}$.
		Suppose that $p \ge de-1$ and $p \ne e$, and that there exists an integer $n \ge 1$ such that $\operatorname{Im}\rho_{E,p^n} \subseteq C_{ns}^+(p^n)$ up to conjugation. Set $\eta := \frac{de}{\gcd(de,2)}$ and let $F$ be the compositum $$F := \prod_{\substack{q \text{ prime} \\ q \ne p}} K(E[q^\infty]).$$
		\begin{itemize}
			\item The ramification index of $K(E[p^n])$ over $F \cap K(E[p^n])$ is a multiple of $\frac{p^{2n-2}(p^2-1)}{\gcd(de, p^2-1)}$, and the Galois group $\Gal\left(\faktor{K(E[p^n])}{F \cap K(E[p^n])} \right)$ has an element of order $\frac{p^{n-1}(p^2-1)}{\gcd(de, p^2-1)}$.
			\item The ramification index of $K(E[p^n])$ over $F(\zeta_{p^\infty}) \cap K(E[p^n])$ is a multiple of $\frac{p^{n-1}(p + 1)}{\gcd(\eta, p + 1)}$.
		\end{itemize}
	\end{theorem}
	
	\begin{proof}
		Consider the maximal unramified extension $K^{nr}$ of $K$, and let $\faktor{L}{K^{nr}}$ be the minimal extension over which $E$ acquires good reduction. As $p \ge 5$, by \cite[Proposition 1]{kraus90} we have $d = [L : K^{nr}] \in \{1,2,3,4,6\}$. By the N\'eron--Ogg--Shafarevich criterion, for every prime $q \ne p$, as $\faktor{L(E[q^\infty])}{L}$ is unramified, we have $L(E[q^\infty])=L$, and so $FL=L$. Notice that the ramification index of $K(E[p^n])$ above $F \cap K(E[p^n])$ is equal to
		$$\deg_{K^{nr}}(p^n : 0/p^\infty) = [K^{nr}(E[p^n]) : F \cap K^{nr}(E[p^n])] = [F(E[p^n]) : F].$$
		By Theorem \ref{thm:canonicalsbg}, we know that $E$ does not have a canonical subgroup, so by \cite[Theorem 4.6]{smith23} (which is stated over number fields, but its proof holds over $p$-adic fields) we know that $L(E[p^n])$ contains elements with valuation $\frac{1}{p^{2n} - p^{2n-2}}$, and hence the degree of $L(E[p^n])$ over $L$ is divisible by $\frac{(p^2-1)p^{2(n-1)}}{\gcd(de,p^2-1)}$. This implies the first part of the first statement of the theorem, as $[F(E[p^n]) : F]$ is a multiple of $[L(E[p^n]) : L]$, because $F$ is contained in $L$. \\
		We now show that $\Gal\left(\faktor{K(E[p^n])}{F \cap K(E[p^n])} \right)$ has an element of order $\frac{p^{n-1}(p^2-1)}{\gcd(de, p^2-1)}$. Since $E$ does not have a canonical subgroup, by Lemma \ref{lemma:locallemmaSamuel} we know that $\Gal\left(\faktor{L(E[p^n])}{L}\right)$ contains an element of order $\frac{(p^2-1)p^{n-1}}{\gcd(de,p^2-1)}$.
		This gives the desired property, as $\Gal\left(\faktor{L(E[p^n])}{L}\right)$ embeds into
		\begin{align*}
			\Gal\left(\faktor{F(E[p^n])}{F} \right) \cong \Gal\left(\faktor{K(E[p^n])}{F \cap K(E[p^n])} \right).
		\end{align*}
		We now prove that $L(E[p^n]) \cap L(\zeta_{p^\infty})$ is equal to $L(\zeta_{p^n})$. First, by the properties of the Weil pairing, we have $L(\zeta_{p^n}) \subseteq L(E[p^n])$. Since $\faktor{L(\zeta_{p^\infty})}{L(\zeta_{p^n})}$ is a procyclic extension, every proper subextension must contain $L(\zeta_{p^{n+1}})$. It then suffices to show that $\zeta_{p^{n+1}} \notin L(E[p^n])$. Since $\faktor{L}{\Q_p^{nr}}$ is a tamely ramified extension, $\Gal\left( \faktor{L(\zeta_{p^{n+1}})}{L} \right)$ contains elements of order $p^n$. On the other hand, $\Gal\left( \faktor{L(E[p^n])}{L} \right)$ is a subgroup of $C_{ns}^+(p^n)$, and hence does not contain elements of order $p^n$. We then obtain that $\zeta_{p^{n+1}} \notin L(E[p^n])$ as claimed. \\
		Notice that, similarly to the first part, the ramification index of $K(E[p^n])$ over $F(\zeta_{p^\infty}) \cap K(E[p^n])$ is equal to
		$$\deg_{K^{nr}(\zeta_{p^\infty})}(p^n : 0/p^\infty) = [K^{nr}(E[p^n]) : F(\zeta_{p^\infty}) \cap K^{nr}(E[p^n])] = [F(E[p^n], \zeta_{p^\infty}) : F(\zeta_{p^\infty})].$$
		Since $\faktor{L(\zeta_p)}{L}$ is totally ramified, we have $[L(\zeta_{p}) : L] = \frac{(p-1)}{\gcd(de,p-1)}$. As $p \ge de-1$ and $p \ne e$, we necessarily have $p \nmid de$, and so $[L(\zeta_{p^n}) : L(\zeta_p)] = p^{n-1}$, because $\faktor{L(\zeta_{p^n})}{L(\zeta_p)}$ is totally wildly ramified. This implies that $[L(\zeta_{p^n}) : L] = \frac{(p-1)p^{n-1}}{\gcd(de,p-1)}$.
		
		\[\begin{tikzcd}
			& {L(E[p^n],\zeta_{p^\infty})} \\
			{L(E[p^n])} & {L(\zeta_{p^\infty})} \\
			\quad {L(\zeta_{p^n})} \\
			L & {FK^{nr}(\zeta_{p^\infty})} \quad \\
			\quad {FK^{nr}} & {K^{nr}(\zeta_{p^\infty})} \\
			{K^{nr}}
			\arrow[no head, from=1-2, to=2-1]
			\arrow[no head, from=1-2, to=2-2]
			\arrow[no head, from=3-1, to=2-1]
			\arrow[no head, from=3-1, to=2-2]
			\arrow["{k \cdot \frac{(p^2-1)p^{2(n-1)}}{\gcd(de,p^2-1)}}", curve={height=-18pt}, no head, from=4-1, to=2-1]
			\arrow["{\frac{(p-1)p^{n-1}}{\gcd(de,p-1)}}"', no head, from=4-1, to=3-1]
			\arrow[no head, from=4-2, to=2-2]
			\arrow[no head, from=5-1, to=4-1]
			\arrow[no head, from=5-1, to=4-2]
			\arrow["{\frac{d}{\gcd(d,p-1)}}"', curve={height=30pt}, no head, from=5-2, to=2-2]
			\arrow[no head, from=5-2, to=4-2]
			\arrow["d", curve={height=-18pt}, no head, from=6-1, to=4-1]
			\arrow[no head, from=6-1, to=5-1]
			\arrow[no head, from=6-1, to=5-2]
		\end{tikzcd}\]
		We then obtain that the degree of $\faktor{L(E[p^n])}{L(\zeta_{p^n})}$ is a multiple of 
		$$\frac{\gcd(de,p-1)}{\gcd(de,p^2-1)} \cdot (p+1)p^{n-1}.$$ Moreover, since $F \subseteq L$, we have that $\deg_{K(\zeta_{p^\infty})}(p^n : 0/p^\infty)$ is a multiple of
		\begin{align*}
			[L(E[p^n],\zeta_{p^\infty}) : L(\zeta_{p^\infty})] &= [L(E[p^n]) : L(\zeta_{p^\infty}) \cap L(E[p^n])] = [L(E[p^n]) : L(\zeta_{p^n})].
		\end{align*}
		In particular, we showed that $\deg_{K(\zeta_{p^\infty})}(p^n : 0/p^\infty)$ is a multiple of $\frac{\gcd(de,p-1)}{\gcd(de,p^2-1)} \cdot (p+1)p^{n-1}$. To conclude, it suffices to show that, for $\eta = \frac{de}{\gcd(de,2)}$, the number $\frac{\gcd(de,p^2-1)}{\gcd(de,p-1)}$ is a divisor of $\gcd(\eta, p+1)$. This is a consequence of the following lemma.
		\begin{lemma}
			Let $a,b,c \in \Z$ such that $ab \ne 0$. Then, $\frac{\gcd(a,bc)}{\gcd(a,b)}$ divides $\gcd\left(\frac{a}{\gcd(a,b,c)}, c\right)$.
		\end{lemma}
		\begin{proof}
			We want to show that $\gcd(a,bc)$ divides $\gcd(a,b) \gcd\left(\frac{a}{\gcd(a,b,c)}, c\right)$. Notice that $(a,bc) = (a,b) \left(\frac{a}{(a,b)}, \frac{b}{(a,b)} \cdot c\right)$. Since $\frac{a}{(a,b)}$ and $\frac{b}{(a,b)}$ are coprime, we have $\left(\frac{a}{(a,b)}, \frac{b}{(a,b)} \cdot c\right) = \left(\frac{a}{(a,b)}, c\right)$, and hence $(a,bc) = (a,b)\left(\frac{a}{(a,b)}, c\right) = (a, c \cdot (a,b))$. On the other hand, $(a, c \cdot (a,b))$ divides $\left(a \cdot \frac{(a,b)}{(a,b,c)}, c \cdot (a,b)\right) = (a,b)\left(\frac{a}{\gcd(a,b,c)}, c\right)$, concluding the proof.
		\end{proof}
		
		Now, if we apply the lemma with $(a,b,c) = (de, p-1, p+1)$, the statement follows.
	\end{proof}
	
	\begin{corollary}\label{cor:entangQ}
		Let $\faktor{E}{\Q}$ be a non-CM elliptic curve and let $p>7$ and $n \ge 1$ be integers such that $p$ is prime and $\operatorname{Im}\rho_{E,p^n} \subseteq C_{ns}^+(p^n)$. Let $F$ be the compositum $$F := \prod_{\substack{q \text{ prime} \\ q \ne p}} \Q(E[q^\infty]) = \Q\left(E\left[\frac{0}{p^\infty}\right]\right).$$
		There exists $\eta \in \{1,2,3\}$ such that
		\begin{gather*}
			\deg_\Q(p^n : 0/p^\infty) = [F(E[p^n]) : F] \quad \text{is a \emph{proper} multiple of} \quad \frac{p^{2n}-p^{2n-2}}{12}, \quad \text{and} \\
			\deg_{\Qab}(p^n : 0/p^\infty) = [F\Qab(E[p^n]) : F\Qab] \quad \text{is a multiple of} \quad \frac{p^n+p^{n-1}}{\eta}.
		\end{gather*}
		Moreover, if $E$ has good reduction at $p$ we have $\eta=1$.
	\end{corollary}
	
	\begin{proof}
		We notice that we can assume that $E$ has potentially good supersingular reduction modulo $p$. Indeed, by Corollary \ref{cor:supersingularQ} the prime $p$ is always supersingular for $p>7$ and $ \ne 13$. However, by \cite[Corollary 1.3]{balakrishnan19} we know that for $p=13$ the image of $\rho_{E,p}$ is not contained in $C_{ns}^+(p)$.
		Consider the set $R:= \{r\ge 1 \ : \ p \nmid r\}$, define the extension $K=\Q(\{\zeta_r\}_{r \in R})$ and consider $E$ to be defined over $K$. As $p$ is unramified in $K$, by Theorem \ref{thm:entanginnf} we know that there exists $\eta \in \{1,2,3\}$ such that
		\begin{align*}
			[F\Qab(E[p^n]) : F\Qab] &= [FK(\zeta_{p^\infty}, E[p^n]) : FK(\zeta_{p^\infty})] = [K(E[p^n]) : FK(\zeta_{p^\infty}) \cap K(E[p^n])]
		\end{align*}
		is a multiple of $\frac{p^n+p^{n-1}}{\eta}$.
		The fact that $[F(E[p^n]) : F]$ is a proper multiple of $\frac{p^{2n}-p^{2n-2}}{12}$ immediately follows from Theorem \ref{thm:entanginnf}.
	\end{proof}
	
	\begin{lemma}\label{lemma:intersezionestupida}
		Let $E$ be an elliptic curve over a field $K$ and let $p$ be a prime. Let $B$ be a set of primes such that for every $q \in B$ the prime $p$ does not divide $q(q^2-1)$. Define $m:= \prod_{q \in B} q$ (possibly a supernatural number) and consider the compositum $K(E[m^\infty]) := \prod_{q \in B} K(E[q^\infty])$. We have
		$$K(E[m^\infty], E[p]) \cap K(E[p^\infty]) = K(E[p]), \qquad \text{and so} \qquad \Ent_K(m^\infty, p^\infty) = \Ent_K(m^\infty, p).$$
	\end{lemma}
	
	\begin{proof}
		Set $F:=K(E[m^\infty])$. We notice that for every $q \in B$ we have $$p \nmid [K(E[q]) : K] \mid \#\GL_2(\F_q) = q(q-1)^2(q+1).$$ As $F$ is the composite of $K(E[q^\infty])$ for $q \in B$ and $\faktor{K(E[q^\infty])}{K(E[q])}$ is a pro-$q$ extension, this implies that $F$ does not contain any finite subextension with degree multiple of $p$. In particular, the same holds for $\faktor{F(E[p])}{K(E[p])}$. On the other hand, $K(E[p^\infty])$ is a pro-$p$ extension of $K(E[p])$, and so $\faktor{F(E[p]) \cap K(E[p^\infty])}{K(E[p])}$ must be trivial.
	\end{proof}
	
	\begin{corollary}\label{cor:intersechino}
		Let $\faktor{E}{\Q}$ be a non-CM elliptic curve and let $p>7$ be a prime such that $\operatorname{Im}\rho_{E,p} = C_{ns}^+(p)$. Let $B$ be a set of primes such that for every $q \in B$ the prime $p$ does not divide $q(q^2-1)$. Define $m:= \prod_{q \in B} q$ (possibly a supernatural number), and consider the compositum $\Q(E[m^\infty]) := \prod_{q \in B} \Q(E[q^\infty])$. We have
		$$\Ent_{\Qab}(m^\infty, p^\infty) = [\Qab(E[m^\infty]) \cap \Qab(E[p^\infty]) : \Qab] \le 6.$$
		Moreover, if $E$ has good reduction at $p$ we have $\Ent_{\Qab}(m^\infty, p^\infty) \le 2$.
	\end{corollary}
	
	\begin{proof}
		By Lemma \ref{lemma:intersezionestupida} we know that $\Ent_{\Qab}(m^\infty, p^\infty) = \Ent_{\Qab}(m^\infty, p)$, and so it suffices to compute
		\begin{align*}
			\Ent_{\Qab}(m^\infty, p) &= \frac{\left[\Qab(E[p]) : \Qab\right]}{\deg_{\Qab}(p : m^\infty)}.
		\end{align*}
		However, by Corollary \ref{cor:entangQ} we know that $\deg_{\Qab}(p : m^\infty) \ge \deg_{\Qab}(p : 0/p^\infty) \ge \frac{p+1}{3}$ (and greater than or equal to $p+1$ in the case of good reduction), and
		\begin{align*}
			\kern-8pt \left[\Qab(E[p]) : \Qab\right] &= [\Q(E[p]) : \Q(E[p]) \cap \Qab] \le [\Q(E[p]) : \Q(\zeta_p)] = \frac{[\Q(E[p]) : \Q]}{[\Q(\zeta_p) : \Q]}
			= 2(p+1). \kern7pt \qedhere
		\end{align*}
	\end{proof}
	
	\begin{lemma}\label{lemma:sl2entang}
		Let $\faktor{E}{\Q}$ be an elliptic curve without CM. Let $\mathcal{P}$ be a set of primes containing $2,3,5$ and all primes $p$ for which $\rho_{E,p}$ is not surjective. Let $m$ be the product of all the primes in $\mathcal{P}$ and write $\Z_m := \prod_{p \in \mathcal{P}} \Z_p$ and $\rho_{E,m^\infty} := \prod_{p \in \mathcal{P}} \rho_{E,p^\infty}$. Call $S:=\rho_E\left( \Gal\left({\overline{\Q}}/{\Qab}\right) \right) < \SL_2(\widehat{\Z})$ and $S_\mathcal{P}$ its image under the projection on $\SL_2(\Z_m)$. We have
		\begin{align*}
			[\GL_2(\widehat{\Z}) : \operatorname{Im}\rho_E] = \left[ \SL_2(\widehat{\Z}) : S \right] = \left[ \SL_2(\Z_m) : S_\mathcal{P} \right].
		\end{align*}
	\end{lemma}
	
	\begin{proof}
		The first equality follows from surjectivity of $\det \circ \rho_E$ onto $\widehat{\Z}^\times$. To prove the second inequality, it suffices to notice that we can view $S$ as a closed subgroup of $\prod_p S_p \subseteq \prod_p \SL_2(\Z_p) = \SL_2(\widehat{\Z})$, and by \cite[IV \S3.4 Lemma 5]{serre-abrep} we know that $S$ contains the subgroup $\prod_{p \notin \mathcal{P}} \SL_2(\Z_p)$, concluding the proof.
	\end{proof}
	
	\begin{lemma}\label{lemma:goursatgalois}
		Let $E$ be an elliptic curve defined over a field $K$. Let $m,n$ be coprime squarefree supernatural numbers. Set $G:=\operatorname{Im}\rho_E$ and define $G_m$, $G_n$, $G_{mn}$ to be its projections on $\GL_2(\Z_m)$, $\GL_2(\Z_n)$, $\GL_2(\Z_{mn})$ respectively. We have $$[G_m \times G_n : G_{mn}] = \Ent_K(m^\infty, n^\infty).$$
	\end{lemma}
	
	\begin{proof}
		Set $F:= K(E[m^\infty]) \cap K(E[n^\infty])$. If we write
		\begin{align*}
			G_m = \Gal\left(\faktor{K(E[m^\infty])}{K}\right) \quad \text{and} \quad G_n = \Gal\left(\faktor{K(E[n^\infty])}{K}\right),
		\end{align*}
		we know that $G_{mn} = \Gal\left(\faktor{K(E[(mn)^\infty])}{K}\right)$ is isomorphic to the subgroup of $G_m \times G_n$ described as $\left\lbrace (\sigma, \tau) \in G_m \times G_n \ : \ \sigma|_F = \tau|_F \right\rbrace$.
		We conclude the proof noting that $[G_m \times G_n : G_{mn}] = [F : K]$.
	\end{proof}
	
	\begin{corollary}\label{cor:goursatgaloisnf}
		Let $K$ be a number field and let $\faktor{E}{K}$ be an elliptic curve without CM. Let $m,n$ be coprime squarefree supernatural numbers. Set $G:=\operatorname{Im}\rho_E$ and $S:=\rho_E(\Gal(\overline{K}/K\Qab))$, and define $G_m$, $G_n$, $G_{mn}$, $S_m$, $S_n$, $S_{mn}$ to be their projections on $\GL_2(\Z_m)$, $\GL_2(\Z_n)$, $\GL_2(\Z_{mn})$, $\SL_2(\Z_m)$, $\SL_2(\Z_n)$, $\SL_2(\Z_{mn})$ respectively. We have
		\begin{align*}
			[\GL_2(\Z_{mn}) : G_{mn}] = [\GL_2(\Z_{m}) : G_m] \cdot [\GL_2(\Z_{n}) : G_n] \cdot \Ent_K(m^\infty, n^\infty)
		\end{align*}
		and
		\begin{align*}
			[\SL_2(\Z_{mn}) : S_{mn}] = [\SL_2(\Z_{m}) : S_m] \cdot [\SL_2(\Z_{n}) : S_n] \cdot \Ent_{K\Qab}(m^\infty, n^\infty).
		\end{align*}
	\end{corollary}
	
	\begin{proof}
		The first statement follows from Lemma \ref{lemma:goursatgalois} noting that
		\begin{align*}
			[\GL_2(\Z_{mn}) : G_{mn}] &= [\GL_2(\Z_{mn}) : G_m \times G_n] \cdot [G_m \times G_n : G_{mn}] \\
			&= [\GL_2(\Z_{m}) : G_m] \cdot [\GL_2(\Z_{n}) : G_n] \cdot [G_m \times G_n : G_{mn}].
		\end{align*}
		The second statement is proved in the same way replacing $K$ with $K\Qab$.
	\end{proof}

	\subsection{A bound in terms of the Faltings height}

	We now start combining all the results from the previous sections to prove Theorem \ref{thm:adelicbound}. Before doing this, we will prove many intermediate lemmas and propositions, that allow us to organise the proof in different steps. We will distinguish cases according to whether the elliptic curve $E$ satisfies the uniformity conjecture or not, treating separately the subcases $j(E) \in \Z$ and $j(E) \notin \Z$ to optimise the bound.
	
	We extend the definition of $C_{ns}^+(p)$ for $p=2$ as $C_{ns}^+(2) = \GL_2(\F_2)$. Indeed, this is the normaliser of the only subgroup (up to conjugation) of $\GL_2(\F_2)$ which is isomorphic to $\F_4^\times$.
	
	\begin{lemma}\label{lemma:exceptional}
		Let $\faktor{E}{\Q}$ be an elliptic curve without CM and let $p$ be a prime number such that $\operatorname{Im}\rho_{E,p}$ is contained in an exceptional subgroup, i.e.\! a proper subgroup of $\GL_2(\F_p)$ which is not contained in a Borel subgroup or in the normaliser of a Cartan subgroup. There are two possible cases:
		\begin{itemize}
			\item $p=5$ and $[\GL_2(\F_5) : \operatorname{Im}\rho_{E,5}] = 5$;
			\item $p=13$, the $j$-invariant $j(E)$ is one among
			\begin{gather}\label{eq:j13S4}
					\frac{2^4 \cdot 5 \cdot 13^4 \cdot 17^3}{3^{13}}, \quad -\frac{2^{12} \cdot 5^3 \cdot 11 \cdot 13^4}{3^{13}} , \quad \frac{2^{18} \cdot 3^3 \cdot 13^4 \cdot 127^3 \cdot 139^3 \cdot 157^3 \cdot 283^3 \cdot 929}{5^{13} \cdot 61^{13}},
			\end{gather}
			and $[\GL_2(\widehat{\Z}) : \operatorname{Im}\rho_E] = 182$.
		\end{itemize}
	\end{lemma}
	
	\begin{proof}
		By \cite[\S8.4, Lemme 18]{serre81} we know that $p \le 13$. Using \cite[Theorems 1.1, 1.2, 1.4, 1.5, 1.6, 1.8]{zywina15} and \cite[Remark 1.9]{zywina15}, we note that there are just two possible groups: one for $p=5$ with index $5$ and one for $p=13$ with index $91$ (respectively 5S4 and 13S4). However, by \cite[Theorem 1.1]{balakrishnan23} we know that in the case $p=13$ the $j$-invariant must belong to the list \eqref{eq:j13S4}. We can then apply the algorithm \verb|FindOpenImage| developed in \cite{zywina22} (See the \href{https://github.com/davidzywina/OpenImage}{GitHub repository} accompaining the paper) to compute the index of $\operatorname{Im}\rho_E$, which is $182$. Indeed, by \cite[Corollary 2.3]{zywina15index} the index only depends on the $j$-invariant.
	\end{proof}
	
	We collect together the facts we know about the possible images of $\rho_{E,p}$ in the following proposition.
	
	\begin{proposition}\label{prop:groupingfacts}
		Let $\faktor{E}{\Q}$ be an elliptic curve without CM. If $j(E)$ is one of the $j$-invariants of the following list
		\begin{gather}\label{eq:jlist}
			\begin{split}
				-11 \cdot 131^3, \quad -11^2, \quad -\frac{17^2 \cdot 101^3}{2}, \quad -\frac{17 \cdot 373^3}{2^{17}}, \quad -7 \cdot 11^3, \quad -7 \cdot 137^3 \cdot 2083^3, \ \ \\ 
				\frac{2^4 \cdot 5 \cdot 13^4 \cdot 17^3}{3^{13}}, \quad -\frac{2^{12} \cdot 5^3 \cdot 11 \cdot 13^4}{3^{13}}, \quad \frac{2^{18} \cdot 3^3 \cdot 13^4 \cdot 127^3 \cdot 139^3 \cdot 157^3 \cdot 283^3 \cdot 929}{5^{13} \cdot 61^{13}},
			\end{split}
		\end{gather}
		then $[\GL_2(\widehat{\Z}) : \operatorname{Im}\rho_{E}] \le 2736$. Suppose now that $j(E)$ is not in the list above.
		\begin{itemize}
			\item If $\operatorname{Im}\rho_{E,p}$ is contained in a Borel subgroup, then $p \in \{2,3,5,7,13\}$.
			\item If $\operatorname{Im}\rho_{E,p}$ is contained in the normaliser of a split Cartan subgroup, then $p \le 7$.
			\item If $\operatorname{Im}\rho_{E,p}$ is contained in the normaliser of a non-split Cartan subgroup and $p \ge 5$, then either $\operatorname{Im}\rho_{E,p} = C_{ns}^+(p)$ and $p \in \{5,7,11\} \cup \{N \ge 19\}$, or $[C_{ns}^+(p) : \operatorname{Im}\rho_{E,p}] = 3$ and $p=5$.
			\item If $\operatorname{Im}\rho_{E,p}$ is contained in an exceptional subgroup but is not contained in one of the groups in the cases above, then $p = 5$ and $[\GL_2(\Z_5) : \operatorname{Im}\rho_{E,5^\infty}] = 5$.
		\end{itemize}
	\end{proposition}
	
	\begin{proof}
		To prove that the $j$-invariants in the list have $[\GL_2(\widehat{\Z}) : \operatorname{Im}\rho_{E}] \le 2736$, it suffices to notice that by \cite[Corollary 2.3]{zywina15index} the index only depends on the $j$-invariant, and then we can compute it using the algorithm \verb|FindOpenImage| developed in \cite{zywina22}.
		If $j(E)$ is not in the list, the statement follows combining Theorem \ref{thm:mazurisogeny}, Theorem \ref{thm:split}, Theorem \ref{thm:index3incartan}, \cite[Corollary 1.3]{balakrishnan19}, \cite[Theorem 1.2]{balakrishnan23}, and Lemma \ref{lemma:exceptional}.
	\end{proof}
	
	Another result we will use is the following theorem by Lemos (\cite[Theorem 1.1]{lemos19borel} and \cite[Theorem 1.4]{lemos19split}).
	\begin{theorem}[Lemos]
		Let $\faktor{E}{\Q}$ be an elliptic curve without CM. Suppose that there exists a prime $q$ for which $\operatorname{Im}\rho_{E,q}$ is contained either in a Borel subgroup or in the normaliser of a split Cartan subgroup: then $\rho_{E,p}$ is surjective for every $p > 37$.
	\end{theorem}
	
	Lemos's arguments actually show the following stronger statement.
	
	\begin{theorem}\label{thm:lemosbenfatto}
		Let $\faktor{E}{\Q}$ be an elliptic curve without CM and let $p > 13$ be a prime such that $\operatorname{Im}\rho_{E,p} \subseteq C_{ns}^+(p)$. For every prime $q \ne p$, the image of $\rho_{E,q}$ is contained neither in a Borel subgroup nor in the normaliser of a split Cartan subgroup.
	\end{theorem}
	
	\begin{proof}
		By Theorem \ref{thm:mazurisogeny}, we know that if $E$ admits a rational $q$-isogeny then $q$ belongs to the set $\{2,3,5,7,11,13,17,37\}$. However, by \cite[Proposition 2.1]{lemos19borel} we know that either $j(E) \in \Z$ or $q \in \{11,17,37\}$.
		If $E$ admits a rational isogeny of degree $q \in \{11,17,37\}$, by Theorem \ref{thm:mazurisogeny} we know that $j(E)$ is one among
		\begin{equation*}
			-11 \cdot 131^3, \quad -11^2, \quad -\frac{17^2 \cdot 101^3}{2}, \quad -\frac{17 \cdot 373^3}{2^{17}}, \quad -7 \cdot 137^3 \cdot 2083^3, \quad -7 \cdot 11^3.
		\end{equation*}
		One can check on the LMFDB \cite{lmfdb} that in these cases $\operatorname{Im}\rho_{E,p}$ is not contained in $C_{ns}^+(p)$ for $p>13$.
		If instead $\operatorname{Im}\rho_{E,q}$ is contained in the normaliser of a split Cartan subgroup, then by \cite[Proposition 1.5]{lemos19split} we know that $j(E) \in \Z$.
		From now on, we can therefore assume that $j(E) \in \Z$. Following the proof of \cite[Theorem 1.4]{lemos19split} we have that if $\operatorname{Im}\rho_{E,q} \subseteq C_{sp}^+(q)$ then $j(E) \in \{-5000, -1728\}$, for which $\operatorname{Im}\rho_{E,p}$ is not contained in $C_{ns}^+(p)$ for $p>13$. If $j(E) \in \Z$ and the image of $\rho_{E,q}$ is contained in a Borel subgroup, then following the proof of \cite[Theorem 1.1]{lemos19borel} we have that $j(E)$ belongs to the list in \cite[page 142]{lemos19borel}, and one can check again on the LMFDB that none of those curves admits a prime $p>13$ for which $\operatorname{Im}\rho_{E,p} \subseteq C_{ns}^+(p)$.
	\end{proof}
	
	Combining Proposition \ref{prop:groupingfacts} and Theorem \ref{thm:lemosbenfatto} we obtain the following.
	
	\begin{proposition}\label{prop:dichotomy}
		Let $\faktor{E}{\Q}$ be an elliptic curve without CM. Suppose that $j(E)$ does not belong to the list \eqref{eq:jlist}. One of the following holds.
		\begin{enumerate}[(A)]
			\item There exists $p>13$ such that  $\operatorname{Im}\rho_{E,p} = C_{ns}^+(p)$, and for every $q \ne 5$ for which $\rho_{E,q}$ is not surjective we have $\operatorname{Im}\rho_{E,q} \subseteq C_{ns}^+(q)$.
			\item For every $p > 13$ the representation $\rho_{E,p^\infty}$ is surjective.
		\end{enumerate}
	\end{proposition}
	
	\begin{proof}
		Suppose first that there exists $p > 13$ such that $\operatorname{Im}\rho_{E,p} \subseteq C_{ns}^+(p)$. By Theorem \ref{thm:index3incartan} we know that $\operatorname{Im}\rho_{E,p} = C_{ns}^+(p)$. Using Theorem \ref{thm:lemosbenfatto} we see that if $q$ is a prime for which $\rho_{E,q}$ is not surjective and its image is not contained in $C_{ns}^+(q)$, then it must be contained in an exceptional subgroup. By Proposition \ref{prop:groupingfacts} this implies that $q=5$. If instead there are no primes $p>13$ for which $\operatorname{Im}\rho_{E,p} \subseteq C_{ns}^+(p)$, by Proposition \ref{prop:groupingfacts} we see that $\rho_{E,p}$ is surjective for every $p>13$. By \cite[IV-23, Lemma 3]{serre-abrep} we have that $\operatorname{Im}\rho_{E,p^\infty} \cap \SL_2(\Z_p) = \SL_2(\Z_p)$, and by surjectivity of the determinant this implies that $\operatorname{Im}\rho_{E,p^\infty} = \GL_2(\Z_p)$.
	\end{proof}
	
	\begin{definition}
		Let $\faktor{E}{\Q}$ be an elliptic curve without CM. For every integer $n > 1$ set $\Z_n = \prod_{p \mid n} \Z_p$ and $\rho_{E,n^\infty} = \prod_{p \mid n} \rho_{E,p^\infty} : \Gal\left(\faktor{\overline{\Q}}{\Q}\right) \to \GL_2(\Z_n)$. For any coprime integers $m,n > 1$, using Corollary \ref{cor:goursatgaloisnf} and the surjectivity of $\det \circ \rho_{E}$ define
		\begin{align*}
			\Ind(m) :=& \, \left[\GL_2(\Z_m) : \operatorname{Im}\rho_{E,m^\infty}\right] = \left[\SL_2(\Z_m) : \operatorname{Im}\rho_{E,m^\infty} \cap \SL_2(\Z_m)\right], \\
			\Ind_S(m) :=& \, \left[\SL_2(\Z_m) : \rho_{E,m^\infty}(\Gal(\overline{\Q}/\Qab))\right] \ge \Ind(m), \\
			\Ent(m,n) :=& \, \frac{\Ind_S(mn)}{\Ind_S(m)\Ind_S(n)} 
			= \Ent_{\Qab}(m^\infty, n^\infty).
		\end{align*}
	\end{definition}
	
	We are now ready to prove Theorem \ref{thm:adelicbound}. We will split the proof in multiple steps and cases, to make it clearer.
	
	\begin{proposition}\label{prop:productofindices}
		Let $\faktor{E}{\Q}$ be an elliptic curve without CM and suppose that $j(E)$ does not belong to the list \eqref{eq:jlist}. 
		Consider the two cases (A) and (B) of Proposition \ref{prop:dichotomy}. In the respective cases we have
		\begin{enumerate}[(A)]
			\item Let $\mathcal{C}_{ns}$ be the set of the primes $p \ge 7$ such that $\operatorname{Im}\rho_{E,p} \subseteq C_{ns}^+(p)$, and let $\beta$ be the number of primes in $\mathcal{C}_{ns}$ at which $E$ has bad reduction. We have $$[\GL_2(\widehat{\Z}) : \operatorname{Im}\rho_{E}] \le \Delta_7 \cdot 2^{|\mathcal{C}_{ns}|} \cdot 3^\beta \cdot \Ind_S(30) \cdot \prod_{p \in \mathcal{C}_{ns}} \Ind(p),$$
			where 
			\begin{align*}
				\Delta_7:=\begin{cases} 1 &\text{ if } 7 \notin \mathcal{C}_{ns}, \\ 8 &\text{ if } 7 \in \mathcal{C}_{ns} \text{ and } E \text{ has good reduction at } 7, \\ \frac{8}{3} &\text{ if } 7 \in \mathcal{C}_{ns} \text{ and } E \text{ has bad reduction at } 7. \end{cases}
			\end{align*}
			\item Let $\{2,3,5\} \subseteq \mathcal{L} \subseteq \{2,3,5,7,11,13\}$ be the set of primes containing $2,3,5$ and every $p$ for which $\rho_{E,p}$ is not surjective. Let $m_p$ be the product of primes $q < p$ that belong to $\mathcal{L}$. We have
			$$[\GL_2(\widehat{\Z}) : \operatorname{Im}\rho_{E}] = \prod_{p \in \mathcal{L}} \Ent(m_p,p) \Ind_S(p).$$
		\end{enumerate}
	\end{proposition}
	
	\begin{proof}
		Set $S = \rho_E(\Gal(\overline{\Q}/\Qab))$ and let $m$ be the product of $2$, $3$, $5$, and all primes $p$ for which $\rho_{E,p}$ is not surjective. Lemma \ref{lemma:sl2entang} yields
		\begin{align}\label{eq:finproof1}
			\left[ \GL_2(\widehat{\Z}) : \operatorname{Im}\rho_E \right] &= \left[ \operatorname{SL}_2(\widehat{\Z}) : S \right] = \Ind_S(m).
		\end{align}
		Case (B) follows from the definition of $\Ent(m_p,p)$ and $\Ind_S(p)$, we then focus on case (A).
		By Theorem \ref{thm:lemosbenfatto} we know that for every prime $p$, if the representation $\rho_{E,p}$ is not surjective, its image is contained either in the normaliser of a non-split Cartan or in an exceptional subgroup.
		If $\operatorname{Im}\rho_{E,p}$ is contained in an exceptional subgroup, we know by Lemma \ref{lemma:exceptional} that either $p=5$ and $\operatorname{Im}\rho_{E,5^\infty}$ has index $5$ or the index of $\operatorname{Im}\rho_E$ is $182$, and hence satisfies the inequality in the statement of the lemma.
		We will then assume that for every prime $p$ for which $\rho_{E,p}$ is not surjective, either $\operatorname{Im}\rho_{E,p} \subseteq C_{ns}^+(p)$ or $p=5$ and $[\GL_2(\Z_5) : \operatorname{Im}\rho_{E,5^\infty}] = 5$.
		Define the set $\mathcal{P} := \{2,3,5\} \cup \mathcal{C}_{ns}$ and consider $m:=\prod_{p \in \mathcal{P}} p$.
		Let $\Delta$ be the minimal discriminant of $E$, and define the element $u := \prod_{p \in \mathcal{C}_{ns}} p^{\frac{v_p(\Delta)}{12}}$. Consider the number field $K := \Q(u)$ of degree dividing $12$. Since by Corollary \ref{cor:goodreductionforcartanQ} the curve $E$ has potentially good reduction at $p$ for every prime $p \in \mathcal{C}_{ns}$, we know that $E$ acquires good reduction over $K$ at each of these primes (see for example \cite[Proposition 1]{kraus90}). Set $e_p := [K\Q_p^{nr} : \Q_p^{nr}] \in \{1,2,3,4,6\}$ the semi-stability defect of $E$ at $p$: this is the denominator of $\frac{v_p(\Delta)}{12}$ (notice that $e_p = 1$ if and only if $E$ has good reduction at $p$ over $\Q$).
		We can write
		\begin{align*}
			\Ind_S(m) &= \frac{[\SL_2(\Z_m) : \rho_{E,m^\infty}(\Gal(\overline{\Q}/K\Qab))]}{[K\Qab \cap \Q(E[m^\infty]) : \Qab \cap \Q(E[m^\infty])]}.
		\end{align*}
		We will call $$\Ind_K(m) := [\SL_2(\Z_m) : \rho_{E,m^\infty}(\Gal(\overline{\Q}/K\Qab))].$$
		Let $p$ be the largest prime in $\mathcal{C}_{ns}$ and set $B:=\mathcal{P} \setminus \{p\}$. By Corollary \ref{cor:goursatgaloisnf} and Lemma \ref{lemma:intersezionestupida} we have
		\begin{align}\label{eq:finproof2}
			\Ind_K(m) &= \Ent_{K\Qab}((m/p)^\infty, p^\infty) \cdot \Ind_K(m/p) \cdot \Ind_K(p) \nonumber \\
			&= \Ent_{K\Qab}((m/p)^\infty, p) \cdot \Ind_K(m/p) \cdot \Ind_K(p) \nonumber \\
			&= \frac{[K\Qab(E[p]) : K\Qab]}{\deg_{K\Qab}(p : (m/p)^\infty)} \cdot \Ind_K(m/p) \cdot \Ind_K(p).
		\end{align}
		On the other hand, 
		since $p$ is coprime with the degree of $K$, by Proposition \ref{prop:ellcurvecartanlift} and Lemma \ref{lemma: p-Sylow of S_n}(4) we have
		\begin{align*}
			\Ind_K(p) &= [\SL_2(\Z_p) : G(p^\infty) \cap \SL_2(\Z_p)] \cdot [G(p^\infty) \cap \SL_2(\Z_p) : \rho_{E,p^\infty}(\Gal(\overline{\Q}/K\Qab))] \\
			&= \Ind(p) \cdot [G(p^\infty) \cap \SL_2(\Z_p) : S(p^\infty)] \cdot [S(p^\infty) : \rho_{E,p^\infty}(\Gal(\overline{\Q}/K\Qab))] \\
			&= \Ind(p) \cdot [G(p) \cap \SL_2(\F_p) : S(p)] \cdot [S(p) : \rho_{E,p}(\Gal(\overline{\Q}/K\Qab))] \\
			&= \Ind(p) \cdot [G(p) \cap \SL_2(\F_p) : \rho_{E,p}(\Gal(\overline{\Q}/K\Qab))] \\
			&= \Ind(p) \cdot [\Q(E[p]) \cap K\Qab : \Q(\zeta_{p})] = \Ind(p) \cdot \frac{[\Q(E[p]) : \Q(\zeta_{p})]}{[K\Qab(E[p]) : K\Qab]}.
		\end{align*}
		By Theorem \ref{thm:index3incartan}, we have $[\Q(E[p]) : \Q(\zeta_{p})] = 2(p+1)$, and putting the formula for $\Ind_K(p)$ in equation \eqref{eq:finproof2} we obtain
		\begin{align}\label{eq:finproof3}
			\Ind_K(m) &= \Ind_K(m/p) \cdot \Ind(p) \cdot \frac{2(p+1)}{\deg_{K\Qab}(p : (m/p)^\infty)}.
		\end{align}
		Suppose first that $p>7$. By Corollary \ref{cor:supersingularQ} we know that either $p=13$ or $E$ has good supersingular reduction at $p$ over $K$; however, by \cite[Corollary 1.3]{balakrishnan19} we know that for $p = 13$ the image of $\rho_{E,p}$ is not contained in $C_{ns}^+(p)$. Setting $\eta := \frac{e_p}{\gcd(e_p, 2)} \le 3$, we can then apply Theorem \ref{thm:entanginnf} to show that $\deg_{K\Qab}(p : (m/p)^\infty)$ is a multiple of $\frac{p+1}{\gcd(\eta, p+1)}$, and hence it is at least $\frac{p+1}{3}$. We then obtain that
		$$\Ind_K(m) \le 6 \cdot \Ind_K(m/p) \cdot \Ind(p).$$
		Similarly, if we further assume that $E$ has good reduction at $p$ over $\Q$, and hence that $e_p = 1$, we have
		\begin{equation*}
			\Ind_K(m) \le 2 \cdot \Ind_K(m/p) \cdot \Ind(p).
		\end{equation*}
		If instead $p=7$, we have that $\frac{m}{p}=30$, and by equation \eqref{eq:finproof3} we simply obtain
		\begin{equation*}
			\Ind_K(m) \le 16 \cdot \Ind_K(30) \cdot \Ind(7).
		\end{equation*}
		We can now iterate this argument on $\frac{m}{p}$ in place of $m$, so that we obtain
		\begin{align*}
			\Ind_K(m) \le \Delta_7 \cdot 2^{|\mathcal{C}_{ns}|-\beta} \cdot 6^\beta \cdot \Ind_K(30) \cdot \prod_{p \in \mathcal{C}_{ns}} \Ind(p).
		\end{align*}
		To conclude, it suffices to notice that
		$$\Ind_K(30) = [K\Qab \cap \Q(E[30^\infty]) : \Qab \cap \Q(E[30^\infty])] \cdot \Ind_S(30),$$
		and that we have the inequality
		\begin{equation*}
			[K\Qab \cap \Q(E[30^\infty]) : \Qab \cap \Q(E[30^\infty])] \le [K\Qab \cap \Q(E[m^\infty]) : \Qab \cap \Q(E[m^\infty])]. \qedhere
		\end{equation*}
	\end{proof}
	
	\begin{lemma}\label{lemma:boundnonsuq}
		Let $\faktor{E}{\Q}$ be an elliptic curve without CM. Define $\mathcal{C}$ as the set of all odd primes $p$ for which $\operatorname{Im}\rho_{E,p} \subseteq C_{ns}^+(p)$. Let $\beta$ be the number of primes $p>5$ in $\mathcal{C}$ for which $E$ has bad reduction at $p$.
		For every $p \in \mathcal{C}$, call $n_p$ the largest integer $n$ for which $\operatorname{Im}\rho_{E,p^n} \subseteq C_{ns}^+(p^n)$, and define $\Lambda := \prod_{p \in \mathcal{C}} p^{n_p}$. Suppose that $\mathcal{C}$ contains a prime greater than $13$ (Case (A) of Proposition \ref{prop:dichotomy}). We have
		\begin{equation*}
			[\GL_2(\widehat{\Z}) : \operatorname{Im}\rho_E] \le 2488320 \cdot \Delta_7 \cdot 3^{\beta} \cdot \Lambda^3,
		\end{equation*}
		where 
		\begin{align*}
			\Delta_7:=\begin{cases} 1 &\text{ if } 7 \notin \mathcal{C}, \\ 8 &\text{ if } 7 \in \mathcal{C} \text{ and } E \text{ has good reduction at } 7, \\ \frac{8}{3} &\text{ if } 7 \in \mathcal{C} \text{ and } E \text{ has bad reduction at } 7. \end{cases}
		\end{align*}
	\end{lemma}
	
	\begin{proof}
		Set $\mathcal{C}_{ns} = \mathcal{C} \setminus \{3,5\}$. As we are in case (A) of Proposition \ref{prop:productofindices}, we have
		\begin{align*}
			[\GL_2(\widehat{\Z}) : \operatorname{Im}\rho_E] \le \Delta_7 \cdot 2^{|\mathcal{C}_{ns}|} \cdot 3^\beta \cdot \Ind_S(30) \cdot \prod_{p \in \mathcal{C}_{ns}} \Ind(p).
		\end{align*}
		The proof consists in bounding the indices $\Ind(p)$ using Corollary \ref{cor:padicindices} and the factor $\Ind_S(30)$ using entanglement properties shown in Section \ref{subsec:entanglement}. The latter part will be divided in two steps: we first separate the contribution to $\Ind_S(30)$ given by the prime $5$ from $\Ind_S(6)$, and then we separate the contributions to $\Ind_S(6)$ given by the primes $2$ and $3$. \\
		\underline{\textit{Separating the prime $5$.}}
		We notice that 
		$$\Ind_S(30) = \Ind_S(5) \cdot \Ind_S(6) \cdot \Ent(6,5).$$
		Set $S = \rho_{E}(\Gal(\overline{\Q}/\Qab))$ and call $S_m$ the projection of $S$ in $\SL_2(\Z_m)$. By Lemma \ref{lemma:exceptional} and Theorem \ref{thm:lemosbenfatto} we know that $\operatorname{Im}\rho_{E,5}$ is $\GL_2(\F_5)$, or it is conjugate to either a subgroup of $C_{ns}^+(5)$, or to the exceptional subgroup $5S4$. If $\operatorname{Im}\rho_{E,5} = \GL_2(\F_5)$ we know by Lemma \ref{lemma: p-Sylow of S_n}(5) that $S_5 = \SL_2(\Z_5)$.
		We can apply Goursat's lemma to show that the image of $S_{30}$ in the product $\frac{S_{6}}{N_{6}} \times \frac{\SL_2(\Z_5)}{N_5}$ corresponds to the graph of an isomorphism $\frac{S_{6}}{N_6} \cong \frac{\SL_2(\Z_5)}{N_5}$, where $N_6$ and $N_5$ are the kernels of the projections on $S_5$ and $S_6$ respectively. However, the group $S_{6}$ is solvable, while following the description of $\operatorname{Occ}(\GL_2(\Z_p))$ in \cite[IV-25]{serre-abrep} we see that $\SL_2(\Z_5)$ contains $\operatorname{PSL}_2(\F_5)$ in its composition series. This implies that $N_5$ must surject onto $\operatorname{PSL}_2(\F_5)$. In particular, by \cite[IV \S3.4 Lemmas 2 and 3]{serre-abrep} this implies that $N_5 = \SL_2(\Z_5)$, and so
		$$\Ind_S(30) = \Ind_S(6) \cdot \Ind_S(5) = \Ind_S(6).$$
		If instead $\rho_{E,5}$ is not surjective, consider the field $K = \Q(E[5])$ and notice that $\Ent_{K\Qab}(6^\infty, 5^\infty) = 1$, because $K\Qab(E[5^\infty])/K\Qab$ is a pro-$5$ extension, while $K\Qab(E[6^\infty])/K\Qab$ is a pro-$6$ extension. As in the proof of Proposition \ref{prop:productofindices}, set $\Ind_K(m) := [\SL_2(\Z_m) : \rho_{E,m^\infty}(\Gal(\overline{\Q}/K\Qab))]$. Using Corollary \ref{cor:goursatgaloisnf}, we can write
		\begin{align*}
			\Ind_S(30) = \frac{\Ind_K(30)}{[K\Qab : \Qab]} = \frac{\Ind_K(6) \cdot \Ind_K(5)}{[K\Qab : \Qab]} \le \Ind_S(6) \cdot \Ind_K(5).
		\end{align*}
		However, by Proposition \ref{prop:3e5speciali} and Lemma \ref{lemma: p-Sylow of S_n}(4,6), we know that $\rho_{E,5^\infty}(\Gal(\overline{\Q}/K\Qab))$ contains all the matrices in $\operatorname{Im}\rho_{E,5^\infty} \cap \SL_2(\Z_5)$ congruent to the identity modulo $5$. Hence we have $\Ind_K(5) \le [K\Qab : \Qab] \cdot \Ind(5) \le 24\Ind(5)$, where the last inequality holds because $\operatorname{Im}\rho_{E,5}$ is contained either in $C_{ns}^+(5)$ or in $5S4$.
		Combining all cases we have
		\begin{align*}
			[\GL_2(\widehat{\Z}) : \operatorname{Im}\rho_E] \le 24\Delta_7 \cdot 2^{|\mathcal{C}_{ns}|} \cdot 3^\beta \cdot \Ind_S(6) \prod_{p \in \mathcal{C}_{ns} \cup \{5\}} \Ind(p).
		\end{align*}
		\underline{\textit{Separating the primes $2$ and $3$.}}
		We now notice that both $S_2$ and $S_3$ are solvable, and that $S_2$ has just one copy of $\faktor{\Z}{3\Z}$ in its composition series, while $S_3$ has 3 copies of $\frac{\Z}{2\Z}$ in its composition series. The quotients $\frac{S_2}{N_2} \cong \frac{S_3}{N_3}$ will then have order less than or equal to $24$. In particular, we have $\Ent(2,3) \mid 24$, and so $\Ind_S(6) \le 24 \cdot \Ind_S(2) \cdot \Ind_S(3)$.
		We then obtain
		\begin{align}\label{eq:finproof4}
			[\GL_2(\widehat{\Z}) : \operatorname{Im}\rho_E] \le 24^2\Delta_7 \cdot 2^{|\mathcal{C}_{ns}|} \cdot 3^\beta \cdot \Ind_S(2) \cdot \Ind_S(3) \cdot \prod_{p \in \mathcal{C}_{ns} \cup \{5\}} [\GL_2(\Z_p) : \operatorname{Im}\rho_{E,p^\infty}].
		\end{align}
		\underline{\textit{Bounding the $p$-adic indices.}}
		By Theorem \ref{thm:lemosbenfatto} we know that $\operatorname{Im}\rho_{E,3}$ is contained neither in a Borel subgroup nor in the normaliser of a split Cartan subgroup. In particular, by \cite[Theorem 1.2]{zywina15} we know that it must be either equal to the normaliser of a non-split Cartan or to $\GL_2(\F_3)$. By Proposition \ref{prop:3e5speciali} and Lemma \ref{lemma: p-Sylow of S_n}(4) we know that if $3 \notin \mathcal{C}$, then $\Ind_S(3) = \operatorname{Ind}(3) \le 27$. On the other hand, if $3 \in \mathcal{C}$ then by Corollary \ref{cor:padicindices} and Lemma \ref{lemma: p-Sylow of S_n}(3) we have $\Ind_S(3) \le 2 \cdot \operatorname{Ind}(3) \le 2 \cdot 3^{3n_3-1}$. In all cases, we can write $\operatorname{Ind}_S(3) \le \max\{27, 2\cdot 3^{3n_3-1}\} \le 27 \cdot {3^{3n_3}}$.
		Similarly, if $\rho_{E,5}$ is not surjective and $5 \notin \mathcal{C}$, we can apply Proposition \ref{prop:3e5speciali} to obtain that $\operatorname{Ind}(5) = 5$. If instead $5 \in \mathcal{C}$, by Corollary \ref{cor:padicindices} we have $\operatorname{Ind}(5) \le \frac{2}{5} \cdot 5^{3n_5}$. Since if $\rho_{E,5}$ is surjective we have $\operatorname{Ind}(5) =1$, in all cases we can write $\operatorname{Ind}(5) \le 5 \cdot {5^{3n_5}}$.
		By Proposition \ref{prop:2adicimages}, we know that either $[\GL_2(\widehat{\Z}) : \operatorname{Im}\rho_E] = 128$ or $\operatorname{Ind}(2)$ divides $32$ and the modular curve corresponding to $\operatorname{Im}\rho_{E,2^\infty}$ has infinitely many rational points. We can exclude the first case, as it is better than the inequality in the statement of the lemma. We can then use \cite[Theorem 1.1]{sutherlandzywina17} and the group table in the file \verb*|GaloisImages.m| attached to the paper \cite{sutherlandzywina17} to explicitly determine all the possible values of $\operatorname{Ind}_S(2)$ for the remaining 2-adic images, namely those for which the corresponding modular curve has infinitely many rational points. To do that, let the level of the $2$-adic image be $2^n$: by Lemma \ref{lemma: p-Sylow of S_n}(1) it suffices to check $\Ind_S(2)$ modulo $2^{2n+1}$. We then compute $S_2$ as the derived group of $\operatorname{Im}\rho_{E,2^{2n+1}}$. In particular, we compute that $\Ind_S(2) \le 32$. The MAGMA code for the proof can be found in the ancillary files of the arXiv version of this paper \cite{furio24}.
		For all the other primes $p \in \mathcal{C}$, by Corollary \ref{cor:padicindices} we have $\operatorname{Ind}(p) \le \frac{p^{3n_p}}{2}$.
		Replacing all these bounds in equation \eqref{eq:finproof4} we obtain
		\begin{align*}
			[\GL_2(\widehat{\Z}) : \operatorname{Im}\rho_E] &\le 24^2 \cdot 32 \cdot 27 \cdot 5 \cdot \Delta_7 \cdot 2^{|\mathcal{C}_{ns}|} \cdot 3^\beta \cdot 3^{3n_3} \cdot 5^{3n_5} \prod_{p \in \mathcal{C}_{ns}} \frac{p^{3n_p}}{2} \nonumber \\
			&= 2488320 \cdot \Delta_7 \cdot 3^{\beta} \cdot \Lambda^3. \qedhere
		\end{align*}
	\end{proof}
	
	\begin{remark}\label{rmk: adelic vs p-adic}
		Notice that, in the setting of Lemma \ref{lemma:boundnonsuq},
		$$\Delta_7 \cdot 2^{|\mathcal{C}_{ns}|} \cdot 3^{\beta} \le \max\left\{1 \cdot 6^{|\mathcal{C}_{ns}|}, 8 \cdot 2 \cdot 6^{|\mathcal{C}_{ns}|-1}, \frac{8}{3} \cdot 6^{|\mathcal{C}_{ns}|} \right\} = \frac{8}{3} \cdot 6^{|\mathcal{C}_{ns}|},$$
		and therefore by equation \eqref{eq:finproof4} we have
		$$[\GL_2(\widehat{\Z}) : \operatorname{Im}\rho_E] \le 73728 \cdot 6^{|\mathcal{C}_{ns}|} \cdot \prod_{p \in \mathcal{P}} [\GL_2(\Z_p) : \operatorname{Im}\rho_{E,p^\infty}],$$
		where $\mathcal{C}_{ns} = \{p \text{ prime} \mid p \ge 7, \ \operatorname{Im}\rho_{E,p} \subseteq C_{ns}^+(p)\}$ and $\mathcal{P} = \mathcal{C}_{ns} \cup \{2,3,5\}$.
	\end{remark}
	
	\begin{lemma}\label{lemma:boundsuq}
		Let $\faktor{E}{\Q}$ be an elliptic curve without CM. 
		Define $\mathcal{C}$ as the set of all odd primes $p$ for which $\operatorname{Im}\rho_{E,p} \subseteq C_{ns}^+(p)$. For every $p \in \mathcal{C}$, call $n_p$ the largest integer $n$ for which $\operatorname{Im}\rho_{E,p^n} \subseteq C_{ns}^+(p^n)$, and define $\Lambda := \prod_{p \in \mathcal{C}} p^{n_p}$. Suppose that $j(E)$ does not belong to the list \eqref{eq:jlist} and that $\rho_{E,p}$ is surjective for every prime $p>13$ (Case (B) of Proposition \ref{prop:dichotomy}). We have
		\begin{equation*}
			[\GL_2(\widehat{\Z}) : \operatorname{Im}\rho_E] \le 1.9 \cdot 10^{14} \cdot \Lambda^2.
		\end{equation*}
	\end{lemma}
	
	\begin{proof}
		By Theorem \ref{thm:mazurisogeny}, we know that if $p$ is an odd prime for which $E$ has a rational $p$-isogeny, then $p \in \{3,5,7,13\}$. Define the set $\mathcal{P} = \{2,3,5\} \cup \{p \mid \rho_{E,p} \text{ is not surjective}\} \subseteq \{2,3,5,7,11,13\}$, and set as before $m:=\prod_{p \in \mathcal{P}} p$ and $S:= \rho_E\left( \Gal\left(\overline{\Q}/\Qab\right) \right)$. Moreover, we will call $S_p := \rho_{E,p^\infty}(\Gal(\overline{\Q}/\Qab))$, which is the projection of $S$ on the $p$-adic component. Define the set $B_p:=\{q \in \mathcal{P} : q<p\}$ and write $m_p:=\prod_{q \in B_p} q$ (where $m_2=1$). By Proposition \ref{prop:productofindices} we have
		\begin{align}\label{eq:finproof6.1}
			[\GL_2(\widehat{\Z}) : \operatorname{Im}\rho_E] = \prod_{p \in \mathcal{P}} \Ind_S(p) \cdot \Ent(m_p, p).
		\end{align}
		If $p \ge 5$ we can apply Lemma \ref{lemma:intersezionestupida} and obtain
		$$K_p := \Qab(E[p^\infty]) \cap \Qab(E[m_p^\infty]) = \Qab(E[p]) \cap \Qab(E[m_p^\infty]).$$
		Moreover, similarly to the proof of Lemma \ref{lemma:intersezionestupida}, since the Galois group $\Gal\left(\faktor{\Qab(E[m_p^\infty])}{\Qab}\right)$ does not contain any finite group of order divisible by $p$ in its composition series, the field $K_p$ must be a subextension of $\faktor{\Qab(E[p])}{\Qab}$ of degree coprime with $p$. In particular, if $P_p$ is a $p$-Sylow of $S_p$, we have that $[K_p : \Qab] \le [S_p : P_p]$, and so
		\begin{equation}\label{eq:finproof6.2}
			\Ind_S(p) \cdot \Ent(m_p,p) = [\SL_2(\Z_p) : S_p] \cdot [K_p : \Qab] \le [\SL_2(\Z_p) : P_p].
		\end{equation}
		We now proceed by providing a bound on the indices of the groups $P_p$ prime by prime, assuming that $\rho_{E,p^\infty}$ is not surjective. To optimise the result, we will bound the degree $[K_3 : \Qab]$ in part together with the index $[\SL_2(\Z_3) : S_3]$ and in part together with $[\SL_2(\Z_2) : S_2]$.
		\begin{itemize}
			\item \underline{$p=13$}. We can apply Proposition \ref{prop:groupingfacts} to show that $\operatorname{Im}\rho_{E,13}$ is contained in a Borel subgroup. By \cite[Theorem 1]{greenberg12} and \cite[Remark 4.2.1]{greenberg12} we have that the $13$-Sylow of $\GL_2(\Z_{13})$ is contained in $\operatorname{Im}\rho_{E,13^\infty}$, so the $13$-Sylow of $\SL_2(\Z_{13})$ must be contained in $\operatorname{Im}\rho_{E,13^\infty} \cap \SL_2(\Z_{13})$ and hence also in $S_{13}$ by Lemma \ref{lemma: p-Sylow of S_n}(2). This means that it coincides with $P_{13}$. We then obtain $[\SL_2(\Z_{13}) : P_{13}] \le 12 \cdot 14$.
			\item \underline{$p=11$}. By Proposition \ref{prop:groupingfacts}, $\operatorname{Im}\rho_{E,11}$ is equal to the normaliser of a non-split Cartan subgroup. In particular, by Lemma \ref{lemma: p-Sylow of S_n}(3) and Corollary \ref{cor:intersechino} we have $\Ind_S(11) \cdot [K_{11} : \Qab] \le 12\Ind(11)$. As explained in Remark \ref{rmk: conjecture on p-adic images}, we know that $\operatorname{Im}\rho_{E,11^\infty}$ never falls in the second case of Theorem \ref{thm:ellipticcartantower}, and hence in Proposition \ref{prop:padicindices} the index $\Ind(11)$ is never equal to $\frac{p^3-p^2}{2}$. Using Proposition \ref{prop:padicindices}, we can then bound $\Ind(11) \le \frac{5}{11} \cdot 11^{2n_{11}}$, and obtain $[\operatorname{SL}_2(\Z_{11}) : S_{11} ] \cdot [K_{11} : \Qab] \le \frac{60}{11} \cdot 11^{2n_{11}}$.
			\item \underline{$p=7$}. By \cite[Theorem 1.6]{rszb22} we see that there are three possible cases: $\operatorname{Im}\rho_{E,7^\infty} \supseteq I + 7M_2(\Z_7)$ (as we can check in the online supplement of \cite{sutherlandzywina17}), the image of $\rho_{E,7}$ is contained in $C_{ns}^+(7)$, or $E$ corresponds to one of the two exceptional points in \cite[Table 1]{rszb22}. In the last case, we can compute that $[\GL_2(\widehat{\Z}) : \operatorname{Im}\rho_E] = 224$. In the Cartan case, using Proposition \ref{prop:padicindices} and Lemma \ref{lemma: p-Sylow of S_n}(3) we have $[ \operatorname{SL}_2(\Z_{7}) : P_7 ] \le (7^2-1)\cdot 7^{2n_{7}}$. If instead $\operatorname{Im}\rho_{E,7^\infty} \supseteq I + 7M_2(\Z_7)$, by \cite[Theorem 1.5]{zywina15} we know that the image of $\rho_{E,7^\infty}$ over $\Q(\zeta_7)$ either contains a $7$-Sylow of $\SL_2(\Z_7)$ or it is an N-Cartan lift. In both cases, by Lemma \ref{lemma: p-Sylow of S_n}(2,3) we know that $S_7$ contains a $7$-Sylow of $\operatorname{Im}\rho_{E,7^\infty}$, and hence we obtain $[\SL_2(\Z_7) : P_7] \le (7^2-1) \cdot 7$. In all cases, we have $[\SL_2(\Z_7) : P_7] \le 7 \cdot 48 \cdot 7^{2n_7}$.
			\item \underline{$p=5$}. Similarly to the case $p=7$, by \cite[Theorem 1.6]{rszb22} we have three cases: $\langle \operatorname{Im}\rho_{E,5^\infty}, -I \rangle$ is one of the groups in \cite[Table 2]{sutherlandzywina17}, or $\operatorname{Im}\rho_{E,25} \subseteq C_{ns}^+(25)$, or $E$ corresponds to one of two exceptional points with $[\GL_2(\widehat{\Z}) : \operatorname{Im}\rho_E] \in \{200,300\}$. In the Cartan case we have $n_5 \ge 2$, so by Proposition \ref{prop:padicindices} and Lemma \ref{lemma: p-Sylow of S_n}(3) we have $[\SL_2(\Z_5) : P_5] \le (5^2-1) \cdot 5^{2n_5 -1}$. In the first case, as we did for $\Ind_S(2)$ in the part `\textit{Bounding the $p$-adic indices}' in the proof of Lemma \ref{lemma:boundnonsuq}, we can compute in MAGMA that $[\SL_2(\Z_5) : P_5] \le 600$ (code at \cite{furio24}), and hence we obtain that in all cases $[\SL_2(\Z_5) : P_5] \le 600 \cdot 5^{2n_5}$.
			\item \underline{$p=3$}. Again, by \cite[Theorem 1.6]{rszb22} we have that either $\langle \operatorname{Im}\rho_{E,3^\infty}, -I \rangle$ is one of the groups in \cite[Table 1]{sutherlandzywina17}, or $\operatorname{Im}\rho_{E,27} \subseteq C_{ns}^+(27)$. In the first case, we compute in MAGMA that $[\SL_2(\Z_3) : S_3] \le 648$ and that its $3$-adic valuation is at most $4$. As shown in the proof of Lemma \ref{lemma:boundnonsuq} (in the part `\textit{Separating the primes $2$ and $3$'}), we know that $[K_3 : \Qab] = \Ent(3,2)$ divides $24$. In particular, since $\SL_2(\Z_3)$ only has 3 copies of $\faktor{\Z}{2\Z}$ in its composition series, we have that $[\SL_2(\Z_3) : S_3] \cdot [K_3 : \Qab] \le 81 \cdot 24$. By Theorem \ref{thm:ellipticcartantower}, Corollary \ref{cor:padicindices}, and Lemma \ref{lemma: p-Sylow of S_n}(3), in both cases we have $[ \operatorname{SL}_2(\Z_{3}) : S_{3} ] \le \max\{648, 2 \cdot 3^{2n_3-1}\} \le 648 \cdot 3^{2n_{3}}$ and $[\SL_2(\Z_3) : S_3] \cdot [K_3 : \Qab] \le 81 \cdot 24 \cdot 3^{2n_{3}}$.
			\item \underline{$p=2$}. Since $m_2=1$ by definition, we have $[K_2 : \Qab] = 1$. Repeating the arguments in the case `$p=3$' and in the proof of Lemma \ref{lemma:boundnonsuq}, we can compute that $\operatorname{Ind}_S(2)$ divides $1536$. To do that, we use again Proposition \ref{prop:2adicimages}, the data of \cite{sutherlandzywina17}, and the MAGMA algorithm attached to \cite{furio24}. Since $\SL_2(\Z_2)$ only contains one copy of $\faktor{\Z}{3\Z}$ in its composition series, we have $\Ind_S(2) \cdot [K_3 : \Qab] \le 512 \cdot 24$. Combining it with the argument for $p=3$, we obtain $\Ind_S(3) \cdot \Ind_S(2) \cdot \Ent(3,2) \le 81 \cdot 512 \cdot 24$.
		\end{itemize}
		Writing $\Lambda=\prod_{p \in \mathcal{P}} p^{n_p}$, combining the bounds above with equations \eqref{eq:finproof6.1} and \eqref{eq:finproof6.2}, we obtain
		\begin{align}\label{eq:finproof7}
			\left[ \operatorname{SL}_2(\Z_m) : S_\mathcal{P} \right] &\le 168 \cdot \frac{60}{11} \cdot 336 \cdot 600 \cdot 81 \cdot 512 \cdot 24 \cdot \Lambda^2 \le 1.9 \cdot 10^{14} \cdot \Lambda^2,
		\end{align}
		concluding the proof.
	\end{proof}
	
	We now give the final part of the proof of Theorem \ref{thm:adelicbound}, treating separately cases (A) and (B) of Proposition \ref{prop:dichotomy}.
	
	\begin{proof}[\textbf{Proof of Theorem \ref{thm:adelicbound}}]
		We notice that if $j(E)$ belongs to the list \eqref{eq:jlist}, by Proposition \ref{prop:groupingfacts} we have $[\GL_2(\widehat{\Z}) : \operatorname{Im}\rho_E] \le 2736$, hence we can assume that $j(E)$ is not in the list.
		\begin{enumerate}
		\item[\textbf{(A)}] Suppose first that case (A) of Proposition \ref{prop:dichotomy} holds.
		Let $\mathcal{C}$ be the set of all odd primes $p$ such that $\operatorname{Im}\rho_{E,p} \subseteq C_{ns}^+(p)$ and let $\mathcal{C}_{ns}=\mathcal{C} \setminus \{3,5\}$. For every $p \in \mathcal{C}$, define $n_p$ as the largest integer $n$ for which $\operatorname{Im}\rho_{E,p^n} \subseteq C_{ns}^+(p^n)$, and let $\Lambda:=\prod_{p \in \mathcal{C}} p^{n_p}$. By Lemma \ref{lemma:boundnonsuq} we know that
		\begin{align}\label{eq:finproof5}
			[\GL_2(\widehat{\Z}) : \operatorname{Im}\rho_E] &\le  2488320 \cdot \Delta_7' \cdot 3^{|\mathcal{C}_{ns}|} \cdot \Lambda^3,
		\end{align}
		where we can assume that $\Delta_7':=1$ if $7 \notin \mathcal{C}$ and $\Delta_7':=\frac{8}{3}$ otherwise: indeed, we have $\Delta_7 \cdot 3^\beta \le \max\{1 \cdot 3^{|C_{ns}|}, \frac{8}{3} \cdot 3^{|C_{ns}|}, 8 \cdot 3^{|C_{ns}|-1}\}$, so we can set $\Delta_7' := \min\{\Delta_7, \frac{8}{3}\}$.
		As in the proof of Theorem \ref{thm:totaleffiso}, we treat separately the cases in which $j(E) \in \Z$ and $j(E) \notin \Z$.
		\begin{itemize}
		\item Suppose first that $j(E) \notin \Z$. We can write $2488320 \cdot \Delta_7' \le 6635520$. By Proposition \ref{prop:groupingfacts} we know that $\mathcal{C}_{ns} \subseteq \{7,11\} \cup \{p \in \mathcal{C}_{ns} \mid p \ge 19\}$, and so, as in the proof of Theorem \ref{thm:totaleffiso} we have
		\begin{align*}
			|\mathcal{C}_{ns}| &\le \max \left\lbrace \log_{19}\Lambda, 1+\log_{19}\frac{\Lambda}{7}, 1+\log_{19}\frac{\Lambda}{11}, 2+\log_{19}\frac{\Lambda}{77} \right\rbrace \\
			&\le \log_{19}\Lambda + 1 -\log_{19}7 < \log_{19} \Lambda + 0.525.
		\end{align*}
		Applying Theorem \ref{thm:multiplecartanboundwithdenominator} we obtain
		\begin{align}\label{eq:finproof6}
			[\GL_2(\widehat{\Z}) : \operatorname{Im}\rho_E] &< 6635520 \cdot 3^{\log_{19}\Lambda + 0.525} \cdot \Lambda^3 = 6635520 \cdot 3^{0.525} \cdot \Lambda^{3+\log_{19}3} \nonumber \\
			&< 6635520 \cdot 3^{0.525} \cdot \left(\frac{12}{\log 2}\right)^{3+\log_{19}3} \cdot \left(\Fheight(E) + 1.5 \right)^{3+\log_{19}3} \nonumber \\
			&< 1.78 \cdot 10^{11} \cdot (\Fheight(E)+1.5)^{3.38}.
		\end{align}
		Notice that this inequality is better than the first statement of the theorem. We now prove the second part.
		Write $|\mathcal{C}_{ns}| \le |\mathcal{C}| = \omega(\Lambda)$, which is the function counting the distinct prime divisors of $\Lambda$.
		We can assume that $\Lambda \ge 26$, otherwise we would get a stronger statement, hence by \cite[Th\'eor\`eme 13]{robin83} and Theorem \ref{thm:totaleffiso} we have
		\begin{align*}
			\omega(\Lambda) &< \frac{\log\Lambda}{\log\log\Lambda - 1.1714} < \frac{1.308\log(\Fheight(E)+40) + \log21000}{\log(1.308\log(\Fheight(E)+40) + \log21000) - 1.1714} \\
			&< (1.308\log(\Fheight(E)+40) + \log21000) \delta(\Fheight(E)).
		\end{align*}
		Using this bound in equation \eqref{eq:finproof5} and applying again Theorem \ref{thm:multiplecartanboundwithdenominator} we obtain
		\begin{align*}
			[\GL_2(\widehat{\Z}) : \operatorname{Im}\rho_E] &< 6635520 \cdot 3^{\log21000 \cdot \delta(\Fheight(E))} (\Fheight(E) + 40)^{1.308 \cdot \log 3 \cdot \delta(\Fheight(E))} \cdot \Lambda^3 \\
			&< 5 \cdot 10^{13} (\Fheight(E) +40)^{1.437 \cdot \delta(\Fheight(E))} (\Fheight(E) + 1.5)^3,
		\end{align*}
		where we used that $3^{\log21000 \cdot \delta(\Fheight(E))} < 1340$ for $\Fheight(E) > -0.75$ (which can be assumed by Remark \ref{minimalheight}). This inequality is better than the statement of the theorem.
		\item Suppose now that $j(E) \in \Z$. By Lemma \ref{lemma:cartan15e7e9}(1) we know that either $j(E)$ belongs to the list \eqref{eq:intjcartan7} and $[\GL_2(\widehat{\Z}) : \operatorname{Im}\rho_E] \le 504$, or $7 \notin \mathcal{C}_{ns}$. In the former case, the theorem trivially holds, hence we can assume that $7 \notin \mathcal{C}_{ns}$. By Lemma \ref{lemma:cartan15e7e9}(3) we also have that $11 \notin \mathcal{C}_{ns}$. Using again Proposition \ref{prop:groupingfacts}, we obtain that $|\mathcal{C}_{ns}| \le \log_{19}\Lambda$ and $\Delta_7' = 1$, hence applying Theorem \ref{thm:totaleffiso} to equation \eqref{eq:finproof5} we have
		\begin{align*}
			[\GL_2(\widehat{\Z}) : \operatorname{Im}\rho_E] &< 2488320 \cdot 3^{\log_{19}\Lambda} \cdot \Lambda^3 = 2488320 \cdot \Lambda^{3+\log_{19}3} \\
			&< 2488320 \cdot \left(21000\right)^{3+\log_{19}3} \cdot \left(\Fheight(E) + 40 \right)^{1.308 \cdot (3+\log_{19}3)} \nonumber \\
			&< 9.5 \cdot 10^{20} \cdot (\Fheight(E)+40)^{4.42}.
		\end{align*}
		This proves the first part of the theorem.
		Assume now that $\Fheight(E) > 4 \cdot 10^{15}$. We have $$2488320 \cdot 3^{\log 21000 \cdot \delta(\Fheight(E))} < 1.13 \cdot 10^8.$$ As before, we can apply Theorem \ref{thm:totaleffiso} in equation \eqref{eq:finproof5} to obtain
		\begin{align*}
			[\GL_2(\widehat{\Z}) : \operatorname{Im}\rho_E] &< 2488320 \cdot 3^{\log21000 \cdot \delta(\Fheight(E))} (\Fheight(E) +40)^{1.308 \cdot \log 3 \cdot \delta(\Fheight(E))} \Lambda^3 \\
			&< 3.38 \cdot 10^{20} (\Fheight(E) +40)^{4.158 \cdot \delta(\Fheight(E))} (\Fheight(E) + 22.5)^3.
		\end{align*}
		To conclude, it suffices to notice that for $\Fheight(E) \ge 4 \cdot 10^{15}$ we have
		\begin{align*}
			\left(\frac{x+40}{x+22.5}\right)^{4.158 \cdot \delta(\Fheight(E))} < 1 + 10^{-5}
		\end{align*}
		and for $\Fheight(E) \le 4 \cdot 10^{15}$ we have
		\begin{align*}
			9.5 \cdot 10^{20} (\Fheight(E) + 40)^{4.42} < 3.4 \cdot 10^{20} (\Fheight(E) + 22.5)^{3 + 4.158 \cdot \delta(\Fheight(E))}.
		\end{align*}
		\end{itemize}
		\item[\textbf{(B)}] Assume now that we are in case (B) of Proposition \ref{prop:dichotomy}. By Lemma \ref{lemma:boundsuq} and Theorem \ref{thm:totaleffiso} we have
		\begin{align*}
			[ \GL_2(\widehat{\Z}) : \operatorname{Im}\rho_E ] &< 1.9 \cdot 10^{14} \cdot 21000^2 \cdot (\Fheight(E)+40)^{2.616} < 8.4 \cdot 10^{22} \cdot (\Fheight(E)+40)^{2.616}, \nonumber
		\end{align*}
		which is better than the first statement of the theorem for $\Fheight(E) > -0.75$. Similarly, we have
		\begin{align*}
			[ \GL_2(\widehat{\Z}) : \operatorname{Im}\rho_E ] &\le 1.9 \cdot 10^{14} \cdot \Lambda^2 
			< 4 \cdot 10^{22} \cdot (\Fheight(E)+40)^{1.814 \cdot \delta(\Fheight(E))} (\Fheight(E)+22.5)^2,
		\end{align*}
		which is again better than the second statement of the theorem for $\Fheight(E) > -0.75$. \qedhere
		\end{enumerate}
	\end{proof}

	\subsection{A bound in terms of the conductor}\label{subsec:conductor}

	We conclude Section \ref{sec:adelicbound} by giving another bound on the index of the adelic representation $\rho_E$. This new bound is given in terms of the conductor and not in terms of the height as before. In particular, we prove an effective and improved version of \cite[Theorem 1.1(ii)]{zywina11}.
	
	\begin{theorem}\label{thm:conductoradelicbound}
		Let $\faktor{E}{\Q}$ be an elliptic curve without CM. Let $N$ be the product of the primes of bad reduction of $E$ and let $\omega(N)$ be the number of prime factors of $N$. We have
		\begin{equation*}
			[\GL_2(\widehat{\Z}) : \operatorname{Im}\rho_E] < 2488320 \left(51 N(1+\log\log N)^\frac{1}{2}\right)^{3 \omega(N)}.
		\end{equation*}
	\end{theorem}
	
	To prove this result, we improve Proposition 3.3 of the article of Zywina \cite{zywina11} applying the sharpened version of a lemma of Kraus \cite{kraus95} obtained in the preliminaries. The proof is very similar to that of Zywina, however, we have to slightly modify his argument to make it work for the prime $p=3$.
	
	Let $p$ be an odd prime such that $\operatorname{Im}\rho_{E,p} \subseteq C_{ns}^+(p)$ and consider the quadratic character $\varepsilon_p$ defined as
	\begin{equation*}
		\varepsilon_p : \Gal\left(\faktor{\overline{\Q}}{\Q}\right) \overset{\rho_{E,p}}{\longrightarrow} C_{ns}^+(p) \longrightarrow \frac{C_{ns}^+(p)}{C_{ns}(p)} \cong \{\pm 1\}.
	\end{equation*}
	We can identify $\varepsilon_p$ with a Dirichlet character of the absolute Galois group of $\Q$. If $p>3$, Serre showed that the character $\varepsilon_p$ is unramified at all primes $\ell$ that do not divide $N$ (see \cite[Section 5.8, ($c_2$)]{serre72}). If instead $p=3$, by the N\'eron--Ogg-Shafarevich criterion $\varepsilon_p$ is unramified at all primes $\ell$ such that $\ell \nmid 3N$. We will show that, for our purpose, we can assume that $\varepsilon_3$ is unramified at $3$ whenever $3 \nmid N$.
	
	\begin{lemma}\label{lemma:unramifiedcharacter}
		Let $\faktor{E}{\Q}$ be an elliptic curve without CM and let $p$ be an odd prime such that $\operatorname{Im}\rho_{E,p} \subseteq C_{ns}^+(p)$. The character $\varepsilon_p$, defined as above, is unramified at all primes $\ell \nmid pN$. Moreover, we have the following.
		\begin{itemize}
			\item If $p>3$ and $p \nmid N$, the character $\varepsilon_p$ is unramified at $p$.
			\item If $3 \nmid N$ and $\operatorname{Im}\rho_{E,9} \subseteq C_{ns}^+(9)$, the character $\varepsilon_3$ is unramified at $3$.
		\end{itemize}
	\end{lemma}
	
	\begin{proof}
		We follow the proof of Serre for $p>3$ and we show that in our case the argument also works for $p=3$. The fact that $\varepsilon_p$ is unramified at $\ell \nmid pN$ follows from the N\'eron--Ogg--Shafarevich criterion. Since $p \nmid N$, the curve $E$ has good reduction at $p$. By \cite[Section 1.11, Propositions 11 and 12]{serre72} we know that the image $I:=\rho_{E,p}(I_p)$ of the inertia subgroup $I_p$ at $p$ is either a group of the form $\begin{pmatrix} \ast & 0 \\ 0 & 1 \end{pmatrix}$ or a group of order $p^2-1$, depending on whether the curve $E$ has ordinary or supersingular reduction respectively. In the latter case, the group $I$ is contained in $C_{ns}(p)$, because every element in $C_{ns}^+(p) \setminus C_{ns}(p)$ has order dividing $2(p-1)$ and $p^2-1 > 2(p-1)$ (see also \cite[Section 2.2, Proposition 14]{serre72}). If instead $E$ has ordinary reduction at $p$, for $p>3$ there exists an element in $I$ with eigenvalues $\lambda_1, \lambda_2 \in \F_p$ such that $\lambda_1 \ne \pm \lambda_2$. However, every element in $C_{ns}^+(p)$ has eigenvalues conjugate over $\F_{p^2}$ up to sign, and hence this case never occurs (see also again \cite[Section 2.2, Proposition 14]{serre72}). On the other hand, if $p=3$ and $\operatorname{Im}\rho_{E,9} \subseteq C_{ns}^+(9)$, by Lemma \ref{lemma:ordinarycartan} the curve $E$ cannot have ordinary reduction at $3$.
	\end{proof}
	
	If $\ell \nmid N$, we can consider the reduction $\widetilde{E}$ of $E$ modulo $\ell$. As usual, we define the number $a_\ell(E) := \ell + 1 - |\widetilde{E}(\F_\ell)|$.
	
	\begin{lemma}\label{lemma: p^n divides a_ell}
		Let $E$ be a non-CM elliptic curve defined over $\Q$ and let $p^n \ne 3$ be an odd prime power such that $\operatorname{Im}\rho_{E,p^n} \subseteq C_{ns}^+(p^n)$. Let $N$ be the product of the primes for which $E$ has bad reduction and let $\varepsilon_p$ be defined as above. If $\ell \nmid N$ is a prime for which $\varepsilon_p(\ell) = -1$, then $a_\ell(E) \equiv 0 \pmod {p^n}$.
	\end{lemma}
	
	\begin{proof}
		By Lemma \ref{lemma:unramifiedcharacter}, for every $\ell \nmid N$ we have that $\varepsilon_p$ and $\rho_{E,p^n}$ are unramified. The condition $\varepsilon_p(\ell) = -1$ means that $\rho_{E,p^n}(\operatorname{Frob}_\ell) \in C_{ns}^+(p^n) \setminus C_{ns}(p^n)$, and hence it is an element with trace equal to $0$. This implies that $a_\ell(E) \equiv \operatorname{tr} (\rho_{E,p^n}(\operatorname{Frob}_\ell)) \equiv 0 \pmod {p^n}$.
	\end{proof}
	
	\begin{lemma}\label{lemma:smallconductor}
		Let $E$ be a non-CM elliptic curve defined over $\Q$ and let $N$ be the product of the primes for which $E$ has bad reduction. If the index $[\GL_2(\widehat{\Z}) : \operatorname{Im}\rho_E]$ is greater than $2736$, then $N \ge 30$. Moreover, if $N$ is prime we have $N > 700$.
	\end{lemma}
	
	\begin{proof}
		We know that the LMFDB \cite{lmfdb} contains all elliptic curves with conductor up to $500000$, and that all curves in the database have adelic index at most $2736$. In particular, the conductor $N_E$ of $E$ must be greater than $500000$.
		The conductor $N_E$ is at most $2^6 \cdot 3^3 \cdot N^2$, and if $6 \nmid N$ it is at most $\max\{2^6 \cdot N^2, 3^3 \cdot N^2\} = 64N^2$. If $N$ is not divisible by $6$, we then deduce that $N > \sqrt{\frac{500000}{64}} > 80$. If instead $N$ is divisible by $6$, we notice that $N > \sqrt{\frac{500000}{2^6 \cdot 3^3}} > 17$, and that $N$ must be a squarefree integer multiple of $6$. The smallest number satisfying these conditions is $30$.
		If we further assume that $N$ is prime, then $N_E \le N^2$, and hence $N > \sqrt{500000} > 700$.
	\end{proof}
	
	\begin{proposition}\label{prop: bound a prime with N}
		Let $E$ be a non-CM elliptic curve defined over $\Q$.
		Let $N$ be the product of the primes for which $E$ has bad reduction and let $\varepsilon$ be a quadratic Dirichlet character with conductor dividing $N \cdot \operatorname{lcm}(N,2)$. If $N>6$, there exists a prime $\ell \nmid N$ with
		\begin{equation*}
			\ell < 312 \cdot N^2 (1 + \log\log N)
		\end{equation*}
		such that $\varepsilon(\ell) = -1$ and $a_\ell(E) \ne 0$.
	\end{proposition}
	
	\begin{proof}
		Set $E_1:=E$ and consider the elliptic curve $E_2$ obtained by twisting $E_1$ by the character $\varepsilon$. Let $\ell$ be a prime that does not divide $N$. By definition, $E_2$ also has good reduction at $\ell$ and $a_{\ell}(E_2) = \varepsilon(\ell) a_\ell(E_1)$. In particular, we notice that $a_\ell(E_2) \ne a_\ell(E_1)$ if and only if $a_\ell(E) \ne 0$ and $\varepsilon(\ell) = -1$. Hence, it suffices to prove that there exists a small prime $\ell$ such that $a_\ell(E_2) \ne a_\ell(E_1)$.
		First, we notice that there exists a prime $\ell \nmid N$ such that $a_\ell(E) \ne 0$ and $\varepsilon(\ell) = -1$, otherwise $E$ would have complex multiplication by the quadratic field that corresponds to $\varepsilon$.
		Let $N_i$ be the conductor of $E_i$ and define $N_i':=N_i \prod_{q \mid N} q^{d_i(q)}$, where $d_i(q) = 0$, $1$ or $2$ if $E_i$ has additive, multiplicative or good reduction respectively, at $q$. If $M$ is the least common multiple of $N_1'$ and $N_2'$, by \cite[Section 5 C]{deligne85bis} there exists a prime $\ell \le \frac{M}{6} \prod_{q \mid M} \left(1+ \frac{1}{q}\right)$ such that $a_\ell(E_1) \ne a_\ell(E_2)$. This last property is implied by the modularity of $E_1$ and $E_2$, which follows from \cite{modularity}.
		We see that $M$ divides the number $2^6 \cdot 3^3 \cdot N^2$. In particular, since $N>2$,  we can apply Lemma \ref{lemma: refined Kraus lemma} to obtain
		\begin{equation*}
			\ell \le 2^5 \cdot 3^2 \cdot N^2 \prod_{q \mid N} \left(1+ \frac{1}{q}\right) < 312 \cdot N^2 (1+\log\log N). \qedhere
		\end{equation*}
	\end{proof}
	
	\begin{remark}\label{rmk:primeconductor}
		Notice that by the proof of Proposition \ref{prop: bound a prime with N} we can actually deduce that if $N$ is a prime greater than $3$, then 
		\begin{equation*}
			\ell \le \frac{N(N+1)}{6}.
		\end{equation*}
		Indeed, this follows from the fact that $M$ actually divides $N^2$.
	\end{remark}
	
	\begin{proposition}\label{prop: bound for the large primes}
		Let $E$ be a non-CM elliptic curve over $\Q$. Let $N$ be the product of the primes for which $E$ has bad reduction. Let $\omega(N)$ be the number of prime divisors of $N$. Let $M$ be the minimum positive integer such that if $\rho_{E,p^n}(G_{\Q}) \subseteq C_{ns}^+(p^n)$ for an odd prime power $p^n \ne 3$, then $p^n$ divides $M$.
		\begin{itemize}
			\item If $N>6$ we have
			\begin{align*}
				M &< \left(35.33 \cdot N(1+\log\log N)^\frac{1}{2}\right)^{\omega(N)} \quad \text{for every $N$, and} \\
				M &\le \sqrt{\frac{2N(N+1)}{3}} \qquad \text{for $N$ prime.}
			\end{align*}
			\item If $j(E) \notin \Z$ we have $M \le \frac{N^2}{4} - 1$.
		\end{itemize}
	\end{proposition}
	
	\begin{proof}
		Suppose first that $N>6$. Set $N_0:=N$ if $N$ is odd, and $N_0:=2N$ if $N$ is even. Let $V_1$ be the group of quadratic characters of $\left(\faktor{\Z}{N_0\Z}\right)^\times$. We may view $V_1$ as a vector space of dimension $\omega(N)$ over $\F_2$. We define a sequence of primes $\ell_1, \dots, \ell_{\omega(N)}$ relatively prime to $N$ such that $a_{\ell_i}(E) \ne 0$ for every $i$ and for every non-trivial character $\varepsilon \in V_1$ there exists an $i$ for which $\varepsilon(\ell_i) = -1$. We proceed by induction on $i$. Choose a non-trivial character $\alpha_i \in V_i$. By Proposition \ref{prop: bound a prime with N} there exists a prime $\ell_i \nmid N$ smaller than $312 \cdot N^2(1+\log\log N)$ such that $\alpha_i(\ell_i) = -1$ and $a_{\ell_i}(E) \ne 0$. Let $V_{i+1}$ be the subspace of $V_i$ consisting of characters $\varepsilon$ such that $\varepsilon(\ell_{i}) = 1$. The space $V_{i+1}$ has dimension at most $\omega(N) - i$ over $\F_2$. In particular, $V_{\omega(N) +1} = 1$, and so the sequence of primes $\ell_1, \dots, \ell_{\omega(N)}$ has the desired property.
		Define the integer $M':=\prod_{i=1}^{\omega(N)} |a_{\ell_i}(E)|$. If $p^n \ne 3$ is a prime power such that $\operatorname{Im}\rho_{E,p^n} \subseteq C_{ns}^+(p^n)$, there exists $i$ such that $\varepsilon_p(\ell_i) = -1$, and hence by Lemma \ref{lemma: p^n divides a_ell} we have $p^n \mid |a_{\ell_i}(E)|$, that implies $p^n \mid M'$, and in particular $M \le M'$.
		By the Hasse's bound, for every $\ell_i$ we have
		\begin{equation*}
			|a_{\ell_i}(E)| \le 2\sqrt{\ell_i} < 35.33 \cdot N \sqrt{1+\log\log N},
		\end{equation*}
		and hence $M' < \left(35.33 \cdot N (1+\log\log N)^\frac{1}{2}\right)^{\omega(N)}$. Notice that by Remark \ref{rmk:primeconductor}, if $N=\ell$ is prime we have the stronger inequality $M' = |a_{\ell}(E)| \le \sqrt{\frac{2\ell(\ell+1)}{3}}$. \\
		Suppose now that $j(E) \notin \Z$. By Proposition \ref{prop:goodreductionforcartan} we know that if $\ell$ is a prime of potentially multiplicative reduction, for every odd prime power $p^n$ such that $\operatorname{Im}\rho_{E,p^n} \subseteq C_{ns}^+(p^n)$ we have $p^n \mid \ell^2-1$. We then notice that $M \mid \ell^2-1$. If $N$ is composite, we have $\ell \le \frac{N}{2}$, and hence $M \le \frac{N^2}{4} - 1$.
		If $N=\ell$ is prime, by Remark \ref{rmk:primeconductor} we have $M \le \sqrt{\frac{2\ell(\ell+1)}{3}}$, which is smaller than $\frac{\ell^2}{4} - 1$ for $\ell > 4$. On the other hand, we can check on the LMFDB \cite{lmfdb} that there are no non-CM elliptic curves over $\Q$ with $N=3$, and that all non-CM elliptic curves with $N=2$ do not admit any odd prime $p$ for which $\operatorname{Im}\rho_{E,p} \subseteq C_{ns}^+(p)$.
	\end{proof}
	
	We now divide the proof of Theorem \ref{thm:conductoradelicbound} in two cases, according to whether we are in case (A) or (B) of Proposition \ref{prop:dichotomy}.
	
	\begin{proof}[\textbf{Proof of Theorem \ref{thm:conductoradelicbound}}]
		Using Lemma \ref{lemma:smallconductor}, from now on we can then assume that $N \ge 30$, otherwise we would have $[\GL_2(\widehat{\Z}) : \operatorname{Im}\rho_E] \le 2736$, which is better than the statement of the theorem. Moreover, we can also assume that if $N$ is prime then $N>700$.
		Define the set $\mathcal{C} := \{p \ge 3 \mid \operatorname{Im}\rho_{E,p} \subseteq C_{ns}^+(p) \}$. For every $p \in \mathcal{C}$, let $n_p$ be the largest integer $n$ such that $\operatorname{Im}\rho_{E,p^n} \subseteq C_{ns}^+(p^n)$, and define $\Lambda:=\prod_{p \in \mathcal{C}} p^{n_p}$. Set $$\beta:=\left|\{p \in \mathcal{C} \ : \ p>5 \text{ and } E \text{ has bad reduction at } p\}\right| \le \min\{\omega(\Lambda), \omega(N)\}.$$
		We notice that $j(E)$ does not belong to the list \eqref{eq:jlist}, otherwise by Proposition \ref{prop:groupingfacts} we would have $[\GL_2(\widehat{\Z}) : \operatorname{Im}\rho_E] \le 2736$.
		As in the proof of Theorem \ref{thm:adelicbound}, we divide the proof in two cases, following those of Proposition \ref{prop:dichotomy}.
		\begin{enumerate}
		\item[\textbf{(A)}] Suppose first that we are in case (A) of Proposition \ref{prop:dichotomy}, i.e. that $\mathcal{C}$ contains a prime $p>13$. By Lemma \ref{lemma:boundnonsuq} we have
		\begin{align*}
			[\GL_2(\widehat{\Z}) : \operatorname{Im}\rho_E] \le 2488320 \cdot \Delta_7 \cdot 3^\beta \cdot \Lambda^3,
		\end{align*}
		where $\Delta_7 \in \left\{1, \frac{8}{3}, 8\right\}$. We treat separately the cases $j(E) \in \Z$ and $j(E) \notin \Z$.
		\begin{itemize}
		\item If $j(E) \in \Z$, we can assume that $7 \notin \mathcal{C}$: indeed, using Lemma \ref{lemma:cartan15e7e9}(1) we have that either $[\GL_2(\widehat{\Z}) : \operatorname{Im}\rho_E] \le 504$, which is better than the statement of the theorem, or $7 \notin \mathcal{C}$. In particular, we may assume that $\Delta_7=1$. Moreover, in the proof of Lemma \ref{lemma:boundnonsuq}, in the part `\textit{Bounding the $p$-adic indices}', we used the bound $\Ind_S(3) \le \max\{27, 2\cdot 3^{3n_3-1}\} \le 27 \cdot 3^{3n_3}$, hence we can assume that $n_3 \ne 1$ (since we obtain the same bound as with $n_3 = 0$). We can then apply Proposition \ref{prop: bound for the large primes} and obtain
		\begin{align*}
			[\GL_2(\widehat{\Z}) : \operatorname{Im}\rho_E] &\le 2488320 \cdot 3^{\omega(N)} \cdot \left(35.33 \cdot N(1+\log\log N)^\frac{1}{2}\right)^{3 \omega(N)} \\ 
			&= 2488320 \left(35.33 \sqrt[3]{3} \cdot N(1+\log\log N)^\frac{1}{2}\right)^{3 \omega(N)} \\
			&< 2488320 \left(51 N(1+\log\log N)^\frac{1}{2}\right)^{3 \omega(N)}.
		\end{align*}
		\item If $j(E) \notin \Z$, we can bound $\Delta_7 \le 8$ and using Proposition \ref{prop: bound for the large primes} we obtain
		\begin{align*}
			[\GL_2(\widehat{\Z}) : \operatorname{Im}\rho_E] &< 8 \cdot 2488320 \left(\frac{N^2}{4}-1\right)^{3} < 2488320 \left(51 N(1+\log\log N)^\frac{1}{2}\right)^{3 \omega(N)}
		\end{align*}
		for $\omega(N)>1$, and
		\begin{align*}
			[\GL_2(\widehat{\Z}) : \operatorname{Im}\rho_E] &< 8 \cdot 2488320 \left(\frac{2N(N \! + \! 1)}{3}\right)^\frac{3}{2} < 2488320 \left(51 N(1+\log\log N)^\frac{1}{2}\right)^{3 \omega(N)}
		\end{align*}
		for $\omega(N) = 1$.
		\end{itemize}
		\item[\textbf{(B)}] Suppose now that we are in case (B) of Proposition \ref{prop:dichotomy}, i.e. that for every prime $p>13$ the representation $\rho_{E,p}$ is surjective. By Lemma \ref{lemma:boundsuq} we have $[\GL_2(\widehat{\Z}) : \operatorname{Im}\rho_E] < 1.9 \cdot 10^{14} \cdot \Lambda^2$. Moreover, we notice again that in the proof of Lemma \ref{lemma:boundsuq}, in the case `$p=3$', we used the bound $[\SL_2(\Z_3) : S_3] \le \max\{648, 2 \cdot 3^{2n_3-1}\} \le 648 \cdot 3^{2n_3}$, and hence we can assume that $n_3 \ne 1$.
		We treat again separately the cases $j(E) \in \Z$ and $j(E) \notin \Z$.
		\begin{itemize}
		\item If $j(E)$ is not an integer, we can apply Proposition \ref{prop: bound for the large primes} to obtain
		\begin{align*}
			[\GL_2(\widehat{\Z}) : \operatorname{Im}\rho_E] &< 1.9 \cdot 10^{14} \cdot \Lambda^2 < 1.9 \cdot 10^{14} \cdot \left(\frac{N^2}{4}-1\right)^{2} \mkern+10mu \text{for every $N$, and} \\
			[\GL_2(\widehat{\Z}) : \operatorname{Im}\rho_E] &< 1.9 \cdot 10^{14} \cdot \frac{2}{3} N\left(N+1\right) \qquad \text{for $N$ prime.}
		\end{align*}
		One can verify that the first inequality is always better than the statement of the theorem for $\omega(N)>1$, while for $\omega(N) = 1$ we can use the second inequality, which is better than the statement of the theorem for $N>100$ (while we know that in this case $N>700$).
		\item If instead $j(E)$ is an integer, by Lemma \ref{lemma:cartan15e7e9}(3) and \cite[Corollary 1.3]{balakrishnan19} we can assume that $11, 13 \notin \mathcal{C}$.
		In particular, by Proposition \ref{prop:groupingfacts}, we can also assume that $\rho_{E,11}$ is surjective (otherwise its image would be contained in a Borel and we would have $[\GL_2(\widehat{\Z}) : \operatorname{Im}\rho_E] \le 2736$). Hence, if we look at the case `$p=11$' in the proof of Lemma \ref{lemma:boundsuq}, we deduce that we can save a factor $\frac{60}{11}$ in equation \eqref{eq:finproof7}. Similarly, by Lemma \ref{lemma:cartan15e7e9} we can also assume that $\operatorname{Im}\rho_{E,7} \not\subseteq C_{ns}^+(7)$,and looking at the case `$p=7$' in the proof of Lemma \ref{lemma:boundsuq}, a $7$-Sylow of $S_7$ is also a $7$-Sylow of $\SL_2(\Z_7)$, and hence we can save a factor $7$ in equation \eqref{eq:finproof7}.
		We then obtain
		\begin{align*}
			[\GL_2(\widehat{\Z}) : \operatorname{Im}\rho_E] &< \frac{1.9 \cdot 10^{14} \cdot 11}{60 \cdot 7} \cdot \Lambda^2 < \frac{1.9 \cdot 10^{14} \cdot 11}{60 \cdot 7} \left(51 N(1+\log\log N)^\frac{1}{2}\right)^{2 \omega(N)}
		\end{align*}
		for every $N$, and
		\begin{align*}
			[\GL_2(\widehat{\Z}) : \operatorname{Im}\rho_E] &< \frac{1.9 \cdot 10^{14} \cdot 11}{60 \cdot 7} \left(\frac{2N(N+1)}{3}\right)^{\omega(N)} \quad \text{for $N$ prime.}
		\end{align*}
		The first inequality is always better than the statement for $\omega(N) >1$ (using $N \ge 30$), and the second is better as well for $\omega(N) =1$ (using $N>700$). \qedhere
		\end{itemize}
		\end{enumerate}
	\end{proof}

	\bibliographystyle{abbrv}
	\bibliography{bib}
	
\end{document}